\documentclass[a4paper,12pt]{article}
\usepackage[utf8]{inputenc}
\usepackage{fullpage}
\usepackage{amsmath, amssymb, amstext, amsfonts, amsthm, bbm, array, enumerate,scalerel}
\usepackage{mathrsfs}
\usepackage{nicefrac}
\usepackage[none]{hyphenat}
\usepackage{xcolor}
\usepackage{tikz}
\usepackage{stmaryrd}
\usepackage{tikz-3dplot}
\usetikzlibrary{patterns}
\usetikzlibrary{plotmarks}
\usepackage{enumitem}

\newtheorem{theorem}{Theorem}[section]
\newtheorem{lemma}[theorem]{Lemma}
\newtheorem{corollary}[theorem]{Corollary}
\newtheorem{proposition}[theorem]{Proposition}

\theoremstyle{definition}
\newtheorem{definition}[theorem]{Definition}

\theoremstyle{remark}
\newtheorem{remark}[theorem]{Remark}

\numberwithin{equation}{section}
\numberwithin{theorem}{section}

\RequirePackage[colorlinks,citecolor=blue,urlcolor=blue, linkcolor=blue]{hyperref}

\newcommand\mydots{\hbox to 1em{.\hss.\hss.}}

\title{Ramifications of generalized Feller theory}

\author{Christa Cuchiero\thanks{Vienna University, Department of Statistics and Operations Research, Data Science @ Uni Vienna, Kolingasse 14-16 1, A-1090 Wien, Austria, christa.cuchiero@univie.ac.at} \and Tonio M\"ollmann \thanks{ t.moellmann@gmail.com} \and Josef Teichmann \thanks{ETH Zurich,  Department of Mathematics, Rämistrasse 101, CH-8092 Zurich, Switzerland, josef.teichmann@math.ethz.ch\newline
The first author gratefully acknowledges financial support through grant Y 1235 of the START-program.
			The third author gratefully acknowledges financial support by SNF and ETH Foundation.}}
\date{}

\begin{document}
\maketitle

\begin{abstract}
Generalized Feller theory provides an important analog to Feller theory beyond locally compact state spaces. This is very useful for solutions  of certain stochastic partial differential equations, Markovian lifts of fractional processes, or infinite dimensional affine and polynomial processes which appear prominently in the theory of signature stochastic differential equations. We extend several folklore results related to generalized Feller processes, in particular on their construction and path properties, and  provide the often quite sophisticated proofs in full detail.
We also introduce the new concept of extended Feller processes and compare them with standard and generalized ones.
A key example relates generalized Feller semigroups of  algebra homomorphisms via the method of characteristics to transport equations and continuous semiflows on weighted spaces, i.e.~a remarkably generic way to treat differential equations on weighted spaces. We also provide a counterexample, which shows that no condition of the basic definition of generalized Feller semigroups can be dropped.
\end{abstract}



\tableofcontents

\section{Introduction}

Feller processes on locally compact spaces are based on the well-developed analytical theory of strongly continuous semigroups, which simplifies considerably on the Banach space of continuous functions vanishing at infinity which is the function space of interest. There are two major shortcomings of this classical theory: first the analytical frame appears too narrow, as the semigroups act on continuous functions vanishing at infinity, and second local compactness of the underlying state space is often too restrictive. This concerns important applications where the spatial dynamics are modeled via stochastic partial differential equations (see \cite{RoS, DoT}) or signed measure-valued processes, as for instance Markovian lifts of Volterra processes (see \cite{CuT,cuchiero2019markovian, abi2019Markov, abi2019lifting}). Another recent application is in the area of  infinite stochastic covariance modeling  where the processes take values in the cone of positive
Hilbert-Schmidt operators (see \cite{cox2022affine, friesen2022stationary}). Signature stochastic differential equations (SDEs), where affine processes and  machine learning methods meet, constitute another class of promising applications (see  \cite{arribas2020sig,CSvT:23} and also the related paper \cite{friz2022unified}). There one has to deal with state spaces corresponding to subsets of so-called group-like elements of the extended tensor algebra.

To accommodate such infinite dimensional state spaces, so-called  \emph{generalized Feller semigroups} have been introduced in \cite{RoS, DoT}. Relying on these previous works, the goal of this article is to provide a comprehensive and self-contained presentation of the theory of \emph{generalized Feller processes}  and the new  concept of 
\emph{extended Feller processes.}

The theory of generalized Feller processes is built in analogy to classical Feller processes, i.e.~via strongly continuous semigroups acting on certain Banach spaces of functions. With respect to the classical setting, the space of functions vanishing at infinity is replaced by so-called $\mathcal{\mathscr{B}}^{\rho}(E)$-functions.
These are spaces of functions on a weighted space $E$  equipped with some \emph{admissible weight function}
$\rho$ whose growth is controlled by $\rho$ and which lie in the closure of continuous bounded functions with respect to  a weighted supremum norm induced by $\rho$. Therefore -- unlike Feller semigroups -- generalized Feller semigroups act also on unbounded functions, however, other basic properties remain, in particular the  definition is very similar: a
generalized Feller semigroup is a family of positive linear bounded operators from $\mathcal{\mathscr{B}}^{\rho}(E)$ to $\mathcal{\mathscr{B}}^{\rho}(E)$ such that the semigroup properties are satisfied, the norm of the operators remains uniformly bounded for small times and for any map $f$ in $\mathcal{\mathscr{B}}^{\rho}(E)$ the image under the semigroup converges pointwise to $f$ as $t$ approaches $0$.

The important feature of generalized Feller theory is that the underlying state spaces do not have to be locally compact. Already in simple situations this is  crucial: 
 take a Hilbert space $E$ and consider $d$ bounded linear operators $A_1,\ldots,A_d \in L(E)$. Consider furthermore a $d$ dimensional Brownian motion $B$ and the stochastic differential equation
$$
d \lambda_t = \sum_{i=1}^d A_i \lambda_t dB^i_t,
$$
starting at $\lambda_0 \in E$ defining a family of stochastic processes $\lambda$. The state space of the Markov process $\lambda$ is $E$, which is not locally compact. Whence no results of Feller theory are applicable even though -- due to It\^o-calculus -- we know about continuous trajectories, strong uniqueness, boundedness properties of the solutions, linearity, etc. Since bounded linear operators are weakly continuous, too, we can replace the norm topology on $E$ by the weak topology. With the weak topology $E$ is $\sigma$-compact (since balls are weakly compact) and actually a weighted space with weight, e.g., $\rho(\lambda)=1+ {\|\lambda \|}^2$. Therefore $\lambda$ defines a generalized Feller process associated to the generalized Feller semigroup
$$
P(t)f(\lambda_0) : = \mathbb{E}[f(\lambda_t)]
$$
for $t \geq 0$ and $ f \in \mathcal{\mathscr{B}}^{\rho}(E)$. It would be quite non-trivial to construct other strongly continuous semigroups associated to $\lambda$ due to the absence of local compactness or of tractable invariant measures.

In the following we explain the main contributions of this article.
In Section~\ref{notionsgenFeller} we recall important notions and results of generalized Feller theory, in particular the Riesz representation theorem proved in \cite{DoT}, allowing to characterize the dual space of $\mathcal{\mathscr{B}}^{\rho}(E)$ as signed Radon measures whose total variation measure integrates the weight function $\rho$. We prove several lemmas in this section and  provide in Proposition \ref{prop:Stone-Weierstrass on B-rho spaces} a \emph{weighted space version of the Stone-Weierstrass theorem} on $\mathcal{\mathscr{B}}^{\rho}(E)$ in the spirit of Leopoldo Nachbin \cite{nachbin65}, see also \cite[Theorem 3.6]{CST:23} for a version where the elements of the algebra can be unbounded.

This fundamental approximation result is needed in the \emph{existence proof of generalized Feller processes}, given in  Theorem \ref{thm:GFS induce Markov process}. A similar statement was already formulated in \cite[Theorem 2.11]{CuT}, but the full proof with all details is given in the current article. Theorem \ref{thm:GFS induce Markov process} yields stochastic processes whose conditional expectations are given by a strongly continuous semigroup (a generalized Feller one) even in cases when the space $E$ is neither separable nor locally compact. This is a crucial difference to the theory of Feller processes and thus one of the main results of this article.

Let us mention a subtle point in this context, namely that  generalized Feller semigroups act on $\mathscr{B}^{\rho}(E)$ functions which are in general -- as limits of continuous functions -- only Baire-measurable, see Remark \ref{rem:generalized Feller process not Markov} for further details. The proof of Theorem \ref{thm:GFS induce Markov process} relies on a general version of the Kolmogorov extension theorem (see Theorem 15.26 in \cite{AlB}). In order to apply it we construct a projective family of probability measures. Here, the Riesz  representation theorem extended to $\mathcal{B}^{\rho}(E\times E \times \cdots \times E)$ is essential to express continuous linear functionals  on the sub-level sets of the admissible weight function by a (sub-)probability measure. As we let the sub-level sets of the admissible weight function converge to the whole space, such a sequence of (sub-)probability measures converges to a probability measure on  $E\times E \times \cdots E$, yielding a projective family of probability measures with the desired properties. Then the generalized Feller processes is the canonical process on the product space when equipped with a product measure according to the general version of the Kolmogorov extension theorem.

Having constructed generalized Feller processes we treat their path regularity. This is subject of Theorem \ref{thm:GFP have cadlag version}, where it is shown that generalized Feller processes admit a càdlàg version if certain technical conditions are met (thereby closing a gap in a statement in \cite{CuT}).

Starting from a generalized Feller semigroup there is not only one way to construct a related stochastic process. Indeed, in Definition \ref{def:extended Feller process} we  introduce a new class of processes coined \emph{extended Feller processes} which build on generalized Feller semigroups, but in contrast to generalized Feller processes the weight function enters in the definition of the respective conditional expectations. For some Baire-measurable function $f$ they are given by the quotient between the generalized Feller semigroup applied to the function $f \cdot \rho$ and the weight function $\rho$. Existence of these extended Feller processes is then proved in Theorem \ref{thm:contractive gFs on weighted space is Markov process} and Corollary \ref{cor:M smaller 1 gFs on weighted space is Markov process}, under the  condition that the generalized Feller semigroup is quasi-contractive.

We then also compare extended Feller processes with generalized Feller processes and notice in Proposition \ref{prop:gFp and gamma process comparison} that if both exist, their induced laws are equivalent measures.
In Section \ref{sec:relation} we establish a connection between generalized/extended Feller processes with 
classical ones and get the following two relations: first, on locally compact spaces  a Feller process is  generalized Feller if the Feller semigroup applied to the admissible weight function remains bounded for small times (see Proposition \ref{prop:Generalized Feller sometimes Feller process on locally compact space}); second, if  the admissible weight function is continuous, then $E$ is automatically locally compact and extended Feller processes can be reduced (modulo separability) to classical Feller processes (see Theorem \ref{thm:extended Feller process is extended Feller process}). In general, when $\rho$ is not continuous, extended Feller processes thus
generalize the notion of Feller processes to spaces $E$ that are only $\sigma$-compact, which explains their name.

As important examples of generalized Feller processes we consider in Section \ref{sec:transport} deterministic processes induced by semigroups of transport type. Indeed, we first show that generalized Feller semigroups of smooth operator algebra homomorphisms, precisely introduced in Definition~\ref{def:OH}, are given by 
\[
P(t)f=f \circ  \psi_t, \quad t \geq 0, \, f \in \mathscr{B}^\rho(E),
\]
where $\psi$ is a continuous semiflow in time (and also in space when restricted to compact sets).
The associated infinitesmial generator of such a generalized Feller semigroup is a smooth derivation, also called transport operator.
 This allows to treat transport equations on weighted spaces via strongly continuous semigroup theory. Moreover, if a transport operator generates a generalized Feller semigroup, the associated generalized Feller process corresponds to $(\psi_t)_{t \in \mathbb{R}_+}$, being actually the solution of a differential equation on the weighted space. In Section \ref{sec:affineandpoly} we consider then a stochastic setting with classical affine and polynomial processes on finite dimensional state spaces and show in Proposition \ref{prop:polynomial process is generalized Feller} and Corollary \ref{cor:affine process is generalized Feller} that under certain minor conditions polynomial and affine processes are generalized Feller processes. This adds
to the existing theory, since to date it is not known whether affine or polynomial processes on general state spaces are classically Feller or not. Finally, in Section \ref{eq:secP4} we also provide a counterexample, which shows that 
the additional 
 condition, called \textbf{P4}, used to define generalized Feller semigroups  and not needed for standard Feller processes \emph{cannot} be dropped, as then strong continuity does not necessarily hold true any more.

\subsection{Notation and basic definitions}

We here introduce notation needed throughout the paper.
For a topological space $E$ we denote its  Borel $\sigma$-algebra  by $\mathcal{B}(E)$. We write $C_{b}(E)$, $C_{0}(E)$, and $C_{c}(E)$  for the spaces of continuous maps that are bounded, vanish at infinity and have compact support, respectively. Moreover,  $\text{Id}$  denotes the identity operator on a given space. 

A \emph{transition kernel $\kappa$}  from a measurable space $\left(\Omega,\mathcal{F}\right)$ to a measurable space $\left(E,\mathcal{\mathcal{E}}\right)$\footnote{We use here some generic $\sigma$-algebra $\mathcal{E}$ which does not necessarily correspond to the Borel $\sigma$-algebra. Later on it will usually be either the Baire 
(see Definition \ref{def:Baire sigma algebra}) or the Borel $\sigma$-algebra.} is a map $\kappa: \Omega \times \mathcal{E} \to [0,\infty]$ such that for any fixed $B \in \mathcal{E}$, $\omega \to \kappa(\omega, B)$ is $\mathcal{F}$-measurable
and  every fixed $\omega \in \Omega$, $B \to \kappa(\omega, B)$ is a measure. It is called \emph{transition probability} if $\kappa(\omega, E)=1$ for all $\omega \in \Omega$.

If $\left(\Omega,\mathcal{F}\right)=\left(E,\mathcal{\mathcal{E}}\right)$ we simply speak of a transition kernel/probability on $\left(E,\mathcal{\mathcal{E}}\right)$. A family $\left(p(t)\right)_{t\in\mathbb{R}_{+}}$ of transition  probabilities on $\left(E,\mathcal{E}\right)$ is called $\mathit{semigroup}$ $\mathit{of}$ $\mathit{transition}$ $\mathit{probabilities}$  on $ \left(E,\mathcal{E}\right)$ if for all $ x\in E$, for all $s,t\in\mathbb{R}_{+}$ and all $A\in\mathcal{E}$
\[
p(s+t)(x,A)=\int_{E}p(s)(y,A)p(t)(x,dy)
\]
and
\[
p(0)(x,\cdot)=\delta_{x}
\]
hold. Here, $\delta_{x}$ denotes the Dirac measure. Similar notions apply of course to transition kernels which do not satisfy $\kappa(x, E) =1$ for all $x \in E$.  For transition kernels $\kappa$ with  $\kappa(x,E)\leq1$ for all $x\in E$, one can add a so-called \emph{cemetery state} $ \triangle $ to  $E$,
and define on $E_{\triangle}:=E\cup\left\{ \triangle\right\} $  a
transition probability
\[
\kappa':E_{\triangle}\times\sigma\left(\mathcal{\mathcal{\mathcal{E}}},\left\{ \triangle\right\} \right)\rightarrow\left[0,1\right]
\]
by
\[
\left.\kappa'\right|_{E\times\mathcal{\mathcal{E}}}=\kappa
\]
and  for any $x\in E$ and for any $A\in\mathcal E$
\begin{align*}
\kappa'\left( \triangle  ,\{ \triangle\right\}) &=1\\
\kappa'\left( \triangle ,A\right) & =0  \\
\kappa'\left(x,\{ \triangle\} \right) & =1-\kappa(x,E) .
\end{align*}
For any function $f$ on $E$ the convention
is to extend it to $E_{\triangle}$ by setting $f(\triangle)=0$.
Usually, the precise distinction between $\kappa'$ and $\kappa$
will not be made and $\kappa'$ will simply be called $\kappa$.
Finally, for  $t\in\mathbb{R}_{+}$ we define the $\mathit{translation}$ $\mathit{operator}$  via
 \[ 
\theta_{t}:\,E^{\mathbb{\mathbb{R}}_{+}}\rightarrow E^{\mathbb{\mathbb{R}}_{+}}, \quad
 \left(x(s)\right)_{s\in\mathbb{\mathbb{R}}_{+}}\rightarrow\left(x(s+t)\right)_{s\in\mathbb{\mathbb{R}}_{+}}.
 \]

\section{Notions of generalized Feller theory}\label{notionsgenFeller}

We here recall the essential notions of generalized Feller theory developed in particular in \cite{DoT} and prove some basic lemmas as well as the weighted Stone-Weierstrass theorem which will be used later on.

Let $E$  be a completely regular space. A map $\rho:\,E\rightarrow\left(0,\infty\right)$ is called $\mathit{admissible}$
$\mathit{weight}$ $\mathit{function}$ if the sets 
\[
K_{R}:=\left\{ x\in E:\,\rho(x)\leq R\right\} 
\]
are compact for all $R\geq0$. The pair $\left(E,\,\rho\right)$ is
called $\mathit{weighted}$ $\mathit{space}$. An admissible weight function is clearly lower semicontinuous and attains its minimum.

As proved in the subsequent lemma, the product space of weighted spaces is again a weighted space. This will be essential for the existence proof of generalized Feller processes in Theorem \ref{thm:GFS induce Markov process}.

\begin{lemma}
\label{lem:product space of weighted space is weighted space}Let
$\left(E_{i},\rho_{i}\right)$, $i\in\left\{ 1,...,n\right\} $ be
weighted spaces. Then 
\[
\left(E_{1}\times...\times E_{n},\rho\right)
\]
is a weighted space, where 
\[
\rho\left(x_{1},...,x_{n}\right):=\rho_{1}\left(x_{1}\right)\cdot\cdot\cdot\rho_{n}\left(x_{n}\right).
\]
\end{lemma}

\begin{proof} It is well known that the product space $E_{1}\times...\times E_{n}$ is completely regular. Without
loss of generality let $\rho_{i}\geq1$ for $i\in\left\{ 1,...,n\right\} $ and
let $R>0$ be arbitrary. Then 
\begin{align*}
 & \left\{ \left(x_{1},...,x_{n}\right)\in E_{1}\times...\times E_{n}:\,\rho_{1}\left(x_{1}\right)\cdot\cdot\cdot\rho_{n}\left(x_{n}\right)\leq R\right\} \\
 & \subset\left\{ x_{1}\in E_{1}:\,\rho_{1}\left(x_{1}\right)\leq R\right\} \times...\times\left\{ x_{n}\in E_{n}:\,\rho_{n}\left(x_{n}\right)\leq R\right\} .
\end{align*}
Since the right hand side is compact we only need to show closedness of the left hand side. For $y=\left(y_{1},...,y_{n}\right)$
such that $\rho\left(y_{1},...,y_{n}\right)>R$, by lower semicontinuity
of $\rho_{1}$,...$\rho_{n}$ for any $\varepsilon>0$ there exist open neighborhoods $U_{y_{1}}^{\varepsilon}\subset E_{1}$
of $y_{1}$, ..., $U_{y_{n}}^{\text{\ensuremath{\varepsilon}}}\subset E_{n}$
of $y_{n}$ such that for any $u_{i}\in U_{y_{i}}^{\text{\ensuremath{\varepsilon}}}$,
$i\in\left\{ 1,...,n\right\} $
\[
\rho_{i}(u_{i})>\rho_{i}(y_{i})-\varepsilon.
\]
Hence for $u\in U_{y_{1}}^{\varepsilon}\times...\times U_{y_{n}}^{\text{\ensuremath{\varepsilon}}}$
\[
\rho(u)>\left(\rho_{1}(y_{1})-\varepsilon\right)\cdot\cdot\cdot\left(\rho_{n}(y_{n})-\varepsilon\right)
\]
and the right hand side is larger than $R$ for $\varepsilon$ small
enough.
\end{proof}

Following \cite{DoT}, for an admissible
weight function $\rho$ and a Banach space $Z$ for  $f:\,E\rightarrow Z$  we define the map 
\[
\left\Vert \cdot\right\Vert _{\rho}:\,f\rightarrow\underset{x\in E}{\sup}\,\frac{\left\Vert f(x)\right\Vert }{\rho(x)}
\]
and
\[
B^{\rho}(E;Z):=\left\{ f:\,E\rightarrow Z:\,\left\Vert f(x)\right\Vert_{\rho} <\infty\right\} .
\]
By standard arguments it follows that $B^{\rho}(E;Z)$ is a Banach space with respect to the norm $\left\Vert \cdot\right\Vert_{\rho}$. Furthermore we set\[ \mathcal{\mathscr{B}}^{\rho}(E;Z) := \overline{C_{b}\left(E,Z\right)}^{\rho}
\] and $\mathcal{\mathscr{B}}^{\rho}(E) := \mathcal{\mathscr{B}}^{\rho}(E;\mathbb{R})$. Its elements are called \textit{functions with growth controlled by $\rho$}.

It was proved in
\cite{DoT} that the space $\mathcal{\mathscr{B}}^{\rho}(E)$
is closely related to the space of continuous maps on a compact space. 
\begin{theorem}
\label{thm: equivalence B-rho space} Let $f:\,E\rightarrow\mathbb{R}$.
Then $f\in\mathcal{\mathscr{B}}^{\rho}(E)$ if and only if \\
(i) for all $R>0$
\[
\left.f\right|_{K_{R}}\in C_{b}(K_{R},\mathbb{R}),
\]
 and \\
(ii)
\[
\underset{R\rightarrow\infty}{\lim}\underset{x\in E\setminus K_{R}}{\sup}\frac{\left|f(x)\right|}{\rho(x)}=0.
\]
\end{theorem}

In case of continuous admissible weight functions this relationship can be further specified, as stated in Lemma~\ref{lem: continuous weight function B-rho space is continuous } below. This will be in important in Section \ref{sec:relation} when we establish a relation to standard Feller processes. In particular, continuity of $\rho$ automatically implies local compactness of $E$.

\begin{lemma}
\label{lem: continuous weight function B-rho space is continuous }
If the admissible weight function $\rho$ is continuous, then \\
(i) E is locally compact, \\
(ii) $\mathcal{\mathscr{B}}^{\rho}(E)\subset C(E)$,\\
(iii) $f\in C_{0}(E)$ implies $f\cdot\rho\in\mathcal{\mathscr{B}}^{\rho}(E)$,\\
(iv) $f\in\mathcal{\mathscr{B}}^{\rho}(E)$ implies $\frac{f}{\rho}\in C_{0}(E)$.
\end{lemma}

\begin{proof}
\begin{enumerate}
\item[(i)]
If $\rho$ is continuous, then 
every point has a compact neighborhood
of type $\{\rho  \leq R \}$ for some $R > 0$, whence is  E is locally (for a converse statement under convexity see \cite[Remark 2.2]{CuT}).

\item [
(i)] For $f\in\mathcal{\mathscr{B}}^{\rho}(E)$ by definition of $\mathcal{\mathscr{B}}^{\rho}(E)$,
$\frac{f}{\rho}$ is the uniform limit of $\left(\frac{g_{n}}{\rho}\right)_{n\in\mathbb{N}}$
for some $\left(g_{n}\right)_{n\in\mathbb{N}}\subset C_{b}(E)$. Hence
$\frac{f}{\rho}$ is continuous and therefore also $f$.

\item [(ii)] If $f\in C_{0}(E)$, then $f\cdot\rho$ is continuous and $\underset{n\in\mathbb{N}}{\bigcup}\left\{ \rho<n\right\} $
is an open cover of $E$ hence for any $\varepsilon>0$ finitely many
such sets suffice to cover the compact set $\left\{ \left|f\right|\geq\varepsilon\right\} $.
Thus, for any $\varepsilon>0$ there exists $R_{\varepsilon}>0$ such
that $\left|f\right|<\varepsilon$ on $E\setminus K_{R_{\varepsilon}}$
and by Theorem \ref{thm: equivalence B-rho space} $f\cdot\rho\in\mathcal{\mathscr{B}}^{\rho}(E)$.

\item [(iii)] From  (i) it follows that $\frac{f}{\rho}$ is continuous. By Theorem \ref{thm: equivalence B-rho space}
for any $\varepsilon>0$ there is some $R'_{\varepsilon}>0$ such
that $\left\{ \left|\frac{f}{\rho}\right|\geq\varepsilon\right\} \subset K_{R'_{\varepsilon}}$.
Hence by closedness $\left\{ \left|\frac{f}{\rho}\right|\geq\varepsilon\right\} $
is compact and
\[
\frac{f}{\rho}\in C_{0}(E).
\]
\end{enumerate}
\end{proof}
The following Riesz representation theorem was proved in \cite{DoT}. We here add the observation that the precise statement of 
\cite[§ 5 Proposition 5]{bourbaki1969elements}
 that was used in the proof yields also uniqueness of the signed Radon measure. This allows to characterize the dual space of $\mathcal{\mathscr{B}}^{\rho}(E)$.

\begin{theorem}
(Riesz representation for $\mathcal{\mathscr{B}}^{\rho}(E)$) \label{thm:Riesz-representation-for B rho}
Let $\ell:\mathcal{\mathscr{B}}^{\rho}(E)\rightarrow\mathbb{R}$ be
a continuous linear map. Then, there exists a unique
signed Radon measure $\mu: \mathcal{B}(E) \to [-\infty, \infty]$ such that
\begin{eqnarray}
\ell(f)=\int_{E}f(x)\mu(dx) & & \text{for all }f\in\mathcal{\mathscr{B}}^{\rho}(E).\label{eq:functional b-rho}
\end{eqnarray}
Additionally, 
\[
\left\Vert \ell\right\Vert _{L\left(\mathcal{\mathscr{B}}^{\rho}(E),\mathbb{R}\right)}=\int_{E}\rho(x)\left|\mu\right|(dx).
\]
On the other hand, for any signed Radon measure $\mu$ for which
\[
\int_{E}\rho(x)\left|\mu\right|(dx) < \infty
\]
holds, the linear map
$
\mathcal{\mathscr{B}}^{\rho}(E)  \rightarrow\mathbb{R}, \quad 
f \rightarrow\int_{E}f(x)\mu(dx)
$
is  continuous.
\end{theorem}

\begin{definition}
We denote the space of signed Radon measures satisfying the conditions of Theorem \ref{thm:Riesz-representation-for B rho} by  $\mathcal{M}^{\rho}(E)$. This characterizes the dual space of $\mathcal{\mathscr{B}}^{\rho}(E)$.
\end{definition}

The above results, in particular Theorem \ref{thm: equivalence B-rho space}, show that the space $\mathcal{\mathscr{B}}^{\rho}(E)$
is closely related to  continuous functions on  compacts.
For these kinds of function spaces the Stone-Weierstrass theorem holds true. We now show that a weighted version thereof also holds for $\mathcal{\mathscr{B}}^{\rho}(E).$

\begin{proposition}[Stone-Weierstrass for $\mathcal{\mathscr{B}}^{\rho}(E)$]
\label{prop:Stone-Weierstrass on B-rho spaces}

Let $A\subset C_{b}(E)$ be an algebra with respect to pointwise multiplication
that contains $1_{E}$ and that separates points. Then $A$ is dense
in $\mathcal{\mathscr{B}}^{\rho}(E)$ with respect to $\left\Vert \cdot\right\Vert _{\rho}$.
\end{proposition}
\begin{proof}
The idea of the proof is to approximate elements of $\mathcal{\mathscr{B}}^{\rho}(E)$
by continuous bounded maps on $E$ which in turn can be approximated
on $K_{R}$ for any $R>0$ via the standard Stone-Weierstrass theorem \cite{stone48} by elements in $A$
that are restricted to $K_{R}$. However, such an element in $A$,
albeit bounded, may have an arbitrary large bound that depends on
$R$. Thus, it may not approximate with respect to $\left\Vert \cdot\right\Vert _{\rho}$
on all of $E$. Therefore, it is rescaled by a suitable polynomial
such that an element in $A$ is obtained whose bounds do not depend
on $R$. This yields an approximation with respect to $\left\Vert \cdot\right\Vert _{\rho}$.

In the following this idea will be made rigorous. Let $h\in\mathcal{\mathscr{B}}^{\rho}(E)$
and $\varepsilon>0$. By definition of $\mathcal{\mathscr{B}}^{\rho}(E)$
there exists $g_{\varepsilon}\in C_{b}(E)$ such that 
\[
\left\Vert g_{\varepsilon}-h\right\Vert _{\rho}<\varepsilon.
\]
Set 
\[
R_{\varepsilon}:=\max\left(\frac{\left\Vert g_{\varepsilon}\right\Vert _{\infty}}{\varepsilon},1\right).
\]
The set $A_{\varepsilon}\subset C_{b}(K_{R_{\varepsilon}})$ defined
as 
\[
A_{\varepsilon}:=\left\{ \left.f\right|_{K_{R_{\varepsilon}}}:\,f\in A\right\} 
\]
is an algebra that contains $1_{K_{R_{\varepsilon}}}$ and that separates
points, hence by the standard  Stone-\\Weierstrass theorem $A_{\varepsilon}$ is
dense in $C(K_{R_{\varepsilon}}).$ Thus, there is $f_{\varepsilon}\in A$
such that 
\[
\underset{x\in K_{R_{\varepsilon}}}{\sup}\left|f_{\varepsilon}(x)-g_{\varepsilon}(x)\right|<\varepsilon.
\]
Clearly, 
\[
\alpha_{\varepsilon}:=\underset{x\in K_{R_{\varepsilon}}}{\sup}\left|f_{\varepsilon}(x)\right|\leq\underset{x\in E}{\sup}\left|g_{\varepsilon}(x)\right|+\varepsilon=:\beta_{\varepsilon}.
\]
Set 
\[
\gamma_{\varepsilon}:=\underset{x\in E}{\sup}\left|f_{\varepsilon}(x)\right|.
\]
By the Tietze-Urysohn theorem (see e.g. \cite{kelley2017general}) 
there exists a continuous map 
\[
\varphi_{\varepsilon}:\,\left[-\gamma_{\varepsilon},\gamma_{\varepsilon}\right]\rightarrow\left[-\beta_{\varepsilon},\beta_{\varepsilon}\right]
\]
such that
\[
\varphi_{\varepsilon}(y)=\begin{cases}
y & \text{for }y\in\left[-\alpha_{\varepsilon},\alpha_{\varepsilon}\right]\\
\beta_{\varepsilon} & \text{for }\left|y\right|\geq\beta_{\varepsilon}.
\end{cases}
\]
Again by the standard Stone-Weierstrass theorem, on a compact set the
space of polynomials is dense in the space of continuous maps. This
means that there is a polynomial $p_{\varepsilon}$ on $\left[-\gamma_{\varepsilon},\gamma_{\varepsilon}\right]$
such that 
\[
\underset{y\in\left[-\gamma_{\varepsilon},\gamma_{\varepsilon}\right]}{\sup}\left|p_{\varepsilon}(y)-\varphi_{\varepsilon}(y)\right|<\varepsilon,
\]
hence
\[
\underset{x\in E}{\sup}\left|\frac{\left(p_{\varepsilon}\circ f_{\varepsilon}\right)(x)-\left(\varphi_{\varepsilon}\circ f_{\varepsilon}\right)(x)}{\rho(x)}\right|\leq\frac{\varepsilon}{\underset{x\in E}{\inf}\rho(x)}.
\]
Since $A$ is an algebra $p_{\varepsilon}\circ f_{\varepsilon}\in A$
and 
\begin{alignat*}{2}
\left\Vert h-p_{\varepsilon}\circ f_{\varepsilon}\right\Vert _{\rho} &  \leq &&\left\Vert     h-g_{\varepsilon}\right\Vert _{\rho}+\left\Vert g_{\varepsilon}-\varphi_{\varepsilon}\circ f_{\varepsilon}\right\Vert _{\rho}+\left\Vert \varphi_{\varepsilon}\circ f_{\varepsilon}-p_{\varepsilon}\circ f_{\varepsilon}\right\Vert _{\rho}\\
 & \leq && \varepsilon +\underset{x\in K_{R_{\varepsilon}}}{\sup}\left|\frac{g_{\varepsilon}(x)-\left(\varphi_{\varepsilon}\circ f_{\varepsilon}\right)(x)}{\rho(x)}\right|\\
  & &&+\underset{x\in E\setminus K_{R_{\varepsilon}}}{\sup}\left|\frac{g_{\varepsilon}(x)-\left(\varphi_{\varepsilon}\circ f_{\varepsilon}\right)(x)}{\rho(x)}\right|+\frac{\varepsilon}{\underset{x\in E}{\inf}\rho(x)}\\
 & \leq && \varepsilon +\underset{x\in K_{R_{\varepsilon}}}{\sup}\left|\frac{g_{\varepsilon}(x)-f_{\varepsilon}(x)}{\rho(x)}\right|\\
 & &&+2\,\underset{x\in E}{\sup}\left|\frac{\left|g_{\varepsilon}(x)\right|+\varepsilon}{R_{\varepsilon}}\right|+\frac{\varepsilon}{\underset{x\in K_{R_{\varepsilon}}}{\inf}\rho(x)}\\
& \leq  && \varepsilon +\frac{\varepsilon}{\underset{x\in K_{R_{\varepsilon}}}{\inf}\rho(x)}+2\left(\varepsilon+\varepsilon\right)+\frac{\varepsilon}{\underset{x\in K_{R_{\varepsilon}}}{\inf}\rho(x)},
\end{alignat*}
and $A$ is dense in $\mathcal{\mathscr{B}}^{\rho}(E)$ . 
\end{proof}

We now turn to $\mathit{generalized}$ $\mathit{Feller}$ $\mathit{semigroup}s$, which
 have been introduced by Röckner and Sobol in \cite{RoS} in a specific setting which in turn has been extended in \cite{DoT} to
general $\mathcal{\mathscr{B}}^{\rho}(E)$-spaces.

\begin{definition}
\label{def:generalized Feller semigroup}Let $\left(P(t)\right)_{t\in\mathbb{R}_{+}}$
be a family of bounded linear operators on $ \mathcal{\mathscr{B}}^{\rho}(E)$.
We call the family $\left(P(t)\right)_{t\in\mathbb{R}_{+}}$ $\mathit{generalized}$
$\mathit{Feller}$ $\mathit{semigroup}$ on
$\mathcal{\mathscr{B}}^{\rho}(E)$ if the following conditions hold true:\\
\textbf{P1} $P(0)=\text{Id}$, 
on $\mathcal{\mathscr{B}}^{\rho}(E)$, \\
\textbf{P2 $P(t+s)=P(s)\circ P(t)$ }for all $s,t\in\mathbb{R}_{+}$,
\textbf{}\\
\textbf{P3} for all $f\in\mathcal{\mathscr{B}}^{\rho}(E)$ and all
$x\in E$
\[
\underset{t\searrow0}{\lim}\,P(t)f(x)=f(x),
\]
\\
\textbf{P4} there exists $\varepsilon>0$ and $C\in\mathbb{R}$ such
that for all $t\in\left[0,\varepsilon\right]$
\[
\left\Vert P(t)\right\Vert _{L\left(\mathcal{\mathscr{B}}^{\rho}(E)\right)}\leq C,
\]
\\
\textbf{P5}  $P(t)$ is a positive linear operator
for all $t\in\mathbb{R}_{+}$.
\end{definition}

As proved in \textbf{\cite{DoT}}, this definition yields strong continuity, as stated in the next theorem.

\begin{theorem}
\label{thm:generalized Feller semigroups are strongly continuous}
Let $\left(P(t)\right)_{t\in\mathbb{R}_{+}}$ be a generalized Feller
semigroup on $\mathcal{\mathscr{B}}^{\rho}(E)$. Then $\left(P(t)\right)_{t\in\mathbb{R}_{+}}$
is strongly continuous on $\mathcal{\mathscr{B}}^{\rho}(E)$.
\end{theorem}
 
 Since we know the dual space of $\mathcal{\mathscr{B}}^{\rho}(E)$,
we can connect generalized Feller semigroups to a family of positive
finite Radon measures on $\left(E,\mathcal{B}(E)\right)$. 

\begin{lemma}\label{lem:GFS induces family of measures} 
Let
$\left(P(t)\right)_{t\in\mathbb{R}_{+}}$ be a generalized Feller
semigroup on $\mathcal{\mathscr{B}}^{\rho}(E)$ then there exists a unique family of positive finite Radon measures
\[
\left(p(t)(x,\cdot)\right)_{t\in\mathbb{R}_{+},x\in E}
\]
 on $\left(E,\mathcal{B}(E)\right)$ such that for
all $x\in E$, $t\in\mathbb{R}_{+}$ and $f\in\mathcal{\mathscr{B}}^{\rho}(E)$
\[
P(t)f(x)=\int_{E}f(y)p(t)(x,dy),
\]
and $p(t)(x,\cdot)\in\mathcal{M^{\rho}}(E)$. \\
\end{lemma} 
\begin{proof}
By definition of generalized Feller semigroups, for any $t\in\mathbb{R}_{+}$
and $x\in E$ the map
\begin{align*}
\mathcal{\mathscr{B}}^{\rho}(E) & \rightarrow\mathbb{R}\\
\ell_{t,x}:\,f & \rightarrow P(t)f(x)
\end{align*}
is positive, linear and continuous. Thus, by Theorem \ref{thm:Riesz-representation-for B rho}
for any $t\in\mathbb{R}_{+}$ and any $x\in E$ there is a unique
positive finite Radon measure $p(t)(x,\cdot)\in\mathcal{M}^{\rho}(E)$
such that 
\[
\left(P(t)f\right)(x)=\int_{E}f(y)p(t)(x,dy)
\]
holds true.\\
\end{proof}

Let us also recall the following observation  made in \cite{CuT} Remark 2.8.

\begin{remark}\label{rem: definition P_t rho} 
Let
$\left(P(t)\right)_{t\in\mathbb{R}_{+}}$ be a generalized Feller
semigroup on $\mathcal{\mathscr{B}}^{\rho}(E)$. Then it follows from standard semigroup theory that there exists some $M\geq1$ and $\omega\in\mathbb{R}$ such that for any
$t\in\mathbb{R}_{+}$
\[
\left\Vert P(t)\right\Vert \leq Me^{\omega t}.
\]
For all $x\in E$ and $t\in\mathbb{R}_{+}$ 
one then obtains the bounds
\[
P(t)\rho(x)=\underset{\begin{array}{c}
f\in C_{b}(E)\\
\left|f\right|\leq\rho
\end{array}}{\sup}\left|\left(P(t)f\right)(x)\right|\leq\rho(x)\left\Vert P(t)\right\Vert _{L(\mathcal{B}^{\rho}(E))}\leq\rho(x)Me^{\omega t}.
\]
\end{remark}

The following proposition states that
with respect to the Baire $\sigma$-algebra $\mathcal{B}_{0}(E)$
(see Definition \ref{def:Baire sigma algebra})
the family of positive finite Radon measures from Lemma \ref{lem:GFS induces family of measures} turns out to be  a semigroup
of transition probabilities.

\begin{proposition}

\label{prop:GFS implies existence of semigroup of transition kernels} Let
$\left(P(t)\right)_{t\in\mathbb{R}_{+}}$ be a generalized Feller
semigroup on $\mathcal{\mathscr{B}}^{\rho}(E)$. \\
The family $\left(\hat{p}(t)\right)_{t\in\mathbb{R}_{+}}$ defined
as the restriction 
\begin{eqnarray*}
\hat{p}(t)=\left.p(t)\right|_{E\times\mathcal{B}_{0}(E)} & & \text{ for } t\geq 0
\end{eqnarray*}
is a semigroup of transition kernels
on $(E, \mathcal{B}_{0}(E))$, i.e.~with respect to the Baire $\sigma$-algebra.
\end{proposition}

\begin{remark}\label{rem:BaireBorel}
Note that if $E$ is metrizable,  the Borel and the Baire $\sigma$-algebra coincide (see Corollary 6.3.5 in~\cite{Bogachev2007}). Hence in this case 
$\left(p(t)\right)_{t\in\mathbb{R}_{+}}$ is also a semigroup of transition
kernels on $(E, \mathcal{B}(E))$, i.e.~with respect to the Borel $\sigma$-algebra. For further conditions implying the coincidence of the Borel and the Baire $\sigma$-algebra we refer to Chapter 6.3 of ~\cite{Bogachev2007}.
\end{remark}

\begin{proof}
 For any $C_{b}(E)$-open set $A\in\mathcal{B}(E)$ (see Definition
\ref{def:C_b-open}) there exists a sequence $\left(f_{n}^{A}\right)_{n\in\mathbb{N}}\subset  C_{b}(E)$
such that $f_{n}^{A}\nearrow1_{A}$ pointwise \footnote{Note that by definition any $C_{b}(E)$-open set $A\in\mathcal{B}(E)$  is also Baire-measurable.}. Hence, by dominated convergence for any $t\in\mathbb{R}_{+}$

\[
x\rightarrow p(t)(x,A)=\underset{n\rightarrow\infty}{\lim}P(t)f_{n}^{A}(x)
\]
is measurable with respect to Baire $\sigma$-algebra $\mathcal{B}_{0}(E)$
as limit of maps that lie in $\mathcal{\mathscr{B}}^{\rho}(E)$ and are
therefore Baire-measurable (by virtue of being pointwise limits of
$C_{b}(E)$ functions). This property extends to all sets $A$ in
the Dynkin system generated by the $C_{b}(E)$-open sets. Since the
system of $C_{b}(E)$-open sets is intersection stable
the property holds true also for the $\sigma$-algebra generated by
the $C_{b}(E)$-open sets. By Lemma \ref{lem:C_b-open sets generate Baire sigma-algebra}
this is precisely $\mathcal{B}_{0}(E)$. 

Furthermore, for any $C_{b}(E)$-open set $A$ and any $s,t\in\mathbb{R}_{+}$
by dominated convergence
\begin{align*}
\int_{E}p(s)(y,A)p(t)(x,dy) & =\underset{n\rightarrow\infty}{\lim}P(t)P(s)f_{n}^{A}\\
 & =\underset{n\rightarrow\infty}{\lim}P(s+t)f_{n}^{A}\\
 & =p(s+t)(y,A).
\end{align*}
Since the system of sets such that this equations
holds is a Dynkin system it follows that 
\begin{align*}
\int_{E}p(s)(y,A)p(t)(x,dy) & =p(s+t)(x,A)
\end{align*}
holds true for any $A\in\mathcal{B}_0(E)$.
\end{proof}

\section{Generalized Feller processes}\label{sec:genfeller}

This section is mainly dedicated to prove existence and path properties of generalized Feller processes. Throughout $\left(E,\,\rho\right)$ denotes
a weighted space. We let $I$ be some index set and let $J\subset I$
be a finite subset.  We denote the number of elements in $J$ by $|J|$ and 
consider the product space
\[
E^{J}:=\underbrace{E \times \cdots \times E}_{|J|-\text{times}},
\]
whose elements are denoted by
$
x_{J}:=\left(x_{1},\ldots,x_{|J|}\right)
$.
We recall that by Lemma \ref{lem:product space of weighted space is weighted space}
,
$
\left(E^{J},\rho^{\otimes |J|}\right)
$
 is a weighted space with
$
\rho^{\otimes |J|}(x_{J}):=\rho\left(x_{{1}}\right)\cdots\rho\left(x_{|J|}\right).
$
Let us start by formally introducing generalized Feller processes.

\begin{definition} \label{def:generalized Feller process}
Let $\left(P(t)\right)_{t\in\mathbb{R}_{+}}$ be a generalized Feller
semigroup on $\mathcal{\mathscr{B}}^{\rho}(E)$, let $\nu\in\mathcal{M}^{\rho}(E)$
be a probability measure and let $\left(\lambda_{t}\right)_{t\in\mathbb{R}_{+}}$
be an adapted stochastic process on the filtered probability space
\[
\left(E^{\mathbb{R}_{+}},\mathcal{B}(E)^{\mathbb{R}_{+}},\left(\mathcal{F}_{t}\right)_{t\in\mathbb{R}_{+}},\mathbb{P}_{\nu}\right).
\]
If for any $t\geq s\geq0$ and any $f\in\mathcal{\mathscr{B}}^{\rho}(E)$
\begin{equation}
\mathbb{E}_{\nu}\left[\left.f(\lambda_{t})\right|\mathcal{F}_{s}\right]=P\left(t-s\right)f(\lambda_{s})\label{eq:GFS definition conditional expectation}
\end{equation}
holds true $\mathbb{P}_{\nu}$-almost surely and\textcolor{red}{{}
}
\[
\mathbb{P}_{\nu}\circ\lambda_{0}^{-1}=\nu
\]
 then $\left(\lambda_{t}\right)_{t\in\mathbb{R}_{+}}$ is called $\mathit{generalized}$
$\mathit{Feller}$ $\mathit{process}$ with respect to $\left(\mathcal{F}_{t}\right)_{t\in\mathbb{R}_{+}}$
and $\left(P(t)\right)_{t\in\mathbb{R}_{+}}$ and with initial distribution
$\nu$.

We make the convention $\mathbb{P}_{x}:=\mathbb{P}_{\delta_{x}}$ for any
$x\in E$. 
\end{definition}

\begin{remark}
\label{rem:generalized Feller process not Markov}
Note that  \ref{eq:GFS definition conditional expectation} is only required to hold for $f\in\mathscr{B}^{\rho}(E)$, thus only for Baire-measurable functions and therefore in general not necessarily for Borel-measurable
maps $f$. This is due to the fact, that indicator functions of Borel
sets can be approximated with continuous bounded functions by Corollary
\ref{cor:completely regular space, convergence of continuous bounded functions to open set}
only almost everywhere with respect to one (or countably many) measure(s),
but not necessarily simultaneously with respect to the entire family
of measures $\left(\left(p(t-s)\right)(x,\cdot)\right)_{\,x\in E}$
on $\left(E,\mathcal{B}(E)\right)$ obtained by Lemma \ref{lem:GFS induces family of measures}.
However, for indicator functions of $C_{b}(E)$-open sets
(see Definition \ref{def:C_b-open})  Equation~\eqref{eq:GFS definition conditional expectation}
holds true by dominated convergence. Since by Lemma \ref{lem:C_b-open sets generate Baire sigma-algebra}
the $C_{b}(E)$-open sets generate the Baire $\sigma$-algebra $\mathcal{B}_{0}(E)$
(see Definition \ref{def:Baire sigma algebra}) we conclude again
by dominated convergence that  Equation~\eqref{eq:GFS definition conditional expectation}
holds true for any indicator function of sets in $\mathcal{B}_{0}(E)$.
Thus, in this sense a generalized Feller process $\left(\lambda_{t}\right)_{t\in\mathbb{R}_{+}}$
 with initial distribution $\nu$ is a Markov process only with respect to the measurable space 
$
(E^{\mathbb{R}_{+}},\mathcal{B}_{0}(E)^{\mathbb{R}_{+}}),
$
and the probability measure $\mathbb{P}_{\nu}$
restricted to this space.

In cases where the Borel and the Baire $\sigma$-algebra coincide (such as for metrizable spaces, see Remark \ref{rem:BaireBorel}) there is of course no difference.
\end{remark}

One important result of the current paper is the following existence proof of generalized Feller processes. A similar statement was already formulated in \cite[Theorem 2.11]{CuT}, but the full proof with all details is given below.

\begin{theorem}
\label{thm:GFS induce Markov process}Let $\left(P(t)\right)_{t\in\mathbb{R}_{+}}$
be a generalized Feller semigroup on $\mathcal{\mathscr{B}}^{\rho}(E)$
such that for all $t\in\mathbb{R}_{+}$
\[
P(t)1=1.
\]
Then on the measurable space
$
(E^{\mathbb{R}_{+}},\mathcal{B}(E)^{\mathbb{R}_{+}})
$ 
for any probability measure  $\nu\in\mathcal{M}^{\rho}(E)$ there exists a measure $\mathbb{P}_{\nu}$
and a complete right continuous filtration $\left(\mathcal{F}_{t}\right)_{t\in\mathbb{R}_{+}}$ larger than the natural filtration of the canonical process $\left(\lambda_{t}\right)_{t\in\mathbb{R}_{+}}$
such that $\left(\lambda_{t}\right)_{t\in\mathbb{R}_{+}}$ is adapted with respect to $\left(\mathcal{F}_{t}\right)_{t\in\mathbb{R}_{+}}$ and for any $t\geq s\geq0$ and any $f\in\mathcal{\mathscr{B}}^{\rho}(E)$

\begin{equation} \label{eq:equation of GFS yields Markov process}
\mathbb{E}_{\nu}\left[\left.f(\lambda_{t})\right|\mathcal{F}_{s}\right]=P\left(t-s\right)f(\lambda_{s})
\end{equation}
holds true $\mathbb{P}_{\nu}$ -almost surely, and 
\[
\mathbb{\mathbb{P}}_{\nu}\circ\lambda_{0}^{-1}=\nu.
\]
\end{theorem}

When we speak of generalized Feller processes we mean those
obtained via Theorem~\ref{thm:GFS induce Markov process}. 
We remind the reader that while $\mathcal{B}(E^{\mathbb{R}_{+}})\supset\mathcal{B}(E)^{\mathbb{R}_{+}}$
holds true (because on $\mathcal{B}(E^{\mathbb{R}_{+}})$ every projection
is continuous, hence measurable with respect to $\mathcal{B}(E^{\mathbb{R}_{+}})$),
the inclusion $\mathcal{B}(E^{\mathbb{R}_{+}})\subset\mathcal{B}(E)^{\mathbb{R}_{+}}$
is in general not true when the topology of $E$ does not have a countable
base.

\begin{proof}
The proof has three steps. In the first step, we construct a projective
family of probability measures 
on 
\[
\left(E^{J},\mathcal{B}\left(E^{J}\right)\right)_{J\subset\mathbb{R}_{+},\text{ finite}}.
\]
In the second step we use the generalized Kolmogorov extension theorem (Theorem 15.26 in \cite{AlB}) and obtain a probability measure on $\left(E^{\mathbb{R}_{+}},\mathcal{B}(E)^{\mathbb{R}_{+}}\right)$.
The coordinate process $\left(\lambda_{t}\right)_{t\in\mathbb{R}_{+}}$
on this space then yields for any $t\geq s\geq0$ and any $f\in\mathcal{\mathscr{B}}^{\rho}(E)$
\begin{equation}
\mathbb{E}_{\nu}\left[\left.f(\lambda_{t})\right|\mathcal{F}_{s}^{0}\right]=P\left(t-s\right)f(\lambda_{s}),\label{eq:conditional expectation gFp, natural filtration}
\end{equation}
 where $\left(\mathcal{F}_{t}^{0}\right)_{t\in\mathbb{R}_{+}}$ is
the natural filtration of the coordinate process. In the third step,
we take the right continuous extension of this filtration and show
Equation \eqref{eq:equation of GFS yields Markov process}

\emph{Step 1:} Fix some probability measure  $\nu\in\mathcal{M}^{\rho}(E)$.
For any $r\geq0$ let $p_{\nu}(r)(\cdot)$ be the unique measure in $\mathcal{M}^{\rho}(E)$
given by Theorem  \ref{thm:Riesz-representation-for B rho}
via 
\[
\int_{E}P(r)f(y)\nu(dy)=\int_{E}f(y)p_{\nu}(r)(dy).
\]
Using the  Riesz representation theorem on $E\times E$ and Lemma \ref{lem:B-rho tensor product space, linear isomorphism} we define, for $0\leq r_{1}<r_{2}$ and $R >0$,  the unique measure $\mu_{\nu}^{R,\left\{ r_{1},r_{2}\right\} }\in\mathcal{M}^{\rho^{\otimes 2}}(E\times E)$
via the continuous functional  
\begin{align*}
f_{r_{1}}\cdot f_{r_{2}} & \rightarrow\int_{E}\left(1_{\left\{ \rho (y)<R\right\} }\cdot f_{r_{1}}(y)\cdot P(r_{2}-r_{1})f_{r_{2}}(y)\right)p_{\nu}(r_{1})(dy)
\end{align*}
for $f_{r_{1}}, f_{r_2} \in\mathscr{B}^{\rho}(E)$ 
such that
\begin{align*}
 \int_{E}\left(1_{\left\{ \rho(y)<R\right\} }\cdot f_{r_{1}}(y)\cdot P(r_{2}-r_{1})f_{r_{2}}(y)\right)p_{\nu}(r_{1})(dy)
 = \,\int_{E\times E}f_{r_{1}}(y)f_{r_{2}}(z)\mu_{\nu}^{R,\left\{ r_{1},r_{2}\right\} }(dy,dz).
\end{align*}
By $P(r_{2}-r_{1})1=1$ we obtain
\begin{align*}
\mu_{\nu}^{R,\left\{ r_{1},r_{2}\right\} }(E\times E) & =p_{\nu}(r_{1})\left(\left\{ \rho(y)<R\right\} \right)
  \leq p_{\nu}(r_{1})(E)
 = \nu(E)
 = 1.
\end{align*}
Then for any $A\in\mathcal{B}\left(E\times E\right)$
by monotonicity and boundedness, we can define
\[
p_{\nu}^{\left\{ r_{1},r_{2}\right\} }(A):=\underset{R\rightarrow\infty}{\lim}\mu_{\nu}^{R,\left\{ r_{1},r_{2}\right\} }(A).
\]
One can easily show that the convergence is uniform and that $p_{\nu}^{\left\{ r_{1},r_{2}\right\} }$
is a probability measure on
$
\left(E\times E,\mathcal{B}\left(E\times E\right)\right).
$
 Furthermore, for any $r_{3}>r_{2}$ by Lemma \ref{lem:B-rho tensor product space, linear isomorphism} we can define a continuous functional $j^{R,\left\{ r_{1},r_{2},r_{3}\right\}}$  on $\mathcal{\mathscr{B}}^{\rho^{\otimes 3}}(E\times E\times E)$  by setting for any $f_{r_{i}}\in \mathcal{\mathscr{B}}^{\rho}(E)$,  $i=1,2,3$
\[
f_{r_{1}}\cdot f_{r_{2}}\cdot f_{r_{3}}\rightarrow \int_{E_{r_{1}}\times E_{r_{2}}}1_{\left\{ \rho_{r_{1}}(y)<R\right\} }\cdot1_{\left\{ \rho_{r_{2}}(z)<R\right\} }\cdot f_{r_{1}}(y)\cdot f_{r_{2}}(z)\cdot P(r_{3}-r_{2})f_{r_{3}}(z)p_{\nu}^{\left\{ r_{1},r_{2}\right\} }(dy,dz).
\]
Then again by the Riesz representation theorem
we define
 $\mu_{\nu}^{R,\left\{ r_{1},r_{2},r_{3}\right\} }$ as
the unique measure in $\mathcal{M}^{\rho^{\otimes 3}}(E\times E\times E)$ such that for any $
f\in\mathcal{\mathscr{B}}^{\rho^{\otimes 3}}(E\times E\times E)$ 
\[
j^{R,\left\{ r_{1},r_{2},r_{3}\right\} }(f)=\int_{E\times E \times E}f(x,y,z)\mu_{\nu}^{R,\left\{ r_{1},r_{2},r_{3}\right\} }(dy,dy,dz).
\]
Again for any $A\in\mathcal{B}\left(E\times E\times E\right)$
by monotonicity and boundedness
\[
p_{\nu}^{\left\{ r_{1},r_{2},r_{3}\right\} }(A):=\underset{R\rightarrow\infty}{\lim}\mu_{\nu}^{R,\left\{ r_{1},r_{2},r_{3}\right\} }(A).
\]
Proceeding inductively  in this way we can define via $\mu_{\nu}^{R,J}\in\mathcal{M}^{\rho^{\otimes |J|}}(E^{J})$ a family of probability measures
$
\left(p_{\nu}^{J}\right)_{J\subset\mathbb{R}_{+},\,\text{finite }}
$
on the respective measurable spaces
$$
\left(E^{J},\mathcal{B}\left(E^{J}\right)\right)_{J\subset\mathbb{R}_{+},\,\text{finite }}.$$
 By uniform convergence inner regularity of $\mu_{\nu}^{R,J}\in\mathcal{M}^{\rho^{\otimes |J|}}(E^{J})$ for each $R>0$ and finite
$J\subset\mathbb{R}_{+}$ implies
 that  the measure $p_{\nu}^{J}$ is inner regular,
hence a Radon measure. 

In order to apply the generalized Kolmogorov extension theorem (Theorem 15.26 in \cite{AlB}), we need to show
that this family is projective, i.e. for any finite $J$ and $i\in J$
and any $A\in \mathcal{B}\left(E^{J\setminus\left\{ i\right\} } \right)$
\[
p_{\nu}^{J}(A\times E)=p_{\nu}^{J\setminus\left\{ i\right\} }(A).
\]
We show this property by induction and start with the case $J=\left\{ r_{1},r_{2}\right\} $.
For $f\in C_{b}(E)$ 
\[
\int_{E}f(y)\mu_{\nu}^{R,\left\{ r_{1},r_{2}\right\} }(dy \times E)=\int_{E}f(y)1_{K_{R}}(y)p_{\nu}(r_{1})(dy),
\]
which implies by uniqueness of the Radon measure (see Proposition
\ref{prop:condition when linear funtional on Cb is measure real case})
for any $A\in\mathcal{B}\left(E\right)$
\[
\mu_{\nu}^{R,\left\{ r_{1},r_{2}\right\} }(A \times E)=p_{\nu}(r_{1})(A\cap K_{R}).
\]
Thus, 
\begin{align*}
p_{\nu}^{\left\{ r_{1},r_{2}\right\} }(A\times E) & =\underset{R\rightarrow\infty}{\lim}\mu_{\nu}^{R,\left\{ r_{1},r_{2}\right\} }(A\times E)
 =p_{\nu}(r_{1})(A).
\end{align*}
Furthermore,  by interchangeability of the limits due to the uniform convergence of the measures $(\mu_{\nu}^{R,\left\{ r_{1},r_{2}\right\} })_R$ and dominated
convergence we get
for any $f\in C_{b}(E)$ 
\begin{align*}
  \int_{E \times E}1(y)f(z)p_{\nu}^{\left\{ r_{1},r_{2}\right\} }(dy,dz)
 &= \,\underset{R\rightarrow\infty}{\lim}\int_{E\times E}1(y)f(z)\mu_{\nu}^{R,\left\{ r_{1},r_{2}\right\} }(dy,dz)\\
& = \,\underset{R\rightarrow\infty}{\lim}\int_{E}1_{\left\{ \rho(y)<R\right\} }P(r_{2}-r_{1})f(y)p_{\nu}(r_{1})(dy) \\
&= \,P(r_{1})P(r_{2}-r_{1})f(x)= \,P(r_{2})f(x)= \,\int_{E}f(z)p_{\nu}(r_{2})(dz).
\end{align*}
Since the functionals
$f \rightarrow\int_{E}f(z)p_{\nu}^{\left\{ r_{1},r_{2}\right\} }(E \times dz)$
and 
$
f  \rightarrow\int_{E}f(z)p_{\nu}(r_{2})(dz)$
coincide on $C_{b}(E)$ and satisfy the conditions of Proposition
\ref{prop:condition when linear funtional on Cb is measure real case}, we obtain 
by uniqueness of the Radon measure 
that
for any $A\in\mathcal{B}\left(E\right)$ 
\[
p_{\nu}^{\left\{ r_{1},r_{2}\right\} }(E\times A)=p_{\nu}(r_{2})(A).
\]
This implies in particular that for any $f_{r_{2}}\in\mathscr{B}^{\rho}(E)$
\[
\int_{E\times E}1(y)f_{r_{2}}(z)p_{\nu}^{\left\{ r_{1},r_{2}\right\} }(dy,dz)<\infty.
\]
Next, we assume that for $N\in\mathbb{N}$, 
$n\leq N$, any arbitrary  index set $J_{n}:=\left\{ r_{1},...,r_{n}\right\} \in\mathbb{R}_{+}^{n}$ with
, $0\leq r_{1}<...<r_{n}$, any $i\in\left\{ 1,...,n\right\} $
and any $A\in\mathcal{B}\left(E^{J_{n}\setminus\left\{ r_{i}\right\} }\right)$ we have
\[
p_{\nu}^{J_{n}}(A\times E)=p_{\nu}^{J_{n}\setminus\left\{ r_{i}\right\} }(A),
\]
and 
\[
\int_{E^{J_n}}1(y_{1})\cdot...1(y_{n-1})\cdot f_{r_{n}}(y_{n})p_{\nu}^{J_{n}}(dy_{1},...,dy_{n})<\infty
\]
for any $f_{r_{n}}\in\mathscr{B}^{\rho}(E)$.
 We want to show the analogous assertions for any $J_{N+1}:=\left\{ r_{1},...,r_{N+1}\right\} \in\mathbb{R}_{+}^{N+1}$ with $r_{N+1}>...>r_{1}\geq0$,  for any $i\in\left\{ 1,...,N+1\right\} $
and any $A\in\mathcal{B}\left(E^{J_{N+1}\setminus\left\{ r_{i}\right\} }\right)$.
For  $i=N+1$ and for $f_{r_{i}}\in C_b(E)$  we have by definition of the measures and dominated convergence
\begin{align*}
 & \int_{E^{J_{N+1}}}f_{r_{1}}(y_{1})\cdot...\cdot f_{r_{N}}(y_{N})1(y_{N+1})p_{\nu}^{J_{N+1}}(dy_{1},...,dy_{N+1})\\
 = & \,\underset{R\rightarrow\infty}{\lim}\int_{E^{J_{N+1}}}f_{r_{1}}(y_{1})\cdot...\cdot f_{r_{N}}(y_{N})1(y_{N+1})\mu_{\nu}^{R,J_{N+1}}(dy_{1},...,dy_{N+1})\\
 =& \,\int_{E^{J_{N}}}f_{r_{1}}(y_{1})\cdot...\cdot f_{r_{N}}(y_{N})p_{\nu}^{J_{N}}(dy_{1},...,dy_{N}),
\end{align*}
and we conclude that for any $A\in\mathcal{B}\left(E^{J_{N+1}\setminus\left\{ r_{N+1}\right\} }\right)$
\[
p_{\nu}^{J_{N+1}}(A\times E)=p_{\nu}^{J_{N+1}\setminus\left\{ r_{N+1}\right\} }(A),
\]
by  Proposition \ref{prop:condition when linear funtional on Cb is measure real case}.
Furthermore, 
\begin{align*}
 & \int_{E^{J_{N+1}}}1(y_{1})\cdot...1(y_{N})\cdot f_{r_{N+1}}(y_{N+1})p_{\nu}^{J_{N+1}}(dy_{1},...,dy_{N+1})\\
 =& \,\int_{E^{J_{N+1}}}1(y_{1})\cdot...1(y_{N})\cdot P(r_{N+1}-r_{N})f_{r_{N+1}}(y_{N})p_{\nu}^{J_{N}}(dy_{1},...,dy_{N})
 < \,\infty,
\end{align*}
by assumption.\\
In case $i=N$ and for $f_{r_{i}}\in C_{b}(E)$ and by the same arguments as above 
\begin{align*}
 & \int_{E^{J_ {N+1}}}f_{r_{1}}(y_{1})\cdot...\cdot f_{r_{N-1}}(y_{N-1})1(y_{N})f_{r_{N+1}}(y_{N+1})p_{\nu}^{J_{N+1}}(dy_{1},...,dy_{N+1})\\
 =& \,\underset{R\rightarrow\infty}{\lim}\int_{E^{J_ {N+1}}}f_{r_{1}}(y_{1})\cdot...\cdot f_{r_{N-1}}(y_{N-1})\cdot f_{r_{N+1}}(y_{N+1})\mu_{\nu}^{R,J_{N+1}}(dy_{1},...,dy_{N+1})\\
 =& \,\int_{{E^{J_{N}}}}f_{r_{1}}(y_{1})\cdot...\cdot f_{r_{N-1}}(y_{N-1})\cdot P(r_{N+1}-r_{N})f_{r_{N+1}}(y_{N})p_{\nu}^{J_{N}}(dy_{1},...,dy_{N})\\
 =& \,\int_{E^{J_{N-1}}}f_{r_{1}}(y_{1})\cdot...\cdot f_{r_{N-1}}(y_{N-1})\cdot P(r_{N+1}-r_{N-1})f_{r_{N+1}}(y_{N-1})p_{\nu}^{J_{N-1}}(dy_{1},...,dy_{N-1})\\
 =& \,\int_{{E^{J_{N}}}}f_{r_{1}}(y_{1})\cdot...\cdot f_{r_{N-1}}(y_{N-1})\cdot f_{r_{N+1}}(y_{N+1})p_{\nu}^{J_{N+1}\setminus\left\{ r_{N}\right\} }(dy_{1},...,dy_{N-1},dy_{N+1}),
\end{align*}
and we can conclude as before. 
For $i\in\left\{ 1,...,N-1\right\} $ the desired properties follow
in the same way by definition of $p_{\nu}^{J_{N+1}}$  and from the assumption that the properties
hold true for any $n\leq N$. 
Thus, by induction it follows that for any $m\in\mathbb{N}$ and any
arbitrary finite index set $J_{m}:=\left\{ r_{1},...,r_{m}\right\} \in\mathbb{R}_{+}^{m}$
, $0\leq r_{1}<...<r_{m}$ for any $i\in\left\{ 1,...,m\right\} $
and any $A\in\mathcal{B}\left(E^{J_{m}\setminus\left\{ r_{i}\right\} }\right)$
\[
p_{\nu}^{J_{m}}(A\times E)=p_{\nu}^{J_{m}\setminus\left\{ r_{i}\right\} }(A).
\]
Therefore, the family 
$
\left(p_{\nu}^{J}\right)_{J\subset\mathbb{R}_{+},\,\text{finite }}$ 
 is projective. 
 
\emph{Step 2:}  In order to construct a measure $\mathbb{P}_{\nu}$ on $\left(E^{\mathbb{R}_{+}},\mathcal{B}(E)^{\mathbb{R}_{+}}\right)$
for any $\nu\in\mathcal{M}^{\rho}(E)$ 
we shall use Theorem 15.26 in \cite{AlB}. For this purpose 
we have to find a compact class (see Definition \ref{def:compact class})
$\mathcal{C}$ in $E$ such that for each $t\in\mathbb{R}_{+}$ and
$A\in\mathcal{B}(E)$ 
\begin{equation}
p_{\nu}^{\left\{ t\right\} }(A)=\sup\left\{ p_{\nu}^{\left\{ t\right\} }(C):\,C\subset A\text{\,and }C\in\mathcal{C}\right\} .\label{eq:proof GFS is Markov, equation identity comapct class}
\end{equation}
We show that for the following set
\[
\mathcal{C}:=\left\{ C:\,C\text{ compact},\,C\subset K_{R}\text{ for some }R\geq0\right\}. 
\]
To prove that $\mathcal{C}$ is a compact class, we  choose some arbitrary sequence $\left\{ C_{l}\right\} _{l\in\mathbb{N}}\subset\mathcal{C}$
such that
\[
\underset{l\in\mathbb{N}}{\bigcap}\, C_{l}=\emptyset.
\]
For $C_{1}$ we choose $R_{1}\geq0$ such that $C_{1}\subset K_{R_{1}}.$
Then
$
\underset{l\in\mathbb{N}}{\bigcup}\, E\setminus C_{l}\supset K_{R_{1}}
$
is an open cover of the compact set $K_{R_{1}}$ hence finitely many
sets, say without loss of generality $\left\{ E\setminus C_{2},E\setminus C_{2},...,E\setminus C_{m}\right\} ,$
suffice to cover it. Thus, 
\[
\underset{l\in\left\{ 2,...,m\right\} }{\bigcap}C_{l}\cap K_{R_{1}}=\emptyset
\]
and $C_{1}\subset K_{R_{1}}$ yields that
$
\underset{l\in\left\{ 1,...,m\right\} }{\bigcap}C_{l}=\emptyset
$. 
Hence, $\mathcal{C}$ is a compact class.
Since 
$
\underset{R\geq0}{\bigcup}K_{R}=E,
$
 for any $\varepsilon>0$ and $t\in\mathbb{R}_{+}$ there is some
$R_{\varepsilon}\geq0$ such that 
\[
p_{\nu}^{\left\{ t\right\} }\left(E\right)-p_{\nu}^{\left\{ t\right\} }\left(K_{R_{\varepsilon}}\right)<\varepsilon,
\]
and by inner regularity of the Radon measure $p_{\nu}^{\left\{ t\right\} }$
for any $A\in\mathcal{B}(E)$ there is a compact set $A_{\varepsilon}\subset A$
such that 
\[
p_{\nu}^{\left\{ t\right\} }\left(A\right)-p_{\nu}^{\left\{ t\right\} }\left(A_{\varepsilon}\right)<\varepsilon.
\]
Hence 
\[
p_{\nu}^{\left\{ t\right\} }\left(A\right)-p_{\nu}^{\left\{ t\right\} }\left(A_{\varepsilon}\cap K_{R_{\varepsilon}}\right)<2\varepsilon,
\]
and Equation \eqref{eq:proof GFS is Markov, equation identity comapct class}
holds true. Thus, the conditions of Theorem 15.26 in \cite{AlB}
are satisfied. 
By applying this theorem, we obtain a probability measure $\mathbb{P}_{\nu}$
on the measurable space $\left(E^{\mathbb{R}_{+}},\mathcal{B}(E)^{\mathbb{R}_{+}}\right)$
such that for any finite $J\subset\mathbb{R}_{+}$ and
$
A\in\mathcal{B}(E^J)$
 the probability is given by
\begin{equation}
\mathbb{P}_{\nu}\left(\left(\Pi_{J}^{\mathbb{R}_{+}}\right)^{-1}\left(A\right)\right)=p_{\nu}^{J}(A)\label{eq:projectivity of measure},
\end{equation}
with $\Pi_{J}^{\mathbb{R}_{+}}$ being the projection from $E^{\mathbb{R}_{+}}$
on $E^{J}$. 
Let now $\left(\lambda_{t}\right)_{t\in\mathbb{R}_{+}}:=\left(\Pi_{t}^{\mathbb{R}_{+}}\right)_{t\in\mathbb{R}_{+}}$
be the the coordinate process 
on $\left(E^{\mathbb{R}_{+}},\mathcal{B}(E)^{\mathbb{R}_{+}}\right)$.
Then, 
by definition 
\[
\mathbb{P}_{\nu}\circ\left(\lambda_{0}\right)^{-1}=p_{\nu}^{\left\{ 0\right\} }=\nu,
\]
 and for any $f\in\mathcal{\mathscr{B}}^{\rho}(E)$
\[
\mathbb{E}_{\nu}\left[f(\lambda_{t})\right]=\int_{E}\left(P(t)f\right)(y)\nu(dy)<\infty.
\]

We now denote by $\left(\mathcal{F}_{t}^{0}\right)_{t\in\mathbb{R}_{+}}$ the natural filtration of $(\lambda_t)_{t \in \mathbb{R}_+}$ and show Equation~\eqref{eq:conditional expectation gFp, natural filtration},~i.e.~
\[
\mathbb{E}_{\nu}\left[f(\lambda_{t})\cdot1_{F}\right]=\mathbb{E}_{\nu}\left[P\left(t-s\right)f(\lambda_{s})\cdot1_{F}\right]
\]
for any $f\in\mathcal{\mathscr{B}}^{\rho}(E)$,
 $0\leq s<t$, $F\in\mathcal{F}_{s}^{0}$ and  probability measure  $\nu\in\mathcal{M}^{\rho}(E)$. By $\mathbb{E}_{\nu}\left[f(\lambda_{t})\right]<\infty$
 it is
enough to check 
\[
\mathbb{E}_{\nu}\left[f(\lambda_{t})\cdot1_{G}\right]=\mathbb{E}_{\nu}\left[P\left(t-s\right)f(\lambda_{s})\cdot1_{G}\right]
\]
for all $G\in\mathcal{G}$ of an intersection stable generator $\mathcal{G}\subset\mathcal{F}_{s}^{0}$.
The set 
\[
\left\{ \underset{j\in J}{\bigcap}\left(\lambda_{j}\right)^{-1}\left(O_{j}\right):\,J\subset\left[0,s \right]\,\text{ finite, }O_{j}\subset E\text{ open for all}\,j\in J\right\} 
\]
is such an intersection stable generator. We
fix $k\in\mathbb{N}$ and $0\leq r_{1}\leq r_{2}\leq...\leq r_{k}\leq s$
and a set $J:=\left\{ r_{1},r_{2},...,r_{k},s,t\right\} .$ For
$O_{r_{i}}\in E$ open we have by
definition 
\begin{align*}
 & \mathbb{E}_{\nu}\left[f(\lambda_{t})\cdot1_{O_{r_{1}}}\left(\lambda_{r_{1}}\right)\cdot...\cdot1_{O_{r_{k}}}\left(\lambda_{r_{k}}\right)\right]\\
 & =\int_{E^{J}}\left(f(x_t)\cdot1_{O_{r_{1}}}\left(x_{1}\right)\cdot...\cdot1_{O_{r_{k}}}\left(x_{k}\right)\cdot1_{E}(x_{s})\right)p_{\nu}^{J}(dx_{1},...,dx_{k},dx_s,dx_t)
\end{align*}

and 
\begin{align*}
 & \mathbb{E}_{\nu}\left[P\left(t-s\right)f(\lambda_{s})\cdot1_{O_{r_{1}}}\left(\lambda_{r_{1}}\right)\cdot...\cdot1_{O_{r_{k}}}\left(\lambda_{r_{k}}\right)\right]\\
 & =\int_{E^{J\setminus\left\{ t\right\} }}\left(P\left(t-s\right)f(x_s)\cdot1_{O_{r_{1}}}\left(x_{1}\right)\cdot...\cdot1_{O_{r_{k}}}\left(x_{k}\right)\right)p_{\nu}^{J \setminus\left\{ t\right\} }(dx_{1},...,dx_{k},dx_s).
\end{align*}
By invoking again Proposition \ref{prop:condition when linear funtional on Cb is measure real case}
it is therefore enough to show that the right hand sides of the above equations  
coincide on $C_{b}(E)$.
By Corollary \ref{cor:completely regular space, convergence of continuous bounded functions to open set}
there exist sequences of maps $\left(b_{l}^{i}\right)_{l\in\mathbb{N},i\in\left\{ 1,...,k\right\} }$
where $b_{l}^{i}\in C_{b}\left(E\right)$ for any $l\in\mathbb{N}$
and $i\in\left\{ 1,...,k\right\} $ such that 
$
\underset{i\in\left\{ 1,...,k\right\} }{\prod}b_{l}^{i}$ tends to $1_{O_{r_{1}}\times...\times O_{r_{k}}\times E_{s}\times E_{t}}$ as $l \to \infty$
 $p_{\nu}^{J}$-almost surely and 
$
\underset{i\in\left\{ 1,...,k\right\} }{\prod}b_{l}^{i}$ to $1_{O_{r_{i}}\times...\times O_{r_{k}}\times E_{s}}$, $p_{\nu}^{J\setminus\left\{ t\right\} }$-almost surely. For
$f\in C_{b}(E_{s})$ by the assumption $P\left(t-s\right)1=1$ the
map $P\left(t-s\right)f$ remains bounded and 
\begin{align*}
 & \int_{E^{J}}\left(f(x_{t})\cdot1_{O_{x_{1}}}\left(x_{1}\right)\cdot...\cdot1_{O_{x_{k}}}\left(x_{k}\right)\right)p_{\nu}^{J}(dx_{1},...,dx_{k},dx_{s},dx_{t})\\
 =& \,\underset{l\rightarrow\infty}{\lim}\int_{E^{J}}\left(f(x_t)\cdot b_{l}^{1}\left(x_{1}\right)\cdot...\cdot b_{l}^{k}\left(x_{k}\right)\right)p_{\nu}^{J}(dx_{1},...,dx_{k},dx_{s},dx_{t})\\
 =& \,\underset{l\rightarrow\infty}{\lim}\,
 \underset{R\rightarrow\infty}{\lim}\int_{E^{J}}1_{\left\{ \rho_{r_{1}}(x_1)<R\right\} }\cdot...\cdot1_{\left\{ \rho_{r_{k}}(x_k)<R\right\} }\cdot...\\
 & \,\cdot\left(f(x_t)\cdot b_{l}^{1}\left(x_{1}\right)\cdot...\cdot b_{l}^{k}\left(x_{k}\right)\right)\mu_{\nu}^{R,J}(dx_{1},...,dx_{k},dx_{s},dx_{t})\\
 =&\, \underset{l\rightarrow\infty}{\lim}\,\underset{R\rightarrow\infty}{\lim}\int_{E^{J\setminus\left\{ x_t\right\} }}1_{\left\{ \rho_{r_{1}}(x_1)<R\right\} }\cdot...\cdot1_{\left\{ \rho_{r_{k}}(x_k)<R\right\} }\cdot...\\
 & \,\cdot\left(P\left(t-s\right)f(x_s)\cdot b_{l}^{1}\left(x_{1}\right)\cdot...\cdot b_{l}^{k}\left(x_{k}\right)\right)\mu_{\nu}^{R,J\setminus\left\{ t\right\} }(dx_{1},...,dx_{k},dx_{s})\\
 =& \,\underset{l\rightarrow\infty}{\lim}\int_{E^{J\setminus\left\{ t\right\} }}\left(P\left(t-s\right)f(x_s)\cdot b_{l}^{1}\left(x_{1}\right)\cdot...\cdot b_{l}^{k}\left(x_{k}\right)\right)p_{\nu}^{J\setminus\left\{ t\right\} }(dx_{1},...,dx_{k},dx_{s})\\
 =&\, \int_{E^{J\setminus\left\{ t\right\} }}\left(P\left(t-s\right)f(x_s)\cdot1_{O_{r_{1}}}\left(x_{1}\right)\cdot...\cdot1_{O_{x_{k}}}\left(x_{k}\right)\right)p_{\nu}^{J\setminus\left\{ t\right\} }(dx_{1},...,dx_{k},dx_{s}).
\end{align*}
This shows Equation \eqref{eq:conditional expectation gFp, natural filtration}.
\\

\emph{Step 3:} We now show that for the right continuous
enlargement of the filtration
$
\left(\mathcal{F}_{t}\right)_{t\in\mathbb{R}_{+}}:=\left(\mathcal{F}_{t+}^{0}\right)_{t\in\mathbb{R}_{+}}
$
the equation
\[
\mathbb{E}_{\nu}\left[\left.f(\lambda_{t})\right|\mathcal{F}_{s}\right]=P\left(t-s\right)f(\lambda_{s})
\]
holds $\mathbb{P}_{\nu}$- almost surely for $f\in\mathcal{\mathscr{B}}^{\rho}(E)$, $t\geq s\geq0$ and any  probability measure  $\nu\in\mathcal{M}^{\rho}(E)$.  We fix
 $f\in\mathcal{\mathscr{B}}^{\rho}(E)$. It is a standard regularity result that
\begin{equation}
\mathbb{E}_{\nu}\left[\left.f(\lambda_{t})\right|\mathcal{F}_{s}\right]=\underset{r\searrow s}{\lim}\,\mathbb{E}_{\nu}\left[\left.f(\lambda_{t})\right|\mathcal{F}_{r}^{0}\right]\label{eq:proof GFS is Markov process, equation right continuous enlargement}
\end{equation}
holds true $\mathbb{P}_{\nu}$- almost surely for any $t\geq s\geq0$.
Thus, it is sufficient to show 
\[
P\left(t-s\right)f(\lambda_{s})=\underset{r\searrow s}{\lim}\,P\left(t-r\right)f(\lambda_{r})
\]
 $\mathbb{P}_{\nu}$- almost surely for any $t\geq s\geq0$. 
 
We fix $x_{0}\in E$ and  and first show the special case of $\nu=\delta_{x_0}$ and $s=0$. In this case, since the left hand side is
deterministic, it is sufficient to show convergence in law, i.e.~
\begin{equation}
\underset{r\searrow0}{\lim}\,\mathbb{E}_{x_{0}}\left[h\left(P\left(t-r\right)f(\lambda_{r})\right)\right]=h\left(P\left(t\right)f(x_{0})\right).\label{eq:proof GFS is Markov process ,convergence in law}
\end{equation}
for any $h\in C_{b}(\mathbb{R})$.
The maps
$
r \rightarrow P\left(t-r\right)f
$
and $P\left(t-r\right)f \rightarrow h\circ\left(P\left(t-r\right)f\right)$ are
 continuous by Theorem~\ref{thm:generalized Feller semigroups are strongly continuous} 
and Lemma \ref{lem:composition B-rho is continuous} respectively. 
Hence, by strong continuity of $\left(P(t)\right)_{t\in\mathbb{R}_{+}}$, Lemma I.5.2 in \cite{EnN}  yields
\begin{align*}
{\normalcolor } & \underset{r\searrow0}{\lim}\,\left(\mathbb{E}_{x_{0}}\left[h\left(P\left(t-r\right)f(\lambda_{r})\right)\right]-\left(h\circ P\left(t\right)f\right)(x_{0})\right)\\
{\normalcolor }  = &\,\underset{r\searrow0}{\lim}\,\left(P\left(r\right)\left(h\circ P\left(t-r\right)f\right)(x_{0})-\left(h\circ P\left(t\right)f\right)(x_{0})\right)\\
 =& \,P\left(0\right)\left(h\circ P\left(t\right)f\right)(x_{0})-\left(h\circ P\left(t\right)f\right)(x_{0})\\
 =& \,0.
\end{align*}
Therefore, Equation \eqref{eq:proof GFS is Markov process ,convergence in law}
holds and 
$
P\left(t\right)f(\lambda_{0})=\underset{r\searrow0}{\lim}\,P\left(t-r\right)f(\lambda_{r})
$
 in $\mathbb{P}_{x_{0}}$- probability for $t\geq0$. We still need to show the equation for arbitrary $t\geq s\geq0$, i.e.
\[
P\left(t-s\right)f(\lambda_{s})=\underset{r\searrow s}{\lim}\,P\left(t-r\right)f(\lambda_{r})
\]
$\mathbb{P}_{x_{0}}$-almost surely. 
By definition of $\mathscr{B}^{\rho^{\otimes 2}}\left(E\times E\right)$
there exists $\left(f_{n}\right)_{n\in\mathbb{N}}\subset C_{b}\left(E \times E \right)$
such that $f_{n}\left(x,y\right)\rightarrow P\left(t-s\right)f\left(x\right)-P\left(t-r\right)f\left(y\right)$
for any $\left(x,y\right)\in E \times E$. Then by dominated
convergence we have for
$\varepsilon>0$, 
\begin{align*}
 \underset{r\searrow s}{\lim}\,\mathbb{E}_{x_{0}}\left[1_{\left|P\left(t-s\right)f(\lambda_{0})-P\left(t-r\right)f(\lambda_{r-s})\right|>\varepsilon}\circ\theta_{s}\right]
 =\, \underset{r\searrow s}{\lim}\,\underset{n\rightarrow\infty}{\lim}\,\mathbb{E}_{x_{0}}\left[1_{\left|f_{n}(\lambda_{0},\lambda_{r-s})\right|>\varepsilon}\circ\theta_{s}\right].
\end{align*}
The set 
$
O_{n}:=\left|f_{n}(\lambda_{0},\lambda_{r-s})\right|>\varepsilon
$ 
 is open, hence by Corollary \ref{cor:completely regular space, convergence of continuous bounded functions to open set}
there is a sequence $\left(h_{n,m}\right)_{m\in\mathbb{N}}\subset C_{b}\left(E\times E\right)$
such that $\mathbb{P}_{x_{0}}$-almost surely
$
1_{O_{n}} =\underset{m\rightarrow\infty}{\lim}h_{n,m},
$
and $0\leq h_{n,m}\leq1_{O_{n}}.$ By Lemma \ref{lem:B-rho tensor product space, linear isomorphism}
we can approximate $h_{n,m}$ by cylinder functions and by Proposition
\ref{prop:GFS implies existence of semigroup of transition kernels} and the Markov property we obtain
\begin{align*}
  \underset{r\searrow s}{\lim}\,\mathbb{E}_{x_{0}}\left[1_{\left|P\left(t-s\right)f(\lambda_{0})-P\left(t-r\right)f(\lambda_{r-s})\right|>\varepsilon}\circ\theta_{s}\right]
&= \, \underset{r\searrow s}{\lim}\,\underset{n\rightarrow\infty}{\lim}\,\underset{m\rightarrow\infty}{\lim}\,\mathbb{E}_{x_{0}}\left[h_{n,m}\circ\theta_{s}\right]\\
 &=\, \underset{r\searrow s}{\lim}\,\underset{n\rightarrow\infty}{\lim}\,\underset{m\rightarrow\infty}{\lim}\,\mathbb{E}_{x_{0}}\left[\mathbb{E}_{\lambda_{s}}\left[h_{n,m}\right]\right]\\
& \leq\, \underset{r\searrow s}{\lim}\,\underset{n\rightarrow\infty}{\lim}\,\mathbb{E}_{x_{0}}\left[\mathbb{E}_{\lambda_{s}}\left[1_{O_{n}}\right]\right]\\
& =\, \underset{r\searrow s}{\lim}\,\mathbb{E}_{x_{0}}\left[\mathbb{E}_{\lambda_{s}}\left[1_{\left|P\left(t-s\right)f(\lambda_{0})-P\left(t-r\right)f(\lambda_{r-s})\right|>\varepsilon}\right]\right]= 0.
\end{align*}
This yields
$
P\left(t-s\right)f(\lambda_{s})=\underset{r\searrow s}{\lim}\,P\left(t-r\right)f(\lambda_{r})
$
$\mathbb{P}_{x_{0}}$-
almost surely and 
thus, 
\[
\mathbb{E}_{x_{0}}\left[\left.f(\lambda_{t})\right|\mathcal{F}_{s}\right]=P\left(t-s\right)f(\lambda_{s}).
\]

Last, we show the general case for any  probability measure  $\nu\in\mathcal{M}^{\rho}(E)$. By definition of the Baire $\sigma$-algebra the map 
\[
(x_s,x_r)\rightarrow1_{\left|P\left(t-s\right)f(x_{s})-P\left(t-r\right)f(x_{r})\right|>\varepsilon}
\]
 is measurable with respect to $\mathcal{B}_{0}(E)\times \mathcal{B}_{0}(E)$ . Then by definition of the probability measure $\mathbb{P}_{\nu}$ and dominated convergence 
\begin{align*}
 & \underset{r\searrow s}{\lim}\,\mathbb{E}_{\nu}\left[1_{\left|P\left(t-s\right)f(\lambda_{s})-P\left(t-r\right)f(\lambda_{r})\right|>\varepsilon}\right]=\int_{E}\left(\underset{r\searrow s}{\lim}\,\mathbb{E}_{x_{0}}\left[1_{\left|P\left(t -s\right)f(\lambda_{s})-P\left(t-r\right)f(\lambda_{r})\right|>\varepsilon}\right]\right)d\nu(x_{0}) =0
\end{align*}
and thus
\[
P\left(t-s\right)f(\lambda_{s})=\underset{r\searrow0}{\lim}\,P\left(t-r\right)f(\lambda_{r})
\]
$
\mathbb{P}_{\nu}$-almost
surely. 
Finally, it follows easily that Equation \eqref{eq:equation of GFS yields Markov process} holds true also when the filtration
$
\left(\mathcal{F}_{t}\right)_{t\in\mathbb{R}_{+}}
$
is augmented by all the null set in  
$
\left(E^{\mathbb{R}_{+}},\mathcal{B}(E)^{\mathbb{R}_{+}}\right)
$
with respect to $\mathbb{P}_{\nu}$.
\end{proof}

We next state results regarding regularity of the paths of generalized
Feller processes, and close a gap in \cite{CuT}. We remind the reader that since by Theorem \ref{thm:generalized Feller semigroups are strongly continuous}
a generalized Feller semigroup is strongly continuous, we can define its infinitesimal
generator and for $\lambda$ in its resolvent set we write for the resolvent $R(\beta,A):=(\beta-A)^{-1}$. Moreover, recall from Remark \ref{rem: definition P_t rho}
that for strongly continuous semigroups $(P(t))_{t \in \mathbb{R}_+}$ one can always find constants $M\geq1$ and $\omega\in\mathbb{R}$ such that
\begin{align}\label{eq:Momega}
\left\Vert P(t)\right\Vert \leq Me^{\omega t},  \quad \text{all } t\in\mathbb{R}_{+}.
\end{align} 

\begin{theorem}
\label{thm:GFP have cadlag version}Let $\left(P(t)\right)_{t\in\mathbb{R}_{+}}$
be a generalized Feller semigroup on $\mathcal{\mathscr{B}}^{\rho}(E)$
such that for any $t\in\mathbb{R}_{+}$ we have $P(t)1=1$. 
Let $A$ be the generator
of $\left(P(t)\right)_{t\in\mathbb{R}_{+}}$. Let  $\nu\in\mathcal{M}^{\rho}(E)$ be a probability measure and
let $\left(\lambda_{t}\right)_{t\in\mathbb{R}_{+}}$ be the generalized
Feller process from Theorem \ref{thm:GFS induce Markov process} on
the measurable space
$
(E^{\mathbb{R}_{+}},\mathcal{B}(E)^{\mathbb{R}_{+}})
$
with probability measure $\mathbb{P}_{\nu}$, initial  distribution $\nu$, and complete right continuous
filtration $\left(\mathcal{F}_{t}\right)_{t\in\mathbb{R}_{+}}$ as
in Theorem \ref{thm:GFS induce Markov process}. Then the following assertions hold true:
\begin{enumerate}
\item[
(i)] Let $\left(f_{n}\right)_{n\in\mathbb{N}}\subset\mathcal{\mathscr{B}}^{\rho}(E)$ be a countable family 
and $\beta>\omega$ with $\beta\in\mathbb{N}$. Then, for the family of stochastic
processes $\left(Z_{t}^{\beta,n}\right)_{t\in\mathbb{R}_{+}}$ defined
as
\[
Z_{t}^{\beta,n}:=\beta R(\beta,A)f_{n}(\lambda_{t})
\]
there exists a family of stochastic processes
\[
\left(\bar{Z}_{t}^{\beta,n}\right)_{t\in\mathbb{R}_{+}}
\]
with càdlàg paths, such that for each $t\in\mathbb{R}_{+}$
there is a $\mathbb{P}_{\nu}$-null set $\mathcal{N}_{t}\in\mathcal{F}_{0}$ 
such that 
\[
Z_{t}^{\beta,n}=\bar{Z}_{t}^{\beta,n}\text{ on \ensuremath{E^{\mathbb{R}_{+}}\setminus\mathcal{N}_{t}}}
\]
 for all $n\in\mathbb{N}$ and all $\beta>\omega,\,\beta\in\mathbb{N}$. 

\item[(ii)]
Let $\rho$ be Baire measurable. If additionally to the assumptions in
(i) the constant $M$ in \eqref{eq:Momega}
can be chosen to $1$, then
\[
\left(\exp\left(-\omega t\right)\rho(\lambda_{t})\right)_{t\in\mathbb{R}_{+}}
\]
is a supermartingale with $\omega$ as in \eqref{eq:Momega}. If $t\rightarrow P(t)\rho(x)$ is continuous
for $\nu$-almost any $x\in E$, then the supermartingale has a version
with càdlàg paths. In this case, there exists
a family of stochastic processes with càdlàg paths
\[
\left(\left(\overline{f_{n}(\lambda_{t})}\right)_{t\in\mathbb{R}_{+}}\right)_{n\in\mathbb{N}}
\]
 such that for all $t\in\mathbb{R}_{+}$ there is a null set $\mathcal{N}{}_{t}^{'}\in \mathcal{F}_0$
for which 
\[
f_{n}(\lambda_{t})=\overline{f_{n}(\lambda_{t})}\text{ on \ensuremath{E^{\mathbb{R}_{+}}\setminus\mathcal{N}_{t}^{'}}}
\]
 for all $n\in\mathbb{N}$. 

\item[
(iii)] If additionally to the assumptions in (i) and (ii) there exists
a countable family $\left(f_{n}\right)_{n\in\mathbb{N}}\subset\mathcal{\mathscr{B}}^{\rho}(E)$
of sequentially continuous functions, i.e. for any $\left(x_{m}\right)_{m\in\mathbb{N}}\subset E$
with $x_{m}\rightarrow x\in E$ and for any $n\in\mathbb{N}$
\[
f_{n}(x_{m})\rightarrow f_{n}(x),
\]
and if this family separates points, i.e.~for any $y,z\in E$ with $y\neq z$
there exists $l\in\mathbb{N}$ such that
\[
f_{l}(y)\neq f_{l}(z),
\]
 then $\left(\lambda_{t}\right)_{t\in\mathbb{R}_{+}}$ has a version
with càdlàg paths.
\end{enumerate}
\end{theorem}

\begin{proof}
(i)
Recall that $A$ denotes the infinitesimal generator of the generalized Feller semigroup $(P(t))_{t \in \mathbb{R}_+}$.
By the integral representation of the resolvent we have for 
 $\beta>\omega$ with $\omega$ as in \eqref{eq:Momega} and  for all $f\in\mathcal{\mathscr{B}}^{\rho}(E)$
\[
\left(\beta-A\right)^{-1}f:=R(\beta,A)f=\int_{0}^{\infty}e^{-\beta s}P(s)f\,ds.
\]
In order to apply Theorem II.2.9 in \cite{ReY} we first need to find a suitable supermartingale.
 For this purpose, we fix $0\leq f\in\mathcal{\mathscr{B}}^{\rho}(E)$,
and $\beta>\omega$ and we define the stochastic process $\left(Y_{t}^{\beta,f}\right)_{t\in\mathbb{R}_{+}}$
by
\[
Y_{t}^{\beta,f}:=\exp\left(-\beta t\right)R(\beta,A)f(\lambda_{t}).
\]
We show that $\left(Y_{t}^{\beta,f}\right)_{t\in\mathbb{R}_{+}}$
is a supermartingale with respect to $\left(\mathcal{F}_{t}\right)_{t\in\mathbb{R}_{+}}$.
Let $0\leq s\leq t$ and calculate
\begin{align*}
\mathbb{E}_{\nu}\left[\left.Y_{t}^{\beta,f}\right|\mathcal{F}_{s}\right] & =\exp\left(-\beta t\right)\mathbb{E}_{\nu}\left[\left.\int_{0}^{\infty}\exp\left(-\beta u\right)P(u)f(\lambda_{t})du \, \right|\mathcal{F}_{s}\right].
\end{align*}
By definition of the Riemann integral that takes values in the Banach space $\mathcal{\mathscr{B}}^{\rho}(E)$,
positivity of the semigroup $\left(P(t)\right)_{t\in\mathbb{R}_{+}}$,
and monotone convergence for conditional expectations and thanks to $\left(\lambda_{t}\right)_{t\in\mathbb{R}_{+}}$ being
a generalized Feller process we obtain
\begin{align*}
 & \mathbb{E}_{\nu}\left[\left.\int_{0}^{\infty}\exp\left(-\beta u\right)P(u)f(\lambda_{t})du\right|\mathcal{F}_{s}\right]\\
 = \,& \mathbb{E}_{\nu}\left[\left.\underset{n\rightarrow\infty}{\lim}\,\underset{m\rightarrow\infty}{\lim}\left(n/m\right)\cdot{\sum_{i=0}^{m-1}}\exp\left(-\beta in/m\right)P(in/m)f(\lambda_{t})\right|\mathcal{F}_{s}\right]\\
  = \,& \underset{n\rightarrow\infty}{\lim}\,\underset{m\rightarrow\infty}{\lim}\left(n/m\right)\cdot{\sum_{i=0}^{m-1}}\exp\left(-\beta in/m\right)\mathbb{E}_{\nu}\left[\left.P(in/m)f(\lambda_{t})\right|\mathcal{F}_{s}\right]\\
  = \,& \underset{n\rightarrow\infty}{\lim}\,\underset{m\rightarrow\infty}{\lim}\left(n/m\right)\cdot{\sum_{i=0}^{m-1}}\exp\left(-\beta in/m\right)P(in/m+t-s)f(\lambda_{s})\\
  = \,& \exp\left(\beta\left(t-s\right)\right)\underset{n\rightarrow\infty}{\lim}\,\underset{m\rightarrow\infty}{\lim}\left(n/m\right)\cdot{\sum_{i=0}^{m-1}}\exp\left(-\beta\left(in/m+t-s\right)\right)P(in/m+t-s)f(\lambda_{s}).
\end{align*}
Thus, 
\begin{align*}
\mathbb{E}_{\nu}\left[\left.Y_{t}^{\beta,f}\right|\mathcal{F}_{s}\right] & =\exp\left(-\beta s\right)\int_{t-s}^{\infty}\left(\exp\left(-\beta r\right)P(r)f(\lambda_{s})\right)dr\\
 & \leq Y_{s}^{\beta,f},
\end{align*}
where the last inequality followed again from $f\geq0$ and positivity of the
semigroup $\left(P(t)\right)_{t\in\mathbb{R}_{+}}$. Furthermore,
\[
\mathbb{E}_{\nu}\left[Y_{t}^{\beta,f}\right]=\exp\left(-\beta s\right)\int_{t-s}^{\infty}\left(\exp\left(-\beta r\right)\mathbb{E}_{\nu}\left[P(r)f(\lambda_{s})\right]\right)dr,
\]
which implies continuity of 
$
t\rightarrow\mathbb{E}_{\nu}\left[Y_{t}^{\beta,f}\right].
$

We can apply Theorem II.2.9 in \cite{ReY}
to $\left(Y_{t}^{\beta,f}\right)_{t\in\mathbb{R}_{+}}$and obtain
that there is a set $\Omega_{\beta,f}^{'}\in\mathcal{B}(E)^{\mathbb{R}_{+}}$
with $\mathbb{P}_{\nu}\left(\Omega_{\beta,f}^{'}\right)=1$ on
which
\[
\underset{r\searrow t,\,r\in\mathbb{Q}}{\lim}Y_{s}^{\beta,f}
\]
 exists and there is a set $\tilde{\Omega}_{\beta,f}\subset\Omega_{\beta,f}^{'}$
in $\mathcal{B}(E)^{\mathbb{R}_{+}}$ with $\mathbb{P}_{\nu}\left(\tilde{\Omega}_{\beta,f}\right)=1$
such that $\left(\bar{Y}_{t}^{\beta,f}\right)_{t\in\mathbb{R}_{+}}$
defined as 
\[
\bar{Y}_{t}^{\beta,f}:=\begin{cases}
\underset{r\searrow t,\,r\in\mathbb{Q}}{\lim}Y_{r}^{\beta,f} & \text{on }\Omega_{\beta,f}^{'}\\
0 & \text{elsewhere},
\end{cases}
\]
has càdlàg paths on $\tilde{\Omega}_{\beta,f}.$ Moreover, $\left(\bar{Y}_{t}^{\beta,f}\right)_{t\in\mathbb{R}_{+}}$ is a version of $\left(Y_{t}^{\beta,f}\right)_{t\in\mathbb{R}_{+}}$. Therefore, for any $g\in\mathcal{\mathscr{B}}^{\rho}(E)$,
clearly $g^{+},g^{-}\in\mathcal{\mathscr{B}}^{\rho}(E)$, and $g^{+},g^{-}\geq0$
and the process $\left(Y_{t}^{\beta,g}\right)_{t\in\mathbb{R}_{+}}$ has a version with càdlàg paths. The same holds true for any $n\in \mathbb{N}$ and $\beta>\omega$  for the process $\left(Z_{t}^{\beta,n}\right)_{t\in\mathbb{R}_{+}}$.
The statement of part (i) of the theorem then follows directly from the fact that the countable union of null sets is a null set. \\

(ii) Since $M$ is supposed to be $1$, we have $P(t)\rho\leq\exp\left(\omega t\right)\rho$ holds true for some $\omega\in\mathbb{R}$. Moreover, 
 $\rho$ is Baire measurable, Remark \ref{rem:generalized Feller process not Markov} implies 
\begin{align*}
\mathbb{E}_{\nu}\left[\left.\exp\left(-\omega t\right)\rho(\lambda_{t})\right|\mathcal{F}_{s}\right] & =\exp\left(-\omega t\right)P(t-s)\rho(\lambda_{s})\\
 & \leq\exp\left(-\omega s\right)\rho(\lambda_{s}),
\end{align*}
whence $\left(\exp\left(-\omega t\right)\rho(\lambda_{t})\right)_{t\in\mathbb{R}_{+}}$
is a non-negative supermartingale. 
If $t\rightarrow P(t)\rho(x)$ is continuous for $\nu$-almost any $x\in E$, then by dominated convergence and Theorem II.2.9 in \cite{ReY} the stochastic process $\left(\overline{\rho(\lambda_{t})}\right)_{t\in\mathbb{R}_{+}}$
defined by
\[
\overline{\rho(\lambda_{t})}:=\begin{cases}
\rho(\lambda_{t+}) & \text{on }\Omega_{\rho}^{'}\\
0 & \text{elsewhere}
\end{cases}
\]
is a version of $\left(\rho(\lambda_{t})\right)_{t\in\mathbb{R}_{+}}$ and has alomst surely càdlàg paths. Similarly as above $\Omega_{\rho}^{'}$ denotes a set of measure $1$.

Moreover, by the Yosida approximations
\begin{equation}
\underset{\beta\rightarrow\infty}{\lim}\left\Vert \beta R(\beta,A)f_{n}-f_{n}\right\Vert _{\rho}=0.\label{eq:proof cadlag paths, Yosida approximation}
\end{equation}
Hence uniformly in $t\in\mathbb{R}_{+}$ 
\[
\underset{\beta\rightarrow\infty}{\lim}\left|\underset{r\searrow t,\,r\in\mathbb{Q}}{\limsup}\frac{\beta R(\beta,A)f_{n}(\lambda_{r})}{\rho(\lambda_{r})}-\underset{r\searrow t,\,r\in\mathbb{Q}}{\limsup}\frac{f_{n}(\lambda_{r})}{\rho(\lambda_{r})}\right|=0,
\]
and
\[
\underset{\beta\rightarrow\infty}{\lim}\left|\underset{r\searrow t,\,r\in\mathbb{Q}}{\liminf}\frac{\beta R(\beta,A)f_{n}(\lambda_{r})}{\rho(\lambda_{r})}-\underset{r\searrow t,\,r\in\mathbb{Q}}{\liminf}\frac{f_{n}(\lambda_{r})}{\rho(\lambda_{r})}\right|=0.
\]
 Thus, with $Z_{r}^{\beta,n}:=\beta R(\beta,A)f_{n}(\lambda_{r})$ when 
\[
\underset{r\searrow t,\,r\in\mathbb{Q}}{\lim}\frac{Z_{r}^{\beta,n}}{\rho(\lambda_{r})}
\]
exists for any large $\beta\in\mathbb{N}$, then also
\[
\underset{r\searrow t,\,r\in\mathbb{Q}}{\lim}\frac{f_{n}(\lambda_{r})}{\rho(\lambda_{r})}.
\]
 As seen in (i) for any $\beta>\omega,\,\beta\in\mathbb{N}$ there
is a set $\Omega_{\beta,n}^{'}\subset\Omega$ in $\mathcal{B}(E)^{\mathbb{R}_{+}}$\\
 with
$\mathbb{P}_{\nu}\left(\Omega_{\beta,n}^{'}\right)=1$ on which 
\[
\underset{r\searrow t,\,r\in\mathbb{Q}}{\lim}Z_{r}^{\beta,n}
\]
exists for any $t\in\mathbb{R}_{+}$ and admits left limits. Since $\rho>0$ attains its
minimum on $E$, 
\[
\underset{r\searrow t,\,r\in\mathbb{Q}}{\lim}\frac{Z_{r}^{\beta,n}}{\rho(\lambda_{r})}
\]
 exists also
on $\Omega_{\beta,n}^{'}\cap\Omega_{\rho}^{'}$  
and for any $t\in\mathbb{R}_{+}$ we define
\begin{align*}
\overline{\frac{f_{n}}{\rho}(\lambda_{t})} & :=\begin{cases}
\underset{r\searrow t,\,r\in\mathbb{Q}}{\lim}\frac{f_{n}(\lambda_{r})}{\rho(\lambda_{r})} & \text{on }\underset{\beta\in\mathbb{N},\,\beta>\omega}{\bigcap}\Omega_{\beta,n}^{'}\cap\Omega_{\rho}^{'}\\
0 & \text{elsewhere}.
\end{cases}
\end{align*}
Note that
$\left(\overline{\frac{f_{n}}{\rho}(\lambda_{t})}\right)_{t\in\mathbb{\mathbb{R}}_{+}}$
is a version of $\left(\frac{f_{n}(\lambda_{t})}{\rho(\lambda_{t})}\right)_{t\in\mathbb{\mathbb{R}}_{+}}$
since for any $\delta>0$ and $\beta$ large enough on $\underset{\beta\in\mathbb{N},\,\beta>\omega}{\bigcap}\Omega_{\beta,n}^{'}\cap\Omega_{\rho}^{'}${\small{}
\begin{align*}
\left|\overline{\frac{f_{n}}{\rho}(\lambda_{t})}-\frac{f_{n}(\lambda_{t})}{\rho(\lambda_{t})}\right| & \leq\left|\underset{r\searrow t,\,r\in\mathbb{Q}}{\lim}\frac{f_{n}(\lambda_{r})}{\rho(\lambda_{r})}-\underset{r\searrow t,\,r\in\mathbb{Q}}{\lim}\frac{Z_{r}^{\beta,n}}{\rho(\lambda_{r})}\right|\\
 & +\left|\underset{r\searrow t,\,r\in\mathbb{Q}}{\lim}\frac{Z_{r}^{\beta,n}}{\rho(\lambda_{r})}-\frac{Z_{t}^{\beta,n}}{\rho(\lambda_{t})}\right|+\left|\frac{Z_{t}^{\beta,n}}{\rho(\lambda_{t})}-\frac{f_{n}(\lambda_{t})}{\rho(\lambda_{t})}\right|\\
 & \leq\left|\underset{r\searrow t,\,r\in\mathbb{Q}}{\lim}\frac{Z_{r}^{\beta,n}}{\rho(\lambda_{r})}-\frac{Z_{t}^{\beta,n}}{\rho(\lambda_{t})}\right|+2\delta\\
 & =\left|\frac{\bar{Z}_{t}^{\beta,n}}{\overline{\rho(\lambda_{t})}}-\frac{Z_{t}^{\beta,n}}{\rho(\lambda_{t})}\right|+2\delta,
\end{align*}
}and we know that for all $\beta,n\in\mathbb{N}$ , $\beta>\omega$
\[
\left(\frac{\bar{Z}_{t}^{\beta,n}}{\overline{\rho(\lambda_{t})}}\right)_{t\in\mathbb{R}_{+}}
\quad \text{and} \quad 
\left(\frac{Z_{t}^{\beta,n}}{\rho(\lambda_{t})}\right)_{t\in\mathbb{R}_{+}}
\]
 are versions of each other. We obtain for any $\varepsilon>0$ and
for $\beta$ large enough and any $s,t\in\mathbb{R}_{+}$ on $\underset{\beta\in\mathbb{N},\,\beta>\omega}{\bigcap}\Omega_{\beta,n}^{'}\cap\Omega_{\rho}^{'}${\small{}
\begin{align*}
\left|\overline{\frac{f_{n}}{\rho}(\lambda_{t})}-\overline{\frac{f_{n}}{\rho}(\lambda_{s})}\right| & \leq2\varepsilon+\left|\frac{\bar{Z}_{t}^{\beta,n}}{\overline{\rho(\lambda_{t})}}-\frac{\bar{Z}_{s}^{\beta,n}}{\overline{\rho(\lambda_{s})}}\right|.
\end{align*}
}Therefore, also $\left(\overline{\frac{f_{n}}{\rho}(\lambda_{t})}\right)_{t\in\mathbb{\mathbb{R}}_{+}}$
has right continuous paths.  The existence of left limits of $\left(\overline{\frac{f_{n}}{\rho}(\lambda_{t})}\right)_{t\in\mathbb{\mathbb{R}}_{+}}$ follows from the existence of left limits of 
$
\left(\frac{Z_{t}^{\beta,n}}{\rho(\lambda_{t})}\right)_{t\in\mathbb{R}_{+}}
$
 by the same reasoning as before.
Thus, for any $n\in\mathbb{N}$
\[
\left(\overline{\frac{f_{n}}{\rho}(\lambda_{t})}\cdot\overline{\rho(\lambda_{t})}\right)_{t\in\mathbb{R}_{+}}
\]
has càdlàg paths and is a version of 
$
\left(f_{n}(\lambda_{t})\right)_{t\in\mathbb{R}_{+}}.
$
The statement of the theorem then follows from the fact that the countable
union of null sets is a null set. 

Statement (iii) is a direct consequence of (ii).
\end{proof}

\section{Extended Feller processes}

We shall here introduce a new class of stochastic process, called \emph{extended Feller processes,} which are in a different way connected to generalized Feller semigroups. As killing shall play a certain role, 
we recall the cemetery  state $\Delta$ and equip $E\cup\left\{ \Delta\right\} $ with a topology such that
\[
\mathcal{B}(E\cup\left\{ \Delta\right\} )=\sigma\left(\mathcal{B}(E),\left\{ \Delta\right\} \right).
\]
Consistent with the convention made for the cemetery, we
define $\mathcal{\mathscr{B}}^{\rho}(E\cup\left\{ \triangle\right\}) $
as the space of maps $f$ such that $\left.f\right|_{E}\in\mathcal{\mathscr{B}}^{\rho}(E)$
and $f(\Delta)=0.$ The space $C_{0}\left(E\cup\left\{ \triangle\right\} \right)$
is defined in the same fashion. 

\begin{definition}
\label{def:extended Feller process} Let $\rho $ be measurable with respect to the Baire $\sigma$-algebra $\mathcal{B}_0(E)$ and let $\left(P(t)\right)_{t\in\mathbb{R}_{+}}$
be a generalized Feller semigroup on $\mathcal{\mathscr{B}}^{\rho}(E)$,
let $\nu$ be a probability measure on 
\[
\left(E\cup\left\{ \Delta\right\} ,\mathcal{B}(E\cup\left\{ \Delta\right\} )\right)
\]
 and let $\left(\gamma_{t}\right)_{t\in\mathbb{R}_{+}}$ be an adapted
stochastic process on the filtered probability space
\[
\left(\left(E\cup\left\{ \triangle\right\} \right)^{\mathbb{R}_{+}},\mathcal{B}(E\cup\left\{ \Delta\right\} )^{\mathbb{R}_{+}},\left(\mathcal{F}_{t}\right)_{t\in\mathbb{R}_{+}},\mathbb{P}'_{\nu}\right).
\]
If for any $t\geq s\geq0$ and any real-valued map $f$ on $E\cup\left\{ \Delta\right\} $
that is bounded and Baire-measurable
\begin{equation}
\mathbb{E}_{\mathbb{P}'_{\nu}}\left[\left.f(\gamma_{t})\right|\mathcal{F}_{s}\right]=\frac{P\left(t-s\right)\left(f\cdot\rho\right)}{\rho}(\gamma_{s})\label{eq:GFS definition conditional expectation-1}
\end{equation}
holds true $\mathbb{P}'_{\nu}$-almost surely and 
\[
\mathbb{P}'_{\nu}\circ\gamma_{0}^{-1}=\nu,
\]
 then $\left(\gamma_{t}\right)_{t\in\mathbb{R}_{+}}$ is called $\mathit{extended}$
$\mathit{Feller}$ $\mathit{process}$ with respect to $\left(\mathcal{F}_{t}\right)_{t\in\mathbb{R}_{+}}$
and with respect to $\left(P(t)\right)_{t\in\mathbb{R}_{+}}$ with
initial distribution $\nu$.

Similarly as above we make the convention $\mathbb{P}'_{x}:=\mathbb{P}'_{\delta_{x}}$ for any
$x\in E$. 

\end{definition}

The reason for the choice of the name will become clear
in Theorem \ref{thm:extended Feller process is extended Feller process} which essentially states that when $\rho$ continuous (and thus $E$ locally compact) then the restriction of \eqref{eq:GFS definition conditional expectation-1} to $C_0(E)$ is a Feller semigroup.

\begin{remark}
Just like for generalized Feller processes, 
note that \eqref{eq:GFS definition conditional expectation-1} is only required to hold for Baire-measurable functions and  in general does not hold for any positive Borel measurable function. 
However, the latter does hold true on metrizable spaces. 
\end{remark}

In the following theorem we obtain existence of extended Feller processes for contractive generalized Feller semigroups.
\begin{theorem}
\label{thm:contractive gFs on weighted space is Markov process} Let $\rho$ be measurable with respect to the Baire $\sigma$-algebra
$\mathcal{B}_{0}(E)$. Let $\left(P(t)\right)_{t\in\mathbb{R}_{+}}$
be a generalized Feller semigroup on $\mathcal{\mathscr{B}}^{\rho}(E)$
such that for all $t\in\mathbb{R}_{+}$
\[
\left\Vert P(t)\right\Vert _{L(\mathcal{\mathscr{B}}^{\rho}(E))}\leq1.
\]
 Then for any probability measure $\nu$ on
\[
\left(E\cup\left\{ \Delta\right\} ,\mathcal{B}(E\cup\left\{ \Delta\right\} )\right)
\]
 there exists a probability measure $\mathbb{P}'_{\nu}$ on the measurable
space 
\[
\left(\left(E\cup\left\{ \triangle\right\} \right)^{\mathbb{R}_{+}},\mathcal{B}(E\cup\left\{ \Delta\right\} )^{\mathbb{R}_{+}}\right)
\]
 such that for the canonical process $\left(\gamma_{t}\right)_{t\in\mathbb{R}_{+}}$
and the natural filtration $\left(\mathcal{F}_{t}^{0}\right)_{t\in\mathbb{R}_{+}}$
for any $t\geq s\geq0$ and any real-valued map $f$ on $E\cup\left\{ \Delta\right\} $
that is bounded and Baire-measurable
\begin{equation}
\mathbb{E}_{\mathbb{P}'_{\nu}}\left[\left.f(\gamma_{t})\right|\mathcal{F}_{s}^{0}\right]=\frac{P\left(t-s\right)\left(f\cdot\rho\right)}{\rho}(\gamma_{s})\label{eq:contractive GFS yields Markov process}
\end{equation}
holds true $\mathbb{P}'_{\nu}$ - almost surely and 
\[
\mathbb{P}'_{\nu}\circ\gamma_{\text{0}}^{-1}=\nu.
\]
If $f$ is such that $f\cdot\rho\in\mathcal{\mathscr{B}}^{\rho}(E\cup\left\{ \triangle\right\} )$
then Equation \eqref{eq:contractive GFS yields Markov process} holds
true also for the complete right continuous extension of the filtration. 
\end{theorem}

\begin{proof}
We fix some probability measure $\nu$ on $\left(E\cup\left\{ \Delta\right\} ,\mathcal{B}(E\cup\left\{ \Delta\right\} )\right)$ and divide the proof into three steps similarly as in the proof of  Theorem~\ref{thm:GFS induce Markov process}. We also apply the  notation introduced in the beginning of Section \ref{sec:genfeller}.

In the first step we define a family of sub-probability measures on the space
\[
\left(\left(E\cup\left\{ \Delta\right\} \right)^{J},\left(\mathcal{B}\left(E\cup\left\{ \Delta\right\} \right)\right)^{J}\right)_{J\subset\mathbb{R}_{+},\text{ finite}}.
\]
and show in the second step that this family is projective.  In the third step
we can then apply the generalized Kolmogorov extension theorem (Theorem 15.26 in \cite{AlB}) to obtain the
statement of the theorem. \\

\emph{Step 1:} On
$
\left(\left(E\cup\left\{ \Delta\right\} \right)^{J},\left(\mathcal{B}\left(E\cup\left\{ \Delta\right\} \right)\right)^{J}\right)_{J\subset\mathbb{R}_{+},\text{ finite}},
$ define a family of probability
measures via
$$
\left(\left(q_{\nu}^{J}\right)\right)_{J\subset\mathbb{R}_{+},\text{ finite}}.
$$
We fix some $s\in\mathbb{R}_{+}$ and by Lemma \ref{lem:GFS induces family of measures}
we find $p(s)(x,\cdot)\in\mathcal{M}^{\rho}(E)$ such that
\[
P(s)f(x)=\int_{E}f(y)p(s)(x,dy)\,\text{for all }x\in E.
\]
By Remark \ref{rem: definition P_t rho}
we have $P(s)\rho(x) \leq\rho(x)$ 
for all $x\in E$
and we define  the measures $q(s)(x,\cdot)$ via
\[
q(s)(x,A):=\int_{E}1_{A}(y)\frac{\rho(y)}{\rho(x)}p(s)(x,dy)\text{ for }A\in\mathcal{B}(E).
\]
Consequently, $q(s)(x,E)\leq1$. For any $s\in\mathbb{R}_{+}$ for
any $x\in E$ we define the measures $\tilde{q}(s)(x,\cdot)$ on $E\cup\left\{ \Delta\right\} $
by
\[
\left.\tilde{q}(s)(x,\cdot)\right|_{\mathcal{B}(E)}:=q(s)(x,\cdot)
\]
and 
\[
\tilde{q}(s)(x,\left\{ \Delta\right\} ):=1-q(s)(x,E).
\]
 Furthermore
\[
\tilde{q}(s)(\Delta,\left\{ \Delta\right\} ):=1
\]
for any $s\in\mathbb{R}_{+}$. Thanks to Proposition \ref{prop:GFS implies existence of semigroup of transition kernels},
we can define the semigroup $\left(Q(t)\right)_{t\in\mathbb{R}_{+}}$
on the space 
 of bounded Baire measurable maps 
by 
\[
Q(t)f(x)=\int_{E}f(y)\tilde{q}(t)(x,dy).
\]
For any finite $J:=\left\{ r_{1},...,r_{n}\right\} \subset\mathbb{R}_{+}$
by Lemma \ref{lem:B-rho tensor product space, linear isomorphism}\textcolor{red}{{}
}there is a unique continuous map
\[
j_{J,\nu}:\,\mathcal{\mathscr{B}}^{\rho^{\otimes J}}(\left(E\cup\left\{ \Delta\right\} \right)^{J})\rightarrow\mathbb{R}
\]
such that {\footnotesize
\begin{align*}
f_{r_{1}}\cdot...\cdot f_{r_{n}} & \rightarrow\int_{E}Q(r_{1})\left(\left(\frac{f_{r_{1}}}{\rho_{r_{1}}}\right)\cdot...\cdot\left(Q(r_{n-1}-r_{n-2})\left(\frac{f_{r_{n-1}}}{\rho_{r_{n-1}}}\right)\cdot\left(Q(r_{n}-r_{n-1})\left(\frac{f_{r_{n}}}{\rho_{r_{n}}}\right)\right)\right)\right)(x_{0})\nu(dx_{0})
\end{align*}
}
for any $f\in\mathcal{\mathscr{B}}^{\rho^{\otimes J}}\left(\left(E\cup\left\{ \Delta\right\} \right)^{J}\right)$
given by 
\[
f(x_{J}):=f_{r_1}(x_{r_1})\cdot \ldots \cdot f_{r_n}(x_{r_n}).
\]
By Theorem \ref{thm:Riesz-representation-for B rho} there exists
a unique finite positive Radon measure
\[
\mu_{\nu}^{J}\in\mathcal{M}^{\rho^{\otimes J}}\left(\left(E\cup\left\{ \Delta\right\} \right)^{J}\right),
\]
such that for any $f\in\mathcal{\mathscr{B}}^{\rho^{\otimes J}}\left(\left(E\cup\left\{ \Delta\right\} \right)^{J}\right).$
\[
j_{J,\nu}(f)=\int_{E^{J}}f(x_{J})\mu_{\nu}^{J}(dx_{J}),
\quad \text{and} \quad 
\int_{E^{J}}\rho^{\otimes J}(x_{J})\mu_{\nu}^{J}(dx_{J})=1.
\]
We now define the family of probability measures
$
\left(q_{\nu}^{J}\right)_{J\subset\mathbb{R}_{+},\text{ finite}}
$
 on 
\[
\left(\left(E\cup\left\{ \Delta\right\} \right)^{J},\mathcal{B}(\left(E\cup\left\{ \Delta\right\} \right)^{J})\right)_{J\subset\mathbb{R}_{+},\text{ finite}}
\]
by 
\begin{align*}
\mathcal{B}(E^{J}) & \rightarrow\left[0,1\right]\\
A & \rightarrow\int_{A}\rho^{\otimes J}(x_{J})\mu_{\nu}^{J}(dx_{J}).
\end{align*}
We observe the non-obvious fact that for any finite $J\subset\mathbb{R}_{+}$
the measure $q_{\nu}^{J}$ is a Radon measure (since the space $E^{J}$
is non necessarily Polish). To see this, fix a finite $\tilde{J}\subset\mathbb{R}_{+}$.
Let $A\in\mathcal{B}\left(E^{\tilde{J}}\right)$ and $\varepsilon>0$
be arbitrary. Then by $\underset{R>0}{\bigcup}K_{R}=E$ there exists
$R_{\varepsilon}>0$ such that 
\[
q_{\nu}^{\tilde{J}}\left(E^{\tilde{J}}\setminus\left(K_{R_{\varepsilon}}\right)^{\tilde{J}}\right)<\frac{\varepsilon}{2}.
\]
Since $\mu_{\nu}^{\tilde{J}}$ is a Radon measure there exists $K\subset A\cap\left(K_{R_{\varepsilon}}\right)^{\tilde{J}}$
such that 
\[
\mu_{\nu}^{\tilde{J}}\left(A\cap\left(K_{R_{\varepsilon}}\right)^{\tilde{J}}\setminus K\right)<\frac{\varepsilon}{2R_{\varepsilon}}.
\]
 Thus, 
\begin{align*}
q_{\nu}^{\tilde{J}}\left(A\setminus K\right) & \leq q_{\nu}^{\tilde{J}}\left(E^{\tilde{J}}\setminus\left(K_{R_{\varepsilon}}\right)^{\tilde{J}}\right)+q_{\nu}^{\tilde{J}}\left(A\cap\left(K_{R_{\varepsilon}}\right)^{\tilde{J}}\setminus K\right)\leq\frac{\varepsilon}{2}+\frac{\varepsilon}{2},
\end{align*}
 and the probability measure $q_{\nu}^{\tilde{J}}$ is inner regular,
hence a Radon measure. \\

\emph{Step 2:} We now show that the family $
\left(q_{\nu}^{J}\right)_{J\subset\mathbb{R}_{+},\text{ finite}}
$
 is projective. To this end, it is sufficient to show for any finite
$J:=\left\{ r_{1},...,r_{n}\right\} \subset\mathbb{R}_{+}$ and $j\in\left\{ 1,...,n\right\} $
for any $A_{i}\in\mathcal{B}(E\cup\left\{ \Delta\right\} ){}^{r_{i}}$,
$i\in\left\{ 1,...,n\right\} \setminus\left\{ j\right\} $
\[
q_{\nu}^{J}(A_{1}\times...\times A_{j-1}\times E_{j}\times A_{j+1}...\times A_{n})=q_{\nu}^{J\setminus\left\{ r_{j}\right\} }(A_{1}\times...\times A_{j-1}\times A_{j+1}...\times A_{n}).
\]
We observe that by Corollary \ref{cor:completely regular space, convergence of continuous bounded functions to open set} indicator functions of open sets can be approximated by continuous bounded maps almost surely with respect to finitely many measures. Hence, any set in the Borel $\sigma$-algebra
can be approximated by continuous bounded maps almost surely with respect to finitely many measures. 
Approximating $A_{1}\times...\times A_{j-1}\times E_{j}\times A_{j+1}...\times A_{n}$  as the product of $n$ indicator function in such a way, projectivity of the
family
$
\left(q_{\nu}^{J}\right)_{J\subset\mathbb{R}_{+},\text{ finite}}
$
follows by dominated convergence from the definition of the family $
\left(\mu_{\nu}^{J}\right)_{J\subset\mathbb{R}_{+},\text{ finite}}
$
on the cylinder functions. 

\emph{Step 3:} As in the proof of Theorem \ref{thm:GFS induce Markov process}
one can easily show that 
\begin{align*}
\mathcal{C}: & =\left\{ C:\,C\text{ compact},\,C\subset K_{R}\text{ for }R\geq0\right\} 
 \cup\left\{ C\cup\left\{ \Delta\right\} :\,C\text{ compact},\,C\subset K_{R}\text{ for }R\geq0\right\} 
\end{align*}
is a compact class in $E\cup\left\{ \Delta\right\} $ and that for
each $t\in\mathbb{R}_{+}$ and $A\in\mathcal{B}(E\cup\left\{ \Delta\right\} )$
\[
\left(q_{\nu}^{\left\{ t\right\} }\right)(A)=\sup\left\{ \left(q_{\nu}^{\left\{ t\right\} }\right)(C):\,C\subset A\text{\,and }C\in\mathcal{C}\right\}.
\]
Therefore we can apply Theorem 15.26 in \cite{AlB}.
This yields a measure $\mathbb{P}'_{\nu}$ on
\[
\left(\left(E\cup\left\{ \Delta\right\} \right)^{\mathbb{\mathbb{R}}_{+}},\left(\mathcal{B}(E\cup\left\{ \Delta\right\} )\right){}^{\mathbb{\mathbb{R}}_{+}}\right)
\]
such that its projections are the family $\left(q_{\nu}^{J}\right)_{J\subset\mathbb{R}_{+},\text{ finite}}$.
Furthermore, 
\[
\mathbb{P}'_{\nu}\circ\gamma_{\text{0}}^{-1}=\nu,
\]
by definition of $\mathbb{P}'_{\nu}$ via the functional $j_{\left\{ 0\right\} ,\nu}$.

Equation \eqref{eq:contractive GFS yields Markov process} follows from
the fact that we can approximate bounded Baire-measurable functions  by continuous bounded function 
and the same reasoning as in the proof of Theorem \ref{thm:GFS induce Markov process}. 

Finally, for $f$ such that $f\cdot\rho\in\mathcal{\mathscr{B}}^{\rho}(E\cup\left\{ \triangle\right\} )$
the statement concerning the 
right continuous extension of the filtration follows as in the proof of Theorem
\ref{thm:GFS induce Markov process}. 
\end{proof}

The following corollary extends the above statement from contraction semigroups to quasi-contraction semigroups.

\begin{corollary}
\label{cor:M smaller 1 gFs on weighted space is Markov process} Let
$\rho$ be Baire measurable\textcolor{red}{{} }and let $\left(P(t)\right)_{t\in\mathbb{R}_{+}}$
be a generalized Feller semigroup on $\mathcal{\mathscr{B}}^{\rho}(E)$
such that for some $\omega\in\mathbb{R}$ and all $t\in\mathbb{R}_{+}$
\[
\left\Vert P(t)\right\Vert _{L(\mathcal{\mathscr{B}}^{\rho}(E))}\leq e^{\omega t}.
\]
 Then for any probability measure $\nu$ on $\left(E\cup\left\{ \triangle\right\} ,\mathcal{B}(E\cup\left\{ \Delta\right\} )\right)$
there exists a probability measure $\mathbb{P}'_{\nu}$ on 
\[
\left(\left(E\cup\left\{ \triangle\right\} \right)^{\mathbb{R}_{+}},\mathcal{B}(E\cup\left\{ \Delta\right\} ){}^{\mathbb{R}_{+}}\right)
\]
 such that for the canonical process $\left(\gamma_{t}\right)_{t\in\mathbb{R}_{+}}$\textcolor{red}{{}
}for any $t\geq s\geq0$ and any real-valued map $f$ on $E\cup\left\{ \Delta\right\} $
that is bounded and Baire-measurable
\begin{equation}
\mathbb{E}_{\mathbb{P}'_{\nu}}\left[\left.f(\gamma_{t})\right|\mathcal{F}_{s}^{0}\right]=\frac{e^{-\omega t}P\left(t-s\right)\left(f\cdot\rho\right)}{\rho}(\gamma_{s})\label{eq:extended Feller process, existence case M=00003D1}
\end{equation}
holds true $\mathbb{P}'_{\nu}$-almost surely (where $\left(\mathcal{F}_{t}^{0}\right)_{t\in\mathbb{R}_{+}}$
is the natural filtration) and 
\[
\mathbb{P}'_{\nu}\circ\gamma_{\text{0}}^{-1}=\nu.
\]
If $f$ is such that $f\cdot\rho\in\mathcal{\mathscr{B}}^{\rho}(E\cup\left\{ \triangle\right\} )$, then Equation \eqref{eq:extended Feller process, existence case M=00003D1}
holds true also for the right continuous extension of the filtration.
\end{corollary}

\begin{proof}
Define the rescaled semigroup $\left(S(t)\right)_{t\in\mathbb{R}_{+}}$ for any $t\in\mathbb{R}_{+}$
by 
\[
S(t):=e^{-\omega t}P(t).
\]
 Then clearly $\left(S(t)\right)_{t\in\mathbb{R}_{+}}$ is also a
generalized Feller semigroup and satisfies the conditions of Theorem
\ref{thm:contractive gFs on weighted space is Markov process}. This
directly yields the statement of this corollary. 
\end{proof}

\begin{remark}
Note that for
a monotone concave function $\rho$
and a supermartingale $\left(\lambda_{t}\right)_{t\in\mathbb{R}_{+}}$
we have by Jensen's inequality
\[
\mathbb{E}_{x}\left[\rho\left(\lambda_{t}\right)\right]\leq\rho\left(\mathbb{E}_{x}\left[\left(\lambda_{t}\right)\right]\right)\leq\rho\left(\mathbb{E}_{x}\left[\left(\lambda_{0}\right)\right]\right)=\rho(x),
\]
hence the condition 
\[
\left\Vert P(t)\right\Vert _{L(\mathcal{B}^{\rho}(Y))}\leq1
\]
 holds true for $\left(P(t)\right)_{t\in\mathbb{R}_{+}}$ defined
by $P(t):f\rightarrow\mathbb{E}_{x}\left[f\left(\lambda_{t}\right)\right]$
for any $t\in\mathbb{R}_{+}$.
\end{remark}

Next, we want to compare the laws induced by the corresponding canonical
processes in Theorem \ref{thm:GFS induce Markov process} and Theorem
\ref{thm:contractive gFs on weighted space is Markov process}. We here work with the finite time interval $I=[0,T]$ instead of $\mathbb{R}_+$.

\begin{proposition} 
\label{prop:gFp and gamma process comparison} Let $T>0$ and let $I=\left[0,T\right]$ and let $\rho$ be Baire measurable. Let $\left(P(t)\right)_{t\in I}$
be a generalized Feller semigroup on $\mathcal{\mathscr{B}}^{\rho}(E)$
such that both the conditions of Theorem \ref{thm:GFS induce Markov process}
and of Theorem \ref{thm:contractive gFs on weighted space is Markov process}
are fulfilled and let  $\mathbb{\mathbb{P}{}_{\nu}}$ and $\mathbb{P}'_{\nu}$ be the respective probability measures for an initial distribution $\nu$. Let $\left(\lambda_{t}\right)_{t\in I}$ and $\left(\gamma_{t}\right)_{t\in I}$ be the respective canonical processes.
Then for $A\in\mathcal{B}(E)^{\left[0,T\right]}$
\[
\mathbb{P}'_{\nu}\left[A\right]=\mathbb{E}_{\mathbb{P}'_{\nu}}\left[1_{A}\right]=\mathbb{E}_{\mathbb{P}{}_{\nu}}\left[1_{A}\cdot\frac{\rho(\lambda_{T})}{\rho(\lambda_{0})}\right],
\]
 and 
\[
\mathbb{E}_{\mathbb{P}'_{\nu}}\left[1_{A}\cdot\frac{\rho(\gamma_{0})}{\rho(\gamma_{T})}\right]=\mathbb{E}_{\mathbb{P}{}_{\nu}}\left[1_{A}\right]=\mathbb{P}{}_{\nu}\left[A\right]
\]
hold true, hence $\left.\mathbb{P}'_{\nu}\right|_{\mathcal{B}(E)^{\left[0,T\right]}}$
and $\mathbb{P}{}_{\nu}$ are equivalent measures. 
\end{proposition}

\begin{proof}
 Let
\[
\left(p(t)\left(x,\cdot\right)\right)_{t\in I,x\text{\ensuremath{\in}}E}
\]
 be the family of probability measures such that for all $x\in E$, $t\in\mathbb{R}_{+}$ and $f\in\mathcal{\mathscr{B}}^{\rho}(E)$
\[
P(t)f(x)=\int_{E}f(y)p(t)(x,dy).
\]

Just like $\left.\mathbb{P}'_{\nu}\right|_{\mathcal{B}(E)^{\left[0,T\right]}}$
the map 
\begin{align*}
\mathbb{Q}_{\nu}: \mathcal{B}(E)^{\left[0,T\right]} & \rightarrow \mathbb{R}_+\\
A &\rightarrow\mathbb{E}_{\mathbb{P}{}_{\nu}}\left[1_{A}\cdot\frac{\rho(\lambda_{T})}{\rho(\lambda_{0})}\right]
\end{align*}
is a measure on 
\[
\left(E^{\left[0,T\right]},\mathcal{B}(E)^{\left[0,T\right]}\right).
\]
 Its mass is given by
\begin{align*}
\mathbb{E}_{\nu}\left[1_{E^{\left[0,T\right]}}\frac{\rho(\lambda_{T})}{\rho(\lambda_{0})}\right] & =\int_{E}\left(\int_{E}\rho(x_{T})p(T)(x_{0},dx_{T})\right)\frac{1}{\rho(x_{0})}d\nu(x_{0})\\
 & =\mathbb{E}'_{\nu}(1_{E}(\gamma_{T}))\\
 & =\mathbb{P}'_{\nu}(E^{\left[0,T\right]}).
\end{align*}
It is enough to show that $\mathbb{Q}_{\nu}$ and $\left.\mathbb{P}'_{\nu}\right|_{\mathcal{B}(E)^{\left[0,T\right]}}$
coincide on an intersection stable generator of $\mathcal{B}(E)^{\left[0,T\right]}$.
This is indeed the case as for any $x_{0}\in E$ , $n\in\mathbb{N}$
, $\left\{ t_{1},...,t_{n}\right\} \subset\left[0,T\right]$, and $A_{t_{1}},...,A_{t_{n}}\in\mathcal{B}(E)$
one can approximate the indicator functions $1_{A_{t_{1}}}$, ...,$1_{A_{t_{n}}}$
and $\rho$ using Corollary \ref{cor:completely regular space, convergence of continuous bounded functions to open set}
by non-negative continuous bounded functions that converge almost
surely with respect to $p_{x_{0}}^{\left\{ 0,t_{1},...,t_{n},T\right\} }$,
as defined in the proof of Theorem \ref{thm:GFS induce Markov process}
and $q_{x_{0}}^{\left\{ 0,t_{1},...,t_{n},T\right\} }$, as defined
in the proof of Theorem \ref{thm:contractive gFs on weighted space is Markov process}.
Then one obtains by dominated convergence, and the definition of the
measures $\mathbb{P}'_{\nu}$ and $\mathbb{P}{}_{\nu}$ 
\begin{align*}
 & \mathbb{E}'_{\nu}\left[1_{E^{\left[0,T\right]}}\cdot1_{A_{t_{1}}}(\gamma_{t_{1}})\cdot...\cdot1_{A_{t_{n}}}(\gamma_{t_{n}})\right]
 =\, \mathbb{E}'_{\nu}\left[1_{E}(\gamma{}_{T})\cdot1_{A_{t_{1}}}(\gamma_{t_{1}})\cdot...\cdot1_{A_{t_{n}}}(\gamma_{t_{n}})\right]\\
& =\int_{E}\int_{A_{t_{1}}}...\int_{A_{t_{n}}}\left(\int_{E}\rho(x_{T})p(T-t_{n})(x_{t_{n}},dx_{T})\right)p(t_{n}-t_{n-1})(x_{t_{n-1}},dx_{t_{n}})...\frac{1}{\rho(x_{0})}d\nu(x_{0})\\
 &=\, \mathbb{E}_{\mathbb{P}{}_{\nu}}\left[1_{E}(\gamma{}_{T})\cdot1_{A_{t_{1}}}(\gamma_{t_{1}})\cdot...\cdot1_{A_{t_{n}}}(\gamma_{t_{n}})\frac{\rho\circ\lambda_{T}}{\rho\circ\lambda_{0}}\right].
\end{align*}
\end{proof}

For $I=\mathbb{R}_{+}$
or $I=\left[ 0,T\right] \subset \mathbb{R}_{+}$ let $\left(P(t)\right)_{t\in\mathbb{R}_{+}}$ be a generalized Feller
semigroup and $\left(\lambda_{t}^{x}\right)_{t\in I,\,x\in E}$ the family of generalized Feller processes constructed according to Theorem~\ref{thm:GFS induce Markov process}, such that $\mathbb{P}_x\left(\lambda_{0}^{x}=x\right)=1$
and 
\[
\mathbb{E}_x\left[\left.f(\lambda^x_{t})\right|\mathcal{F}_{s}\right]=P\left(t-s\right)f(\lambda^x_{s})\label{eq:GFS yields Markov process} 
\]
holds true. 
Similarly we consider a family of extended Feller processes $\left(\gamma_{t}^{x}\right)_{t\in I,\,x\in E}$ such that $\mathbb{P}^{'}_x\left(\gamma_{0}^{x}=x\right)=1$
and 
\[
\mathbb{E}^{'}_x\left[\left.f(\gamma^x_{t})\right|\mathcal{F}_{s}\right]=P\left(t-s\right)f(\gamma^x_{s})
\]
holds true. For convenience we here work 
on one probability space $(\Omega,\mathcal{F},\mathbb{P})$ for the whole family of $(\lambda_{t}^{x})_{t\in I,\,x\in E}$ 
and on $(\Omega,\mathcal{F},\mathbb{P}^{'})$ for $(\gamma_{t}^{x})_{t\in I,\,x\in E}$ respectively. 
The goal is to make the measure change between $\mathbb{P}$ and $\mathbb{P}^{'}$ in the case of diffusion processes precise. This is subject of the following proposition.

\begin{proposition}
\label{prop:measure change gFs for Ito diffusions} Let $I=\mathbb{R}_{+}$
or $I=\left[ 0,T\right] \subset \mathbb{R}_{+}$ and let
$\left(\lambda_{t}^{x}\right)_{t\in I,x\in \mathbb{R}^{d}}$ be a family of Ito-diffusions on the filtered probability space $ \left(\Omega,\mathcal{F},\left(\mathcal{F}_{t}\right)_{t\in I},\mathbb{P}\right) $ with state space $\mathbb{R}^{d}$, $d\in\mathbb{N}$, with drift $\mu$ and diffusion matrix $\sigma$  such that $\lambda_{0}^{x}=x$
$\mathbb{P}$-a.s for any $x\in \mathbb{R}^{d}$, i.e. we consider a solution to the following SDE
\[
d\lambda_t^x= \mu(\lambda_t) dt + \sigma(\lambda_t^x) dW_t, \quad \lambda_0^x= x,
\]
where $W$ is a $d$-dimensional Brownian motion, $\mu: \mathbb{R}^d \to \mathbb{R}^d$ and $\sigma: \mathbb{R}^d \to \mathbb{R}^{d \times d}$ are continuous functions.
Let $\rho\in C^{2}(\mathbb{R}^{d})$ and suppose that
 \begin{align}\label{eq:martingale}
\left(\int_0^t  \nabla_x^{\top} \rho(\lambda_s^x) \sigma(\lambda_s^x) dW_s\right)_{t \in I} \text{ is a true martingale.}
 \end{align}
Moreover, let
$\left(P(t)\right)_{t\in I}$ be the contractive semigroup  on $\mathcal{\mathscr{B}}^{\rho}(\mathbb{R}^{d})$ defined by
\begin{eqnarray}
P(t)f(x):=\mathbb{E}\left[f(\lambda_{t}^{x})\right] & & \text{for  } f\in\mathcal{\mathscr{B}}^{\rho}(\mathbb{R}^{d}).
\end{eqnarray}
Moreover, let $\mathbb{P}'$ be another probability
measure on
$
\left(\Omega,\mathcal{F},\left(\mathcal{F}_{t}\right)_{t\in I}\right)
$
such that for the family of Markov processes $\left(\gamma_{t}^{x}\right)_{t\in I,x\in \mathbb{R}^{d}}$
with $\gamma_{0}^{x}=x$ $\mathbb{P}'$-a.s. for any $x\in \mathbb{R}^{d}$, $t\in I$ and any real-valued map
$f$ on $\mathbb{R}^{d}\cup\left\{ \Delta\right\} $ that is bounded and Baire-measurable  \\
\[
\mathbb{E}'\left[f(\gamma_{t}^{x})\right]=\frac{P\left(t\right)\left(f\cdot\rho\right)}{\rho}(x)
\]
holds true.\\
Then the drift $\mu'=\left(\mu_{1}',...,\mu_{d}'\right)$ of $\left(\gamma_{t}\right)_{t\in I}$
with respect to $\mathbb{P}'$ is given by 
\[
\mu_{i}'=\mu_{i}+{\sum}_{j=1}^{d}\frac{d\rho}{dx_{j}}(x)\frac{\sigma_{ij}^{2}(x)}{\rho(x)},
\]
the diffusion matrix is $\sigma'=\sigma$, and the killing rate $c'<0$
is
\begin{align*}
c'(x) & =\left({\sum}_{i=1}^{d}\frac{d\rho}{dx_{i}}(x)\mu_{i}(x)+\frac{1}{2}{\sum}_{j=1}^{d}{\sum}_{i=1}^{d}\frac{d^{2}\rho}{dx_{i}dx_{j}}(x)\sigma_{ij}^{2}(x)\right)\frac{1}{\rho(x)}.
\end{align*}
\end{proposition}

\begin{proof}
By Ito's formula and the Assumption~\eqref{eq:martingale} as well as
$\left\Vert P(t)\right\Vert _{L(\mathcal{\mathscr{B}}^{\rho}(E))}\leq1$,
we have for any $x\in E$ and $t\in I$
\begin{align*}
\mathbb{E}\left[\rho(\lambda_{t}^{x})\right] & =\rho(x)+\int_{0}^{t}\left({\sum}_{i=1}^{d}\frac{d\rho}{dx_{i}}(x)\mu_{i}(x)+\frac{1}{2}{\sum}_{j=1}^{d}{\sum}_{i=1}^{d}\frac{d^{2}\rho}{dx_{i}dx_{j}}(x)\sigma_{ij}^{2}(x)\right)ds\\
 & \leq\rho(x).
\end{align*}

Hence, 
\[
{\sum}_{i=1}^{d}\frac{d\rho}{dx_{i}}(x)\mu_{i}(x)+\frac{1}{2}{\sum}_{j=1}^{d}{\sum}_{i=1}^{d}\frac{d^{2}\rho}{dx_{i}dx_{j}}(x)\sigma_{ij}^{2}(x)\leq0
\]
which yields the sign of the killing rate $c'$.
 Furthermore, for any $x\in E$ and $f\in C_{c}^{2}(E)$ the infinitesimal
generator $\mathscr{A}'$ of $\left(\gamma_{t}\right)_{t\in I}$ is
given by 
\begin{align*}
& \mathscr{A}'f(x)\\
= \,& \underset{t\searrow0}{\lim}\,\frac{\mathbb{E}'\left[f(\gamma_{t}^{x})\right]-f(x)}{t}\\
 = \,& \underset{t\searrow0}{\lim}\,\frac{1}{t}\left(\frac{P\left(t\right)\left(f\cdot\rho\right)}{\rho}(x)-f(x)\right)\\
 = \,&\underset{t\searrow0}{\lim}\,\frac{1}{t}\left(\frac{\mathbb{E}\left[(f\cdot\rho)(\lambda_{t}^{x})\right]}{\rho(x)}-f(x)\right)\\
 = \,&\underset{t\searrow0}{\lim}\,\frac{1}{t}\left(\frac{(f\cdot\rho)(x)+\int_{0}^{t}{\sum}_{i=1}^{d}\left(\frac{d(f\cdot\rho)}{dx_{i}}\mu_{i}(\lambda_{s}^x)\right)ds+\frac{1}{2}\int_{0}^{t}{\sum}_{j=1}^{d}{\sum}_{i=1}^{d}\left(\frac{d^{2}(f\cdot\rho)}{dx_{i}dx_{j}}\sigma_{ij}^{2}(\lambda_{s}^x)\right)ds}{\rho(x)}-f(x)\right)\\
 = \,&{\sum}_{i=1}^{d}\frac{d(f\cdot\rho)}{dx_{i}}(x)\frac{\mu_{i}(x)}{\rho(x)}+\frac{1}{2}\,{\sum}_{j=1}^{d}{\sum}_{i=1}^{d}\frac{d^{2}(f\cdot\rho)}{dx_{i}dx_{j}}(x)\frac{\sigma_{ij}^{2}(x)}{\rho(x)}.
\end{align*}
Applying the product rule and bringing $\mathscr{A}'f(x)$ into the following form
\[
\mathscr{A}'f(x) = \sum_{i=1}^d \frac{df}{dx_i} \mu_i'(x) + \frac{1}{2}\sum_{i,j=1}^d \frac{d^2f}{dx_{ij}} \sigma_{ij}'(x) + f(x)c'(x)
\]
yields the assertion of the proposition.
\end{proof}
\
As we will see in the next section, extended Feller processes are a generalization of Feller processes to more general state spaces. Thus, it is not surprising that we obtain a regularity result for their paths.
In order to state this result we consider
the following space that will also be needed in the next section:

\begin{definition}\label{ellrho}
\label{def:l-rho}
For $\rho$ being measurable with respect to the Baire $\sigma$-algebra $\mathcal{B}_0(E)$, we define
\[
\ell^{\rho}(E):=\left\{ \frac{f}{\rho}:\,f\in\mathcal{\mathscr{B}}^{\rho}(E)\right\} .
\]
\end{definition}

\begin{remark}
Note that $\ell^{\rho}(E)$ is a 
Banach space
with respect to $\left\Vert \cdot\right\Vert _{\infty}.$
\end{remark}

According to Theorem \ref{thm:contractive gFs on weighted space is Markov process}
the semigroup $\left(Q(t)\right)_{t\in\mathbb{R}_{+}}$ on $\ell^{\rho}(E)$ defined by $$Q(t)f:=\frac{P(t)\left(f\cdot\rho\right)}{\rho}$$ is strongly continuous, contractive and positive.

In order to show regularity of the paths of $f(\gamma_{t})$ for any $f\in\ell^{\rho}(E)$ one can proceed as in the proof of Theorem \ref{thm:GFP have cadlag version} but for the Yosida approximation in Equation~\eqref{eq:proof cadlag paths, Yosida approximation} one obtains an approximation with respect to the norm $\left\Vert \cdot\right\Vert _{\infty}$. This yields the following statement: 

\begin{theorem}
\label{thm:GFP have cadlag version-1}Let $\left(P(t)\right)_{t\in\mathbb{R}_{+}}$
be a generalized $Feller$ semigroup on $\mathcal{\mathscr{B}}^{\rho}(E)$,
let the conditions of Theorem \ref{thm:contractive gFs on weighted space is Markov process}
be satisfied and let $\left(\gamma_{t}\right)_{t\in\mathbb{R}_{+}}$
be the corresponding stochastic process on 
\[
\left(\left(E\cup\left\{ \triangle\right\} \right)^{\mathbb{R}_{+}},\mathcal{B}(E\cup\left\{ \Delta\right\} )^{\mathbb{R}_{+}}\right).
\]

(i) For every countable family  $\left(f_{n}\right)_{n\in\mathbb{N}}\subset\ell^{\rho}(E\cup\left\{ \triangle\right\} )$
there exists a family of stochastic processes with càdlàg 
paths 
\[
\left(\left(\overline{f_{n}(\gamma_{t})}\right)_{t\in\mathbb{R}_{+}}\right)_{n\in\mathbb{N}}
\]
 such that for all $t\in\mathbb{R}_{+}$ there is a null set $\mathcal{N}{}_{t}\in\mathcal{B}(E\cup\left\{ \Delta\right\} )^{\mathbb{R}_{+}}$
for which 
\[
f_{n}(\gamma_{t})=\overline{f_{n}(\gamma_{t})}\text{ on}\left(E\cup\left\{ \triangle\right\} \right)^{\mathbb{R}_{+}}\mathcal{\setminus N}_{t}
\]
 for all $n\in\mathbb{N}$. 

(ii) If additionally to the assumption in (i) there exists a countable
family $\left(f_{n}\right)_{n\in\mathbb{N}}\subset\ell^{\rho}(E\cup\left\{ \triangle\right\} )$
of sequentially continuous functions that separates points, then $\left(\gamma_{t}\right)_{t\in\mathbb{R}_{+}}$
has a version with càdlàg paths.
\end{theorem}

\section{Relation to standard Feller processes}\label{sec:relation}

 In this section we investigate the relationship between generalized and extended Feller processes as of
Theorem~\ref{thm:GFS induce Markov process} and Theorem \ref{thm:contractive gFs on weighted space is Markov process} respectively and
classical Feller processes. Again, $\left(E,\,\rho\right)$ always denotes a weighted
space equipped with the Borel $\sigma$-algebra $\mathcal{B}(E)$. \\

In the following proposition we provide a condition under which Feller processes are generalized Feller processes. 
\begin{proposition}
\label{prop:Generalized Feller sometimes Feller process on locally compact space}
Let $\left(E,\rho\right)$ be a weighted space and $E$ be locally
compact. Let $\left(\lambda_{t}\right)_{t\in\mathbb{R}_{+}}$ be a
Feller process on $E$ with semigroup of transition probabilities
$\left(p(t)\right)_{t\in\mathbb{R}_{+}}$ on $\left(E,\mathcal{B}(E)\right)$
with initial distribution $\nu\in\mathcal{M}^{\rho}(E)$. Let there
be $t_{0}>0$ and $C>0$ such that for all $x\in E$ and $0\leq t\leq t_{0}$
\[
\mathbb{E}_{x}\left[\rho(\lambda_{t})\right]\leq C\rho(x).
\]

Then $\left(\lambda_{t}\right)_{t\in\mathbb{R}_{+}}$ is a generalized
Feller process and for any $x\in E$ and initial distribution $\delta_x$ the process is a also a generalized Feller process with respect to the right continuous extension of its natural filtration. 
\end{proposition}

\begin{proof}
We first show that $\left(\tilde{P}(t)\right)_{t\in\mathbb{R}_{+}}$
given by 
\begin{align*}
\tilde{P}(t):\,\,\mathcal{\mathscr{B}}^{\rho}(E) & \rightarrow\mathcal{\mathscr{B}}^{\rho}(E)\\
f & \rightarrow\int_{E}f(y)p(t)(\cdot,dy)
\end{align*}
 is a generalized Feller semigroup knowing that $\left(P(t)\right)_{t\in\mathbb{R}_{+}}$
given by $
P(t): C_{0}(E)  \rightarrow C_{0}(E), \, 
f  \rightarrow\int_{E}f(y)p(t)(\cdot,dy)
$ is a Feller semigroup. 
To this end, we show that $\tilde{P}(t) $ is linear bounded map satisfying $\tilde{P}(t)\left(\mathcal{\mathscr{B}}^{\rho}(E)\right)=\mathcal{\mathscr{B}}^{\rho}(E)$. For any $f\in\mathcal{\mathscr{B}}^{\rho}(E)$
and $0\leq t\leq t_{0}$
\begin{align*}
\tilde{P}(t)f(x) & =\int_{E}f(y)p(t)(x,dy)\\
 & =\int_{E}\frac{f(y)}{\rho(y)}\rho(y)p(t)(x,dy)\\
 & \leq\left\Vert f\right\Vert _{\rho}C\rho(x),
\end{align*}
proving that  for $0\leq t\leq t_{0}$
$
\tilde{P}(t)$ 
is a linear bounded map with
$
\left\Vert \tilde{P}(t)\right\Vert _{L\left(\mathscr{B}^{\rho}(E)\right)}\leq C.
$
By Lemma \ref{lem:C_c dense in  B-rho for locally compact spaces} we know that
for any $\varepsilon>0$ and any $f\in\mathscr{B}^{\rho}(E)$ there
exists some $g_{\varepsilon}\in C_{0}(E)$ such that $\left\Vert f-g_{\varepsilon}\right\Vert _{\rho}<\varepsilon.$
Hence for $0\leq t\leq t_{0}$
\[
\left\Vert \tilde{P}(t)f-\tilde{P}(t)g_{\varepsilon}\right\Vert _{\rho}<C\varepsilon
\]
 and since $\tilde{P}(t)g_{\varepsilon}=P(t)g_{\varepsilon}\in C_{0}(E)$
it follows that that $\tilde{P}(t)(\mathcal{\mathscr{B}}^{\rho}(E))=\mathcal{\mathscr{B}}^{\rho}(E).$
For any $s>0$ there is $n\in\mathbb{N}$ such that $s/n<t_{0}$ and
since $\left(p(t)\right)_{t\in\mathbb{R}_{+}}$ is a semigroup of
transition probabilities on $\left(E,\mathcal{B}(E)\right)$
\begin{align*}
\tilde{P}(s)f(x) & =\int_{E}f(y)p(\frac{s}{n}+...+\frac{s}{n})(x,dy)\\
 & =\left(\tilde{P}(\frac{s}{n})...\left(\tilde{P}(\frac{s}{n})f\right)\right)(x).
\end{align*}
Hence, $\tilde{P}(t)$
is a linear bounded map for any $t>0$ and
\[
\left\Vert \tilde{P}(t)\right\Vert _{L\left(\mathscr{B}^{\rho}(E)\right)}\leq C^{\left\lceil t/t_{0}\right\rceil }.
\]
In order to show that $\tilde{P}(t)$ 
is indeed a generalized Feller semigroup we have to show the properties
\textbf{P1},...,\textbf{P5} from Definition \ref{def:generalized Feller semigroup}
hold. \textbf{P1} and \textbf{P2} follow immediately from the fact
$\left(p(t)\right)_{t\in\mathbb{R}_{+}}$ is a semigroup of transition
probabilities. \textbf{P4} follows by assumption and positivity (\textbf{P5})
is obvious. It remains to be shown that for all $f\in\mathcal{\mathscr{B}}^{\rho}(E)$
and all $x\in E$
\[
\underset{t\searrow0}{\lim}\,\tilde{P}(t)f(x)=f(x).
\]
Fix $f\in\mathcal{\mathscr{B}}^{\rho}(E)$ and $x\in E$. Again by Lemma
\ref{lem:C_c dense in  B-rho for locally compact spaces} we know that for any
$\varepsilon>0$ there is $g_{\varepsilon}\in C_{0}(E)$ such that
$\left\Vert f-g_{\varepsilon}\right\Vert _{\rho}<\varepsilon.$  As by the Feller property
\[
\underset{t\searrow0}{\lim} |\tilde{P}(t)g_{\varepsilon}(x)-g_{\varepsilon}(x)|=0
\]
for all $x \in E$, we thus have
\begin{align*}
\underset{t\searrow0}{\lim}\,\left|\tilde{P}(t)f(x)-f(x)\right| =  &\,\underset{t\searrow0}{\lim}\,\left|\tilde{P}(t)f(x)-\tilde{P}(t)g_{\varepsilon}(x)\right| +\underset{t\searrow0}{\lim}\,\left|\tilde{P}(t)g_{\varepsilon}(x)-g_{\varepsilon}(x)\right|\\
 &\quad +\left|g_{\varepsilon}(x)-f(x)\right|\\
 \leq & \, \underset{t\searrow0}{\lim}\,\left\Vert \tilde{P}(t)\right\Vert _{L\left(\mathcal{B}^{\rho}(E)\right)}\left\Vert f-g_{\varepsilon}\right\Vert _{\rho}\rho(x)
+\left|g_{\varepsilon}(x)-f(x)\right|\\
 \leq  & \,C\varepsilon\rho(x)+\varepsilon\rho(x).
\end{align*}
Since $\varepsilon>0$ was arbitrary, $\left(\tilde{P}(t)\right)_{t\in\mathbb{R}_{+}}$ is a generalized
Feller semigroup. 
Finally, since $\left(\lambda_{t}\right)_{t\in\mathbb{R}_{+}}$ is
a Markov process with respect to its natural filtration $\left(\mathcal{F}{}_{t}^{0}\right)_{t\in\mathbb{R}_{+}}$
for any initial distribution $\nu$ and $f\in\mathcal{\mathscr{B}}^{\rho}(E)$
and $0\leq s\leq t$ it holds
\[
\mathbb{E}_{\mathbb{P}_{\nu}}\left[\left.f(\lambda_{t})\right|\mathcal{F}{}_{s}^{0}\right]=\int_{E}f(y)p(t-s)(\lambda_{s},dy)=\tilde{P}(t-s)f(\lambda_{s})
\]
$\mathbb{P}_{\nu}$ -almost surely. As in the last step of the proof
in Theorem \ref{thm:GFS induce Markov process}, for any $x\in E$ and initial distribution $\delta_x$ this equation
can be extended to the right continuous extension of $\left(\mathcal{F}{}_{t}^{0}\right)_{t\in\mathbb{R}_{+}}$. This
yields the statement of the proposition. 
\end{proof}

The following proposition yields an isometric isomorphism that we  use later as a tool to link generalized Feller semigroups to Feller semigroups. For its formulation, recall Definition \ref{ellrho}.

\begin{proposition}
\label{prop:Isometry between Feller and gFs} Let $\rho$ be an admissible
weight function. Define

\begin{align*}
\Phi:\, & L(\mathcal{\mathscr{B}}^{\rho}(E))\rightarrow L(\ell^{\rho}(E))\\
&P \rightarrow\frac{P\left(\left(\cdot\right)\cdot\rho\right)}{\rho}.
\end{align*}
Then, 
$\text{\ensuremath{\Phi}}$
is an isometric isomorphism between $L(\mathcal{\mathscr{B}}^{\rho}(E))$
and $L(\ell^{\rho}(E))$.
\end{proposition}

\begin{proof}
Clearly,$\frac{P\left(\left(\cdot\right)\cdot\rho\right)}{\rho}\in L(\ell^{\rho}(E))$
is well defined and $\Phi$ is linear.
We show first that $\text{\ensuremath{\Phi}}$ is an isometry. 
We calculate for any $f \in \ell^{\rho}(E)$
\begin{align*}
\left\Vert \left(\Phi P\right)f\right\Vert _{\infty} & =\left\Vert P\left(f\cdot\rho\right)\right\Vert _{\text{\ensuremath{\rho}}}\\
 & \leq\left\Vert P\right\Vert _{L\left(\mathscr{B}^{\rho}(E)\right)}\cdot\left\Vert f\right\Vert _{\infty}.
\end{align*}
Hence 
\[
\left\Vert \left(\Phi P\right)\right\Vert _{L(\ell^{\rho}(E))}\leq\left\Vert P\right\Vert _{L\left(\mathscr{B}^{\rho}(E)\right)}.
\]
 Furthermore, for $\varepsilon>0$ let $g_{\varepsilon}\in\mathscr{B}^{\rho}(E)$
be such that 
\[
\left\Vert Pg_{\varepsilon}\right\Vert _{\text{\ensuremath{\rho}}}\geq\left(\left\Vert P\right\Vert _{L\left(\mathscr{B}^{\rho}(E)\right)}-\varepsilon\right)\left\Vert g_{\varepsilon}\right\Vert _{\text{\ensuremath{\rho}}}.
\]
Then  $\frac{g_{\varepsilon}}{\rho}\in\ell^{\rho}(E)$
and 
\begin{align*}
\left\Vert (\Phi P)\left(\frac{g_{\varepsilon}}{\rho}\right)\right\Vert _{\infty}
 & =\left\Vert Pg_{\varepsilon}\right\Vert _{\rho}\\
 & \geq\left(\left\Vert P\right\Vert _{L(\mathcal{\mathscr{B}}^{\rho}(E))}-\varepsilon\right)\left\Vert \frac{g_{\varepsilon}}{\rho}\right\Vert _{\text{\ensuremath{\infty}}},
\end{align*}
which shows that
\[
\left\Vert \left(\Phi P\right)\right\Vert _{L(\ell^{\rho}(E))}\geq\left\Vert P\right\Vert _{L\left(\mathscr{B}^{\rho}(E)\right)}-\varepsilon.
\]
Thus, $\Phi$ is an isometry. 
Furthermore, we note that for any $Q\in L(\ell^{\rho}(E))$ the map
\[
Q'(\cdot):=Q\left(\frac{(\cdot)}{\rho}\right)\cdot\rho
\]
yields surjectivity of $\Phi$ via  $\Phi\left(Q'\right)=Q$.

\end{proof}

\begin{corollary}
There is an isometric isomorphism between contractive generalized
Feller semigroups on $\mathcal{\mathscr{B}}^{\rho}(E)$ and strongly
continuous, contractive, positive semigroups on $\ell^{\rho}(E)$.
\end{corollary}

\begin{proof}
Use Proposition \ref{prop:Isometry between Feller and gFs} above
and strong continuity of generalized Feller semigroups (see Theorem
\ref{thm:generalized Feller semigroups are strongly continuous}).
The respective required semigroup properties follow immediately. 
\end{proof}

\begin{lemma}
\label{lem: continuous weight function l-rho space is C_0 } If the
admissible weight function $\rho$ is continuous, then $C_{0}(E)=\ell^{\rho}(E)$. 
\end{lemma}

\begin{proof}
This follows from Lemma \ref{lem: continuous weight function B-rho space is continuous } (iii)
and (iv). 
\end{proof}
\begin{corollary}
\label{cor:isometric isomorphism Feller semigroups/contractive gFs}
If the admissible weight function $\rho$ is continuous, then there
is an isometric isomorphism between contractive generalized Feller
semigroups on $\mathcal{\mathscr{B}}^{\rho}(E)$ and 
Feller semigroups on  $C_{0}(E)$.
\end{corollary}

\begin{proof}
Due to the continuity of $\rho$, $E$ is locally compact (see Lemma \ref {lem: continuous weight function B-rho space is continuous } (i)). Therefore Feller semigroups are well-defined. The rest follows from Lemma \ref{lem: continuous weight function l-rho space is C_0 } and Proposition~\ref{prop:Isometry between Feller and gFs}.
\end{proof}

The following theorem is the reason why $\left(\gamma_{t}\right)_{t\in\mathbb{R}_{+}}$
was named $\mathit{extended}$ $\mathit{Feller}$ $\mathit{process}$.

\begin{theorem}
\label{thm:extended Feller process is extended Feller process} Let
$\left(E,\rho\right)$ be a weighted space and let $\rho$ be continuous.
Let $\left(P(t)\right)_{t\in\mathbb{R}_{+}}$ be a generalized Feller
semigroup on $\mathcal{\mathscr{B}}^{\rho}(E)$ and let $\omega\in\mathbb{R}$
be such that for any $t\in\mathbb{R}_{+}$
\[
\left\Vert P(t)\right\Vert _{L(\mathcal{\mathscr{B}}^{\rho}(E))}\leq e^{\omega t}.
\]
 Then $\left(Q(t)\right)_{t\in\mathbb{R}_{+}}$ defined as 
\[
Q(t)f:=e^{-\omega t}\frac{P\left(t\right)\left(f\cdot\rho\right)}{\rho}
\]
is a 
Feller semigroup on $C_{0}(E)$\footnote{In contrast to the usual literature on Feller semigroups  where  separability is required,  we here call the strongly continuous, positive, contractive semigroup   $(Q(t))_{t \in \mathbb{R}_{+}}$ \emph{Feller}  even  though $E$ is not necessarily separable.} 
and for any probability measure $\nu$ on $\left(E,\mathcal{B}(E)\right)$
there exists a probability measure $\mathbb{P}'_{\nu}$ on
\[
\left(E^{\mathbb{R}_{+}},\mathcal{B}(E)^{\mathbb{R}_{+}}\right)
\]
and a right continuous filtration $\left(\mathcal{F}_{t}\right)_{t\in\mathbb{R}_{+}}$
such that and for any $t\geq s\geq0$ and for any $f\in C_{0}(E)$
for the canonical process $\left(\gamma_{t}\right)_{t\in\mathbb{R}_{+}}$
\begin{equation}
\mathbb{E}_{\mathbb{P}'_{\nu}}\left[\left.f(\gamma_{t})\right|\mathcal{F}_{s}\right]=Q(t-s)f(\gamma_{s})\label{eq:extended Feller process for C_0}
\end{equation}
holds true $\mathbb{P}'_{\nu}$ - almost surely and 
\[
\mathbb{P}'_{\nu}\circ\gamma_{\text{0}}^{-1}=\nu
\]
holds true. 
\end{theorem}

\begin{proof}
This follows from Lemma~\ref{lem: continuous weight function B-rho space is continuous }, 
Theorem~\ref{thm:contractive gFs on weighted space is Markov process}, Corollary~\ref{cor:M smaller 1 gFs on weighted space is Markov process}
and Proposition~\ref{prop:Isometry between Feller and gFs}. 
\end{proof}

\section{Examples}

This section is dedicated to specific examples ranging from deterministic Feller processes corresponding to semigroups of transport type to affine and polynomial processes.
In Section~\ref{eq:secP4} we also show by means of a counterexample the necessity of Condition \textbf{P4} (which is not needed for standard Feller semigroups on $C_0(E)$) to obtain strong continuity for the generalized Feller semigroup.

\subsection{Generalized Feller processes of transport type}\label{sec:transport}

We consider a generalized Feller semigroup such that the corresponding generalized Feller process, given by $\left(\psi_{t}\right)_{t\in\mathbb{R}_{+}}$,  is deterministic. This is an important class of examples, since it connects the theory of ordinary differential equations on weighted spaces with generalized Feller processes. In other contexts, e.g.~in machine learning, this is related to so-called Koopman operators, or, e.g.~in the theory of partial differential equations of transport type, to the method of characteristics.

Let us start by introducing the notion of a smooth algebra homomorphism.

\begin{definition}\label{def:OH}
We call a  continuous linear functional $\ell: \mathcal{\mathscr{B}}^{\rho}(E) \to \mathbb{R} $  a \emph{smooth algebra homomorphism,} if for all bounded smooth functions $\phi:\mathbb{R}^n \to \mathbb{R}$ and for all $f_1,\ldots,f_n \in \mathcal{\mathscr{B}}^{\rho}(E)$ we have that
$$
\ell(\phi(f_1(\cdot),\ldots,f_n(\cdot))) = \phi(\ell(f_1),\ldots,\ell(f_n)) \, .
$$
Analogously, we call a continuous linear map $P: \mathcal{\mathscr{B}}^{\rho}(E) \to \mathcal{\mathscr{B}}^{\rho}(E) $ a \emph{smooth operator algebra homomorphism} if for all bounded smooth functions $\phi:\mathbb{R}^n \to \mathbb{R}$ and for all $f_1,\ldots,f_n \in \mathcal{\mathscr{B}}^{\rho}(E)$ we have that
\begin{align}\label{eq:prop_soah}
P(\phi(f_1(\cdot),\ldots,f_n(\cdot)))(\cdot) = \phi(P(f_1)(\cdot),\ldots,\ell(f_n)(\cdot)) \, .
\end{align}
\end{definition}
The next lemma states that smooth algebra homomorphisms are characterized as point evaluations.
\begin{lemma}
Let $E$ be a weighted space and $\ell: \mathcal{\mathscr{B}}^{\rho}(E) \to \mathbb{R} $ be a continuous linear functional. Then $\ell$ is a smooth algebra homomorphism if and only if there is $x \in E$ such that $\ell(f) = f(x)$ for all $ f \in \mathcal{\mathscr{B}}^{\rho}(E)$. The point $x$ is uniquely determined by $\ell$.
\end{lemma}

\begin{proof}
Clearly, if $\ell$ is given by a point evaluation, it also satisfies the smooth algebra homomorphism property. 

For the converse direction, observe that
a smooth algebra homomorphism can never be a sum of two other non-trivial continuous linear functionals with disjoint supports. Indeed assume $ \ell = \ell_1 +  \ell_2 $ with $\ell_j \neq 0$, $j=1,2$ with disjoint supports, then with $\phi(f_1,f_2)=f_1f_2$ (for some bounded $f_1$, $f_2$), $\ell$ would need to satisfy
$$
\ell_1(f_1f_2) + \ell_2(f_1f_2) = (\ell_1(f_1)+\ell_2(f_1))(\ell_1(f_2)+\ell_2(f_2)).
$$
This is impossible if one chooses $f_i$ with vanishing $f_1f_2$ on the support of $\ell$ such that $\ell_i(f_i)=1$ and $\ell_1(f_2)=\ell_2(f_1)=0$, since then the left hand side gives $0$ while the right hand side is equal to $1$.
Note that  one can always construct such $f_i$, as
$E$ is completely regular and $\sigma$-compact, whence normal. Therefore the measure representing $\ell$ has to have as support only a singleton.
\end{proof}

We can now use this to characterize smooth operator algebra homomorphisms via concatenations with maps $\psi$ whose restrictions to compact sets $K_R$ are continuous.

\begin{lemma}\label{lem:soah}
Consider a map $P:\mathcal{\mathscr{B}}^{\rho}(E) \to \mathcal{\mathscr{B}}^{\rho}(E)$. Then $P$ is a smooth operator algebra homomorphims if and only if there is a map $\psi: E \to E$, 
whose restrictions $\psi|_{K_R}$ are continuous for all $R >0$ and which satisfies 
\begin{align}\label{eq:rhobound}
\sup_{x \in E} \frac{\rho \circ \psi(x)} {\rho(x)} < \infty,
\end{align}
such that $Pf=f \circ \psi$. The map $\psi$ is uniquely determined by $P$.
\end{lemma}

\begin{proof}
Assume first that $P$ is given by $Pf=f \circ \psi$. Then clearly $P$ is linear and  satisfies \eqref{eq:prop_soah}.  Continuity follows since
for a sequence $(f_n)$ converging to $f$ in $\mathcal{\mathscr{B}}^{\rho}(E) $ it holds that for every $\varepsilon >0$ there exists some $N$ such that
\begin{align*}
\| Pf-Pf_N\|_{\rho} &=\underset{x\in E}{\sup}\,\frac{\left|f\circ\psi_{t}(x)-f_{N}\circ\psi(x)\right|}{\rho(x)} \\& =\underset{x\in E}{\sup}\,\frac{\left|f\circ\psi(x)-f_{N}\circ\psi(x)\right|}{\rho\circ\psi(x)}\cdot\frac{\rho\circ\psi(x)}{\rho(x)} < \varepsilon,
\end{align*}
where the last inequaltiy follows from \eqref{eq:rhobound}.

Conversely, if $P$ is a smooth operator algebra homomorphism, then existence of $\psi$ is a consequence of the previous lemma. The proof of the continuity of $\psi|_{K_R}$ for all $R >0$ and \eqref{eq:rhobound} follows similarly as the necessity of Condition (iv)  and (v) in 
Proposition \ref{prop:characterization of gFs of transport type}.   
\end{proof}

\begin{definition}
A linear operator $A : \operatorname{dom}(A) \subset \mathcal{\mathscr{B}}^{\rho}(E) \to \mathcal{\mathscr{B}}^{\rho}(E)$ is called smooth derivation if for all smooth $ \phi: \mathbb{R}^n \to \mathbb{R}$ with $\phi(f_1(\cdot),\ldots,f_n(\cdot)) \in \operatorname{dom}(A) $ for $f_1,\ldots,f_n \in \operatorname{dom}(A)$ we have that
$$
A(\phi(f_1(\cdot),\ldots,f_n(\cdot))) = \sum_{i=1}^n (\partial_i \phi)(f_1(\cdot),\ldots,f_n(\cdot))  A f_i(\cdot) \, .
$$
\end{definition}

With this definition we can already characterize how a generalized Feller semigroup of smooth operator algebra homomorphism is generated.

\begin{proposition}\label{prop:alghomoderiv}
Let $\left(P(t)\right)_{t\in\mathbb{R}_{+}}$ be a generalized Feller semigroup. Then $P$ is a semigroup of smooth operator algebra  homomorphisms if and only if the generator $A$ is a smooth derivation such that for every smooth map $\phi : \mathbb{R}^n \to \mathbb{R}$ with globally bounded first derivatives and for every $f_1,\ldots,f_n \in \operatorname{dom}(A)$ we have $\phi(f_1(\cdot),\ldots,f_n(\cdot)) \in \operatorname{dom}(A)$.
For every $t \geq 0$, such a generalized Feller semigroup  is of the form
$$
P(t)(f):=f\circ\psi_{t}
$$
for some map $\psi_t:E \to E$ which satisfies the properties (i) - (vi) of Proposition \ref{prop:characterization of gFs of transport type}.
\end{proposition}
\begin{proof}
Let $P$ be a generalized Feller semigroup of smooth operator algebra homomorphisms. 
Consider some smooth $ \phi: \mathbb{R}^n \to \mathbb{R}$ with $\phi(f_1(\cdot),\ldots,f_n(\cdot)) \in \operatorname{dom}(A) $ with $f_1,\ldots,f_n \in \operatorname{dom}(A)$. 
Then by the smooth operator algebra homomorphism property and the chain rule, we have
\begin{align*}
A\phi(f_1(\cdot),\ldots,f_n(\cdot))&:=
\frac{d}{dt}|_{t=0}P(t) \phi(f_1(\cdot),\ldots,f_n(\cdot))=\frac{d}{dt}|_{t=0}\phi(P(t)f_1(\cdot),\ldots,P(t)f_n(\cdot))\\
&=\sum_{i=1}^n (\partial_i \phi)(f_1(\cdot),\ldots,f_n(\cdot))  A f_i,
\end{align*}
showing that the generator $A$ is a smooth derivation. If $\phi$ has globally bounded derivatives, the right hand side lies in $\mathcal{\mathscr{B}}^{\rho}(E)$ and therefore 
$\phi(f_1(\cdot),\ldots,f_n(\cdot)) \in \operatorname{dom}(A)$, which proves the first direction.

Conversely, for fixed $t > 0$, $0 \leq s \leq t $, a smooth map $\phi : \mathbb{R}^n \to \mathbb{R}$ with globally bounded first derivatives, and some $f_1,\ldots,f_n \in \operatorname{dom}(A)$ with $\phi(f_1(\cdot),\ldots,f_n(\cdot)) \in \operatorname{dom}(A)$ we obtain by the smooth derivation property of $A$ that
\begin{align*}
&\frac{d}{ds} P(s)\phi(P(t-s)f_1, \ldots, P(t-s)f_n) \\
 &\quad = P(s) A \phi(P(t-s)f_1, \ldots, P(t-s)f_n)  \\ & \quad \quad- P(s)(\sum_{i=1}^n (\partial_i \phi)(P(t-s)f_1,\ldots,P(t-s)f_n)  A P(t-s) f_i)\\
&\quad  =P(s)(\sum_{i=1}^n (\partial_i \phi)(P(t-s)f_1,\ldots,P(t-s)f_n)  A P(t-s) f_i) \\ & \quad \quad  - P(s)(\sum_{i=1}^n (\partial_i \phi)(P(t-s)f_1,\ldots,P(t-s)f_n)  A P(t-s) f_i)=0 \, .
\end{align*}
Therefore,  $s \mapsto P(s)\phi(P(t-s)f_1, \ldots, P(t-s)f_n) $ is constant and by setting $s=t$ and $s=0$ we get $P(t)\phi(f_1,\ldots, f_n) = \phi(P(t) f_1,\ldots, P(t) f_n)$. Since this holds on a dense subset of functions $\phi$ with globally bounded first derivatives, we can conclude.

Concerning the last assertion Lemma \ref{lem:soah} yields for every fixed $t >0$ a map $\psi_t$ such that $P(t)f(x)=f \circ \psi_t(x)$. Since $P$ is assumed to be a generalized Feller semigroup, the properties of $\psi_t$
follow from Proposition \ref{prop:characterization of gFs of transport type}.
\end{proof}

The above proposition characterizes the generators of generalized Feller semigroups of smooth operator algebra homomorphisms as smooth derivations, which are also called \emph{transport operators} motivating the name of the current subsection.
Indeed,  instead of generalized Feller semigroups of  smooth operator algebra homomorphisms we can thus also speak of \emph{generalized Feller semigroup of transport type.} We shall use the  notions smooth derivations and transport operators interchangeably.
The following proposition establishes all properties of the transport map $\psi_t$ for $ t \in \mathbb{R}_+$. It is remarkable that we only need continuity properties of $\psi$ with respect to time and space.

\begin{proposition}
\label{prop:characterization of gFs of transport type} Let $\left(\psi_{t}\right)_{t\in\mathbb{R}_{+}}$
be a family of maps from $E$ to $E.$
 Then $\left(P(t)\right)_{t\in\mathbb{R}_{+}}$ defined as
\[
P(t)(f):=f\circ\psi_{t}
\]
is a generalized Feller semigroup on $\mathcal{\mathscr{B}}^{\rho}(E)$,
 if and only if the following conditions hold:\\
(i) $\psi_{0}=\text{Id}$.\\
(ii) For any $t_{1,}t_{2}\in\mathbb{R}_{+}$
\[
\psi_{t_{1}}\circ\psi_{t_{2}}=\psi_{t_{1}+t_{2}}.
\]
\\
(iii) For any $x\in E$
\[
\underset{t\searrow0}{\lim}\,\psi_{t}(x)=x.
\]
\\
(iv) For any $t\in\mathbb{R}_{+}$ and any $R>0$
\[
\left.\psi_{t}\right|_{K_{R}}:\,K_{R}\rightarrow E
\]
 is continuous.\\
(v) For any $t\in\mathbb{R}_{+}$
\[
\underset{x\in E}{\sup}\,\frac{\rho\circ\psi_{t}(x)}{\rho(x)}=:C_{t}<\infty.
\]
\\
(vi) For some $\delta>0$ there is $C>0$ such that for all $0\leq t<\delta$
\[
C_{t}<C.
\]
\\
Furthermore, for such a generalized Feller semigroup 
the identity
\begin{equation}
P(t)\rho(x)=\underset{\begin{array}{c}
f\in C_{b}(E)\\
\left|f\right|\leq\rho
\end{array}}{\sup}\left|f\circ\psi_{t}(x)\right|=\rho\circ\psi_{t}(x)\label{eq:P(t)rhoforGFSoftransporttype}
\end{equation}
holds true.

\end{proposition}

\begin{proof}
We first show that the conditions (i)-(vi) are sufficient in order
to obtain a generalized Feller semigroup. 

Fix $t\in\mathbb{R}_{+}$. We show that $P(t)$ is a bounded linear
map from $\mathcal{\mathscr{B}}^{\rho}(E)$ to $\mathcal{\mathscr{B}}^{\rho}(E)$.
For $f\in\mathcal{\mathscr{B}}^{\rho}(E)$ and $n\in\mathbb{N}$,
by definition of $\mathcal{\mathscr{B}}^{\rho}(E)$, there is $f_{n}\in C_{b}(E)$
such that
\[
\left\Vert f-f_{n}\right\Vert _{\rho}<\frac{1}{n}.
\]
By Theorem \ref{thm: equivalence B-rho space} we obtain $f_{n}\circ\psi_{t}\in\mathcal{\mathscr{B}}^{\rho}(E)$
for any $n\in\mathbb{N}$ since on the one hand $\left.f_{n}\circ\psi_{t}\right|_{K_{R}}\in C_{b}(E)$
holds for any $R>0$ and on the other hand
\[
\underset{R\rightarrow\infty}{\lim}\underset{x\in E\setminus K_{R}}{\sup}\frac{\left|f_{n}\circ\psi_{t}(x)\right|}{\rho(x)}=0.
\]
The inequality
\begin{align*}
\underset{x\in E}{\sup}\,\frac{\left|f\circ\psi_{t}(x)-f_{n}\circ\psi_{t}(x)\right|}{\rho(x)} & =\underset{x\in E}{\sup}\,\frac{\left|f\circ\psi_{t}(x)-f_{n}\circ\psi_{t}(x)\right|}{\rho\circ\psi_{t}(x)}\cdot\frac{\rho\circ\psi_{t}(x)}{\rho(x)}\\
 & \leq\frac{1}{n}\cdot C_{t}
\end{align*}
yields that $f\circ\psi_{t}\in\mathcal{\mathscr{B}}^{\rho}(E)$ as
a limit of functions in $\mathcal{\mathscr{B}}^{\rho}(E).$ Moreover,
\begin{align*}
\left\Vert f\circ\psi_{t}\right\Vert _{\rho} & =\underset{x\in E}{\sup}\,\frac{\left|f\circ\psi_{t}(x)\right|}{\rho\circ\psi_{t}(x)}\cdot\frac{\rho\circ\psi_{t}(x)}{\rho(x)}\\
 & \leq\left\Vert f\right\Vert _{\rho}\cdot C_{t},
\end{align*}
hence $P(t)$ is a linear bounded operator on $\mathcal{\mathscr{B}}^{\rho}(E)$. Moreover, 
the Properties \textbf{P1, P2, }and\textbf{ P5 }of generalized Feller
semigroups are easy to check. For Property \textbf{P4} we see that
for all $0\leq t<\delta$
\[
\left\Vert P(t)\right\Vert \leq C_{t}\leq C.
\]
Regarding Property \textbf{P3,} we observe that for any $x\in E$
and any $0\leq t<\delta$ the inequality
\begin{align*}
\rho\circ\psi_{t}(x) & \leq C_{\delta}\cdot\rho(x)=:R_{x}
\end{align*}
holds true. Therefore, $\psi_{t}(x)\in K_{R_{x}}$ for $t\in\left[0,\delta\right)$
and because of $\left.f\right|_{K_{R_{x}}}\in C_{b}(E)$ for all $f\in\mathcal{\mathscr{B}}^{\rho}(E)$
(see Theorem \ref{thm: equivalence B-rho space}) we obtain
\[
\underset{t\searrow0}{\lim}\,f\circ\psi_{t}(x)=f(x)
\]
for any $x\in E.$ 

Next, we show that if $\left(P(t)\right)_{t\in\mathbb{R}_{+}}$ is
a generalized Feller semigroup, then Properties (i)-(vi) and Equation~\eqref{eq:P(t)rhoforGFSoftransporttype} hold true.
Property (i) follows from $P(0)=\text{Id}$ which yields
\begin{equation}
f\circ\psi_{0}=f\,\text{ for all }f\in\mathcal{\mathscr{B}}^{\rho}(E).\label{eq:proof transport semigroups are GFS, property (i)}
\end{equation}
So by contradiction, if there was some $x\in E$ such that $\psi_{0}(x)\neq x$,
then by definition of completely regular spaces, one could find some
map $f_{x}\in C_{b}(E)\subset\mathcal{\mathscr{B}}^{\rho}(E)$ such
that $f_{x}(x)=1$ and $f\circ\psi_{0}(x)=0$. But this would contradict
Equation \ref{eq:proof transport semigroups are GFS, property (i)}. 
Regarding Property (ii), as in the proof of Property (i) we obtain
\begin{equation}
f\circ\left(\psi_{t_{1}}\circ\psi_{t_{2}}\right)=f\circ\left(\psi_{t_{1}+t_{2}}\right)\,\text{for all }f\in\mathcal{\mathscr{B}}^{\rho}(E)\label{eq:proof transport semigroups are GFS, property (ii)}
\end{equation}
and as above by contradiction, if Property (ii) did not hold, then
one could find a map in $\mathcal{\mathscr{B}}^{\rho}(E)$ that would
contradict Equation~\eqref{eq:proof transport semigroups are GFS, property (ii)}. 
Property (iii) can be shown in the same way, since by definition of
generalized Feller semigroups 
\[
\underset{t\searrow0}{\lim}\,f\circ\psi_{t}(x)=f(x)
\]
holds for any $x\in E$ and any $f\in\mathcal{\mathscr{B}}^{\rho}(E)$. 
In order to show Property (iv), we fix some $R>0$ and some arbitrary
open set $O$ in $E$. We have to show that $\psi_{t}^{-1}(O)$ is
open in $K_{R}$ with respect to the subspace topology. We know by
Theorem \ref{thm: equivalence B-rho space} that $\left.f\circ\psi_{t}\right|_{K_{R}}$
is continuous for any $f\in\mathcal{\mathscr{B}}^{\rho}(E)$. For
any $x\in O$, by definition of completely regular spaces, we know
that we can find $f_{x}\in C_{b}\left(E\right)$ such that
\[
\left|f_{x}\right|\leq1,
\]
\[
f_{x}(x)=1,
\]
 and
\[
f_{x}(E\setminus O)\subset\left\{ 0\right\} .
\]
Clearly, 
\[
\underset{x\in O}{\bigcup}\left(\left.f_{x}\circ\psi_{t}\right|_{K_{R}}\right)^{-1}\left(0,2\right)
\]
is open in $K_{R}$ with respect to the subspace topology. On the
other hand 
\[
\underset{x\in O}{\bigcup}\left(f_{x}\right)^{-1}\left(0,2\right)=O.
\]
Thus, 
\begin{align*}
\left.\psi_{t}\right|_{K_{R}}^{-1}(O) & =\left.\psi_{t}\right|_{K_{R}}^{-1}\left(\underset{x\in O}{\bigcup}\left(f_{x}\right)^{-1}\left(0,2\right)\right)\\
 & =\underset{x\in O}{\bigcup}\left(\left.f_{x}\circ\psi_{t}\right|_{K_{R}}\right)^{-1}\left(0,2\right)
\end{align*}
is open in $K_{R}$ with respect to the subspace topology. 
Regarding Equation~\eqref{eq:P(t)rhoforGFSoftransporttype}, by Remark
\ref{rem: definition P_t rho} $P(t)\rho(x)$
is given for any $x\in E$ by 
\begin{align*}
P(t)\rho(x) & =\underset{\begin{array}{c}
f\in C_{b}(E)\\
\left|f\right|\leq\rho
\end{array}}{\sup}\left|f\circ\psi_{t}(x)\right|.
\end{align*}
We observe that for any $y\in E$ and any $n\in\mathbb{N}$ there
is an open neighborhood $O_{n,y}$ of $y$ such that
\[
\rho(x)>\rho(y)-\frac{1}{n}
\]
holds true for any $x\in O_{n,y}$. On $E\setminus O_{n,y}\cup\left\{ y\right\} $
we define the function
\[
g_{n,y}(x):=\begin{cases}
\rho(y)-\frac{1}{n} & \text{for }x=y\\
0 & \text{for }x\in E\setminus O_{n,y},
\end{cases}
\]
and by Proposition \ref{prop: Tietze extension completely regular case}
we can extend $g_{n,y}$ to $f_{n,y}\in C_{b}(E)$ such that $\left|f_{n,y}\right|<\rho$
and $\rho(y)-f_{n,y}(y)=\frac{1}{n}$. Hence, for any $x\in E$
\[
\underset{\begin{array}{c}
f\in C_{b}(E)\\
\left|f\right|\leq\rho
\end{array}}{\sup}\left|f\circ\psi_{s}(x)\right|=\rho\circ\psi_{s}(x).
\]
Finally, Property (v) and (vi) follow since for any $x\in E$
\[
\rho\circ\psi_{t}(x)=P(t)\rho(x),
\]
and by Theorem \ref{thm:generalized Feller semigroups are strongly continuous} the estimate $
\left\Vert P(t)\right\Vert \leq Me^{\omega t}
$
holds true for some $M\geq1$ and $\omega\in\mathbb{R}$. Thus, Remark
\ref{rem: definition P_t rho} implies
that for any $x\in E$
\[
\rho\circ\psi_{t}(x)\leq\rho(x)Me^{\omega t}.
\]
\end{proof}

Combining the above assertions yields the following corollary.

\begin{corollary}
A family $(P(t))_{t  \in \mathbb{R}_+}$  is a generalized Feller semigroup of smooth operator algebra homomorphism if and only if there exists a family of maps $\left(\psi_{t}\right)_{t\in\mathbb{R}_{+}}$
from $E$ to $E$ satisfying the conditions of Proposition \ref{prop:characterization of gFs of transport type}
such that
$
P(t)(f)=f\circ\psi_{t}$.
The family of  maps $(\psi_{t})_{t\in\mathbb{R}_{+}}$ is exactly the generalized Feller process constructed in Theorem~\ref{thm:GFS induce Markov process}.
\end{corollary}

\begin{proof}
The first direction was already stated in Proposition \ref{prop:alghomoderiv}. The other one is a consequence of Proposition \ref{prop:characterization of gFs of transport type} which yields a generalized Feller semigroup which clearly satisfies the property of being a smooth operator algebra homomorphism. The last assertion just follows from \eqref{eq:equation of GFS yields Markov process}.
\end{proof}

\begin{remark}
The combination of Proposition \ref{prop:alghomoderiv} and Proposition \ref{prop:characterization of gFs of transport type} tells that any semiflow satisfying the properties of Proposition \ref{prop:characterization of gFs of transport type} on a weighted space $E$ can be associated to a unique transport operator, i.e.~has a version of a tangent direction at any point in time.

Conversely, given a transport operator which generates a generalized Feller semigroup, the associated Feller process construced in Theorem \ref{thm:GFS induce Markov process} corresponds to $(\psi_t)_{t \in \mathbb{R}_+}$. For conditions when a transport operator generates a generalized Feller semigroup we refer to Theorem 3.3 in \cite{DoT}, which provides a reformulation of the Lumer-Philips theorem for the quasi-contractive case. Notice, however, that a transport semigroup does not need to be quasi-contractive.
\end{remark}

\subsection{Polynomial and affine processes}\label{sec:affineandpoly}

For the definition and theory of  polynomial and affine processes we refer to  \cite{CMT, filipovic2016polynomial} and \cite{DFS:03}.
For $n\in \mathbb{N}$,  we denote by $E$ be a closed subset of  $\mathbb{R}^n$.
To introduce polynomial processes, we let $\mathcal{P}_{m}$ be the space of polynomials on $E$ up to degree $m \in \mathbb{N}$. For ($m$-)polynomial processes with state space $E$ we always consider a version with càdlàg paths and denote by $\left(P(t)\right)_{t\in{\mathbb{R}_+}}$ its Markovian semigroup and by $\left(p(t)\right)_{t\in{\mathbb{R}_+}}$ its semigroup of transition probabilities. Note that we here always consider polynomial processes whose solution to the martingale problem is unique assuring the Markov property. We refer to \cite{KK:20} 
for examples where this is not the case, i.e.~where the law is not determined by the (extended) infinitesimal generator.

We start with the following lemma which is essentially a consequence of the definition of polynomial processes.

\begin{lemma}
\label{lem: m-polynomial process on polynomial remains bounded }
Let $\left(\lambda_{t}\right)_{t\in\mathbb{R}_{+}}$ be an $m$-polynomial
process and let $\rho\in\mathcal{P}_{k}$ for some $k\in\left\{ 0,...,m\right\} $.
Then there is a bounded linear map $A_{k}$ on $\mathcal{P}_{k}$ and $C>0$
such that for all $x\in E$ and $t\in\mathbb{R}_{+}$ 
\[
P(t)\rho(x)=\mathbb{E}_{x}\left[\rho(\lambda_{t})\right]=\left(e^{tA_{m}}\rho\right)(x)\leq C e^{t\left\Vert A_{m}\right\Vert } \rho(x)
\]
holds true and $\mathbb{E}_{\nu}\left[\rho(\lambda_{t})\right]<\infty$
for all $t\in\mathbb{R}_{+}$ and for any probability measure  $\nu\in\mathcal{M}^{\rho}(E)$. 
\end{lemma}

\begin{proof}
This follows directly from Theorem 2.7 (ii) in \cite{CMT}.
\end{proof}

Next we establish the generalized Feller property for polynomial processes under a continuity assumption on the semigroup.

\begin{proposition}
\label{prop:polynomial process is generalized Feller} For $m\in\mathbb{N}$, 
let $\left(\lambda_{t}\right)_{t\in\mathbb{R}_{+}}$ be an $m$-polynomial
process and let $\rho\in\mathcal{P}_{m}$ be an admissible weight
function on $E$. For any $f\in C_{b}(E)$ and any $t\in\mathbb{R}_{+}$
let $\left.P(t)f\right|_{K_{R}}$ be continuous for any $R>0.$ Then
$\left(\lambda_{t}\right)_{t\in\mathbb{R}_{+}}$ is a generalized
Feller process on $\left(E,\rho\right)$.
\end{proposition}

\begin{proof}
We have to show that $\left(\lambda_{t}\right)_{t\in\mathbb{R}_{+}}$
is the stochastic process constructed in Theorem \ref{thm:GFS induce Markov process}.
By definition of the Markov process $\left(\lambda_{t}\right)_{t\in\mathbb{R}_{+}}$
for any $t\geq s\geq0$, any probability measure  $\nu\in\mathcal{M}^{\rho}(E)$ and any measurable map $f:\,E\rightarrow\mathbb{\mathbb{R}}_{+}$
\begin{equation}
\mathbb{E}_{\mathbb{P}_{\nu}}\left[\left.f(\lambda_{t})\right|\mathcal{F}_{s}\right]=P\left(t-s\right)f(\lambda_{s})
\end{equation}
holds true $\mathbb{P}_{\nu}$-almost surely and 
\[
\mathbb{\mathbb{P}_{\nu}}\circ\lambda_{0}^{-1}=\nu.
\]
By Lemma \ref{lem: m-polynomial process on polynomial remains bounded }
\[
\mathbb{E}_{\mathbb{P}_{\nu}}\left[\left.f(\lambda_{t})\right|\mathcal{F}_{s}\right]=P\left(t-s\right)f(\lambda_{s})
\]
holds true $\mathbb{P}_{\nu}$ -almost surely for all $f\in\mathscr{B}^{\rho}(E)$
as well. In order to show that $\left(\lambda_{t}\right)_{t\in\mathbb{R}_{+}}$
is a generalized Feller process we still have to prove that $\left(P(t)\right)_{t\in\mathbb{R}_{+}}$
is a generalized Feller semigroup. We fix some $t\in\mathbb{R}_{+}$
and first show that $f\in\mathcal{\mathscr{B}}^{\rho}(E)$ implies
$P\left(t\right)f\in\mathcal{\mathscr{B}}^{\rho}(E)$. By Lemma \ref{lem: m-polynomial process on polynomial remains bounded }
for any $f \in \mathcal{\mathscr{B}}^{\rho}(E)$
the map
\begin{align*}
x & \rightarrow\int_{E}f(y)p(t)(x,dy)
\end{align*}
is well defined and for some $C>0$
\begin{align*}
P\left(t\right)f(x) & =\int_{E}f(y)p(t)(x,dy)\\
 & \leq C e^{t\left\Vert A_{m}\right\Vert }\left\Vert f \right\Vert_{\rho} \rho(x).
\end{align*}
In order to show $P\left(t\right)f\in\mathcal{\mathscr{B}}^{\rho}(E)$
for any $f\in\mathcal{\mathscr{B}}^{\rho}(E)$, by continuity of $P\left(t\right)$
with respect to $\left\Vert \cdot\right\Vert _{\rho}$ and density
of $C_{b}(E)$ in $\mathcal{\mathscr{B}}^{\rho}(E)$ it is sufficient
to show that $P\left(t\right)f\in\mathcal{\mathscr{B}}^{\rho}(E)$
holds true for any $f\in C_{b}(E)$. By Theorem \ref{thm: equivalence B-rho space}
this is the case since $P(t)f$ is clearly bounded and by assumption
$\left.P(t)f\right|_{K_{R}}$ is continuous for any $R>0$. 

Regarding the properties of generalized Feller semigroups in Definition
\ref{def:generalized Feller semigroup},  \textbf{P1}, \textbf{P2} and positivity \textbf{(P5}) are clearly satisfied
for $\left(P(t)\right)_{t\in\mathbb{R}_{+}}$. Property \textbf{P4} holds true
due to Lemma \ref{lem: m-polynomial process on polynomial remains bounded }.
Regarding \textbf{P3} since the paths of $\left(\lambda_{t}\right)_{t\in\mathbb{R}_{+}}$
are càdlàg for all $f\in C_b(E)$ and
all $x\in E$ we obtain by dominated convergence
\[
\underset{t\searrow0}{\lim}\,P(t)f(x)=\underset{t\searrow0}{\lim}\,\mathbb{E}_{x}\left[f(\lambda_{t})\right]=f(x).
\]
Hence \textbf{P3} follows from density of $f\in C_b(E)$ in $\mathcal{\mathscr{B}}^{\rho}(E)$ and from Lemma \ref{lem: m-polynomial process on polynomial remains bounded }.
Thus $\left(P(t)\right)_{t\in\mathbb{R}_{+}}$ is a generalized Feller
semigroup and $\left(\lambda_{t}\right)_{t\in\mathbb{R}_{+}}$ is
a generalized Feller process on $\left(E,\rho\right).$
\end{proof}

We now turn to affine processes. For their definitionon general state spaces $E\subset \mathbb{R}^n$ we refer to $\cite{CTgss}$. Following Theorem 5.8 in $\cite{CTgss}$ we write the part of the differential semimartingale characteristics  of the affine process that corresponds to the compensator of the jump measure as $K(x,d\text{\ensuremath{\xi}})$, its  linear part according to Theorem 6.4 in $\cite{CTgss}$ as
\[
\mu(x,d\text{\ensuremath{\xi}})=x_{1}\mu_{1}(d\text{\ensuremath{\xi}})+...+x_{n}\mu_{n}(d\text{\ensuremath{\xi}}),
\]
and its constant part as $m(x,d\text{\ensuremath{\xi}})$. As in Example 3.1 in \cite{CMT} one can show the following lemma.

\begin{lemma}
\label{lem:affine is r-polynomial}Consider a state space $E\subset\mathbb{R}^{n}$ that contains $n+1$ elements $x_{1},...,x_{n+1}$ such that for
every $j\in\left\{ 1,...,n+1\right\} $ the set
\[
\left(x_{1}-x_{j},...,x_{j-1}-x_{j},x_{j+1}-x_{j},...,x_{n+1}-x_{j}\right)
\]
is linearly independent. Then an affine process $\left(\lambda_{t}\right)_{t\in\mathbb{R}_{+}}$
is $r$-polynomial with $r\geq2$ if the killing rate is constant, and
\[
\int_{\left\Vert \xi\right\Vert >1}\left\Vert \xi\right\Vert ^{r}m\left(d\xi\right)<\infty,
\]
and for any $i\in\left\{ 1,...,n\right\} $
\[
\int_{\left\Vert \xi\right\Vert >1}\left\Vert \xi\right\Vert ^{r}\mu_{i}\left(d\xi\right)<\infty.
\]
\end{lemma}

\begin{proof}
From Theorem 6.4 in $\cite{CTgss}$
it follows that there is $C>0$ such that
\[
\int_{\mathbb{R}^{d}}\left\Vert \xi\right\Vert ^{r}K\left(\lambda_{t},d\xi\right)\leq C\left(1+\left\Vert \lambda_{t}1_{\left\{ \left.t\in\mathbb{R}_{+}\right|\lambda_{t}\neq \Delta\right\} }\right\Vert ^{r}\right).
\]
The lemma follows then from Theorem 2.15 in \cite{CMT}. 
\end{proof}

Under the conditions of the above lemma  and a continuity assumption on the semigroup, we now also obtain the generalized Feller property for affine processes.

\begin{corollary}
\label{cor:affine process is generalized Feller}
Consider a state space $E\subset\mathbb{R}^{n}$ that satisfies the conditions of Lemma \ref{lem:affine is r-polynomial}.  
Let $r\geq2$ and assume that the linear part of the killing rate vanishes and that
\[
\int_{\left\Vert \xi\right\Vert >1}\left\Vert \xi\right\Vert ^{r}m\left(d\xi\right)<\infty, \qquad \text{and} \qquad 
\int_{\left\Vert \xi\right\Vert >1}\left\Vert \xi\right\Vert ^{r}\mu_{i}\left(d\xi\right)<\infty, \quad i\in\left\{ 1,...,n\right\}.
\]
Let $\rho\in\mathcal{P}_{r}$ be an admissible weight function on
$E$. If for any $f\in C_{b}(E)$ and any $t\in\mathbb{R}_{+}$, 
$\left.P(t)f\right|_{K_{R}}$ is continuous for any $R>0$, then the affine process $\left(\lambda_{t}\right)_{t\in\mathbb{R}_{+}}$
is a generalized Feller process on $\left(E,\rho\right).$
\end{corollary}

\begin{proof}
Combine Proposition \ref{prop:polynomial process is generalized Feller}
and Lemma \ref{lem:affine is r-polynomial}.
\end{proof}

\subsection{Necessity of Condition \textbf{P4}} 
\label{eq:secP4}

For classical Feller semigroups it is sufficient to require the properties \textbf{P1}, \textbf{P2}, \textbf{P3} and \textbf{P5} of a generalized Feller semigroup in order to obtain a strongly continuous semigroup on $C_0(E)$. For generalized Feller semigroups we need to require additionally  \textbf{P4}. 
Indeed, subsequently we construct a semigroup on $\mathscr{B}^{\rho}(E)$ that fulfills \textbf{P1}, \textbf{P2}, \textbf{P3} and \textbf{P5}  but not \textbf{P4} and that is \emph{not} strongly continuous. To this end, we consider $E = \mathbb{N}$ as state space
and for $n \geq 0$ take $\rho (n) = \exp (n^2)$. The point $0$ is an absorbing state.
We construct a Markov chain with discrete state space and continuous time whose transition rate `matrix' $A=(a_{ij})_{i,j \in \mathbb{N}} $ is given by
\begin{align*}
&a_{n,n + 1} =
n^\alpha \exp (- n), \, a_{n,n} = - n^\alpha, a_{n,0} = n^\alpha(1 - \exp (- n)), &\quad \text{for  $n$ odd,} \\
&a_{n,n} = - (n-1)^{\alpha}, \, a_{n,0} = (n-1)^{\alpha}  &\quad \text{for  $n$ even,}
\end{align*}
and
$0$ otherwise, where $\alpha > 1$. The process jumps from odd numbers with high
intensity and low probability up to the next even number and then with high
intensity down to $0$ which it the absorbing state. From even numbers it jumps directly down to the absorbing state. 

In the following proposition we show the form of the corresponding Markov semigroup and that it is well defined on $\mathscr{B}^{\rho}(\mathbb{N})$.

\begin{proposition}
  The Markov semigroup $P(t) = \exp (At)$ is well defined on $\mathscr{B}^{\rho}(\mathbb{N})$.
  Furthermore for odd $n$ we have   for $t \geq 0$
\begin{align*}
     P(t) \rho (n) =& \rho (n) \exp (- n^\alpha t) \\&+ \rho (n+1) n^{\alpha} t \exp (- n^{\alpha} t) \exp(-n)
     \\ &+ \rho (0) \left(1 - \exp (- n^\alpha t) -\exp (- n) n^\alpha t \exp (- n^{\alpha} t)\right). 
\end{align*}

\end{proposition}
\begin{proof}
Solving the Kolomogorov forward equations for the transition probabilities, i.e.
\[
\partial_t p_{i,j}(t)= \sum_{k} p_{i,k}(t) a_{k,j}
\]
via the matrix exponential $\exp(At)$
yields for even $n$
\begin{align*}
p_{n,n} (t) = \exp (- (n-1)^{\alpha} t), \quad
p_{n,0} (t) = 1 - \exp (-(n-1)^{\alpha} t).
\end{align*}
 and for odd $n$ 
 \begin{align*}
 p_{n,n} (t) &= \exp (- n^\alpha t), \\ p_{n,n+1} (t)&
 =\exp(-n)  n^\alpha t \exp(-n^{\alpha}t)
 \\ p_{n,0} (t)&= 
1 - \exp (- n^\alpha t) -\exp(-n)  n^\alpha t \exp(-n^{\alpha}t).
\end{align*}
All other quantities are zero except of $p_{0,0}$ which is equal to 1, 
as $0$ is an absorbing state.
This implies existence of the Markov process with transition rate matrix $A$. Furthermore, for every $t$ it follows that $P(t) \rho \leq C_t \rho$ , whence the
semigroup is well defined on $B^{\rho}(\mathbb{N})$. Due to the topology on $\mathbb{N}$, the semigroup is also well defined on  $\mathscr{B}^{\rho}(\mathbb{N})$.
\end{proof}

For this semigroup  the validity of \textbf{P1}, \textbf{P2},  \textbf{P3}, and \textbf{P5} are clear. We show that  \textbf{P4} does not hold. We define for any $t\in \mathbb{R}_+$ 
\[
 s(t):=\sup_{n \in \mathbb{N}} \frac{ \rho (n+1) n^{\alpha} t \exp (- n^{\alpha} t) \exp(-n)}{\rho (n)}. 
\] Maximizing this for $n \in \mathbb{R}_+$ using standard methods yields for the maximizing $n_t$ the condition
\begin{align*}
n_t^{\alpha}&=\frac{n_t}{\alpha t}+\frac{1}{t}. \end{align*}
Thus, $t\rightarrow 0$ implies $n_t \rightarrow \infty$ and we can estimate for large $n_t$
\begin{align*}
s(t) &\geq e^1\left(\lfloor n_t \rfloor^{\alpha}t\right)\cdot\exp\left(\lfloor n_t \rfloor-\lfloor n_t \rfloor^{\alpha}t\right)\\
&\geq e^1\left(\frac{n_t}{\alpha }+1-(n_t^{\alpha}-\lfloor n_t \rfloor^{\alpha})t\right)\cdot\exp\left(-2-\frac{n_t}{\alpha} +n_t\right)\\
&\geq e^1\left(\frac{n_t}{\alpha }+1-\frac{\alpha(\alpha +1)}{n_t}n_t^{\alpha}t\right)\cdot\exp\left(-2-\frac{n_t}{\alpha} +n_t\right),
\end{align*}
where we used a Taylor expansion in the last step. Using again the expression for $n_t^{\alpha}$ we obtain that $t\rightarrow 0$ implies $s(t)\rightarrow \infty$. Hence, \textbf{P4} does not hold. In particular, this calculation shows that for any $f\in \mathscr{B}^{\rho}(\mathbb{N})$ that behaves like $\rho$ around a large enough $n\in \mathbb{N}$ and is $0$ elsewhere $P(t)f$ does not converge in norm to $f$ for $t \rightarrow 0$. In other words, the semigroup is not strongly continuous. However, when identifying the state $0$ with the cemetery $\Delta$, the semigroup has the Feller property. Note that Proposition \ref{prop:Generalized Feller sometimes Feller process on locally compact space} does not apply here, precisely because \textbf{P4} does not hold true.

\appendix
\section{Appendix}

We collect in this appendix on the one hand some functional analytic assertions used throughout the paper and other hand some lemmas for $\mathscr{B}^{\rho}(E)$ functions.

\subsection{Functional analytic tools}

\begin{definition}
\label{def:C_b-open} Let $\Omega\neq\emptyset$. A set $O\subset\Omega$
is called $C_{b}(\Omega)$-$\mathit{open}$, if there exists a sequence
$\left(f_{n}\right)_{n\in\mathbb{N}}\subset C_{b}(\Omega)$ such that
pointwise $f_{n}\nearrow1_{O}$. The system of sets that are $C_{b}(\Omega)$-open is called $\mathcal{G}\left(C_{b}(\Omega)\right)$ .
\end{definition}

We recall here the following lemma, which can be found in  M.~Schweizer's lecture notes ``Measure and Integration'' (version
July 22, 2017), see Lemma IV.1.11.
\begin{lemma}
\label{lem:C_b-open sets generate Baire sigma-algebra} Let $\Omega\neq\emptyset$
and let $\mathcal{G}\left(C_{b}(\Omega)\right)$ be the system of
sets that are $C_{b}(\Omega)$-open. Then
\[
\sigma\left(\mathcal{G}\left(C_{b}(\Omega)\right)\right)=\sigma\left(\left.f\right|\,f\in C_{b}(\Omega)\right),
\]
is the smallest $\sigma$-algebra such that all maps in $C_{b}(\Omega)$
are measurable. 
\end{lemma}

\begin{definition}
\label{def:Baire sigma algebra} The smallest $\sigma$-algebra such
that all maps in $C(\Omega)$ (or all maps in $C_b(\Omega)$) are measurable
is called $\mathit{Baire}$ $\sigma$-$\mathit{algebra}$ and is denoted by
$\mathcal{B}_{0}(\Omega)$.
\end{definition}

Making slight adjustments in the proof of (\cite{BoI}, §5, Proposition
5) one obtains a version of the proposition that holds
on $C_{b}\left(X \right)$:
\begin{proposition}
\label{prop:condition when linear funtional on Cb is measure real case}
Let $X$ be a completely regular space and $\ell:\,C_{b}\left(X\right)\rightarrow\mathbb{R}$
be a continuous linear map. There exists a signed Radon measure $\mu$
on $X$ such that for all $f\in C_{b}\left(X \right)$
\[
\ell(f)=\int_{X}f(x)\mu(dx),
\]
if and only if for each $\varepsilon>0$ there exists a compact set
$K_{\varepsilon}\subset X$ such that for any function $f\in C_{b}\left(X \right)$
with $\left|f\right|\leq1$ and $\left.f\right|_{K_{\varepsilon}}=0$, we have 
$
\left|\ell(f)\right|<\varepsilon
$. The signed Radon measure is unique. 
\end{proposition}

\begin{definition}
\label{def:compact class}

A family $\mathcal{C}$ of subsets of a space $X$ is called $compact$
$class$ if for any sequence $\left(C_{n}\right)_{n\in\mathbb{N}}\subset\mathcal{C}$
such that the intersection $\underset{n\in\mathbb{N}}{\bigcap}C_{n}$
is empty, already some finite intersection $\underset{i\in I,\text{ finite}}{\bigcap}C_{i}$
is empty.
\end{definition}

For a completely regular space statements similar to Tietze-Urysohn
extension theorem and Urysohn's Lemma can be shown.
\begin{proposition}
\label{prop: Tietze extension completely regular case} Let $E$ be
completely regular and $K\subset E$ compact. Then a real-valued continuous
function $f\in C\left(K \right)$ on $K$ can be extended
to a continuous function $F\in C\left(E \right)$ on all
of $E$. If additionally $\left|f\right|<C<\infty$ then there is
an extension $F\in C\left(E \right)$ such that $\left|F\right|<C<\infty$.
\end{proposition}

\begin{proof}
We would like to apply the Tietze-Urysohn extension theorem. However, it allows the
extension only on normal spaces. But since a space is completely regular if and only if it is a  subspace of a compact space, 
we can embed $E$ by an embedding $i$ in a compact Hausdorff
set $N$ which is normal. Then $i(K)$ is also compact on $N$ with respect to the subspace
topology $\tau(i(E))$ on
$N$, hence compact with respect to the topology of $N$. Since $N$ is Hausdorff,
the compact set $i(K)$ is closed.
We can apply the Tietze-Urysohn extension Theorem
on $N$ to extend the function $f\circ i^{-1}\in C(i(K))$
to a continuous function $G\in C(N)$ such 
\[
\left.f\circ i^{-1}\right|_{i(K)}=\left.G\right|_{i(K)}
\]
and $\left|G\right|\leq C$ if $\left|f\right|\leq C$. Therefore,
$F:=G\circ i$ possesses the desired properties. 
\end{proof}   
   
\begin{proposition}
\label{prop: Urysohn's Lemma completely regular case} (Urysohn's
Lemma in the completely regular case) Let $E$ be completely regular,
$K\subset E$ compact, $A\subset E$ closed and $A\cap K=\emptyset$.
Then there is a continuous function $f:\,E\rightarrow\left[0,1\right]$
such that $f(K)=\left\{ 0\right\} $, $f(A)=\left\{ 1\right\} $.
\end{proposition}

\begin{proof}
As in Proposition \ref{prop: Tietze extension completely regular case},
we embed $E$ in a compact Hausdorff set $N$ by an embedding denoted by $i$. Then, 
$i(K)$ is compact, hence closed in the compact Hausdorff space $N$.
Since $i(A)$ is closed in the subspace topology $\tau(i(E))$, there
is a closed set $B\subset N$ such that $B\cap i(E)=i(A)$ and clearly
$B\cap i(K)=\emptyset.$ Applying Urysohn's Lemma in the normal space
$N$ we see that there is a continuous function $g:\,N\rightarrow\left[0,1\right]$
with $g(i(K))=\left\{ 0\right\} $ and $g(B)=\left\{ 1\right\} $.
Setting $f:=g\circ i$, we conclude. 
\end{proof}

\begin{corollary}
\label{cor:completely regular space, convergence of continuous bounded functions to open set}Let
$E$ be a completely regular space, $\mathcal{B}(E)$ its Borel $\sigma$-algebra, $m\in\mathbb{N}$, $\mu_1,\ldots,\mu_m$ a family of measures on $\left(E,\mathcal{B}(E)\right)$ and $B\in\mathcal{B}(E)$.
If there is a sequence of compact sets $\left(K_{n}\right)_{n\in\mathbb{N}}$
and of open sets $\left(O_{n}\right)_{n\in\mathbb{N}}$ such that $K_{n}\subset B\subset O_{n}$
for any $n\in\mathbb{N}$ and if for $i\in{1,\ldots,m}$
\[
\underset{n\rightarrow\infty}{\lim}\mu_i(O_{n}\setminus K_{n})=0,
\]
then there exists a sequence $\left(f_{n}\right)_{n\in\mathbb{N}}$
of non-negative continuous functions with $f_{n}\leq1_{O_{n}}$ for
any $n\in\mathbb{N}$ such that 
\[
\underset{n\rightarrow\infty}{\lim}f_{n}=1_{B}
\]
 in $L^{1}\left(E,\mu_i\right)$ for any $i\in{1,\ldots,m}$. If $\mu_i$, $i\in{1,\ldots,m}$ is $\sigma$-finite, then
this convergence holds true also $\mu_i$-almost surely.
\end{corollary}

\begin{proof}
Thanks to Urysohn's Lemma in the complete regular case there is a
sequence $\left(g_{n}\right)_{n\in\mathbb{N}}$ of non negative continuous
functions with $1_{K_{n}}\leq g_{n}\leq1_{O_{n}}$ for any $n\in\mathbb{N}$
such that $g_{n}\rightarrow1_{B}$ in $L^{1}\left(E,\mu_i\right)$ for $i\in{1,\ldots,m}$. Thus, if $\mu_i$ is $\sigma$-finite then there exists a subsequence $\left(g_{n_{k}}\right)_{k\in\mathbb{N}}$
such that $g_{n_{k}}\rightarrow1_{B}$ $\mu_i$-almost surely.
\end{proof}

\subsection{Some lemmas for $\mathscr{B}^{\rho}(E)$ functions }

The following lemma is needed for the existence proof of generalized Feller processes in Theorem \ref{thm:GFS induce Markov process} and is shown by adapting the reasoning of Example 1.13 in \cite{KBi}.

\begin{lemma}
\label{lem:B-rho tensor product space, linear isomorphism}Let $n\in\mathbb{N}$ and let $\left(E_{i},\rho_{i}\right)$
$i\in\left\{ 1,...,n\right\} $ be weighted spaces and
\[
\rho\left(x_{1},...,x_{n}\right):=\rho_{1}\left(x_{1}\right)\cdot\cdot\cdot\rho_{n}\left(x_{n}\right).
\]
 Then the linear map 
\begin{align*}
\Psi:\,\,\,\,\,\mathcal{\mathscr{B}}^{\rho_{1}}(E_{1})\otimes...\otimes\mathcal{\mathscr{B}}^{\rho_{n}}(E_{n}) & \rightarrow\mathcal{\mathscr{B}}^{\rho}(E_{1}\times...\times E_{n})\\
f_{1}\otimes...\otimes f_{n} & \rightarrow f_{1}\cdot\cdot\cdot f_{n}
\end{align*}
is injective and its image is a dense linear subspace of $\mathcal{\mathscr{B}}^{\rho}(E_{1}\times...\times E_{n}).$
\end{lemma}
\begin{proof}
By Lemma \ref{lem:product space of weighted space is weighted space}
$
\left(E_{1}\times...\times E_{n},\rho\right)
$
 is indeed a weighted space. Furthermore, for $f_{i}\in\mathcal{\mathscr{B}}^{\rho_{i}}(E_{i})$,
$i\in\left\{ 1,...,n\right\} $ the map 
\[
\left(x_{1},...,x_{n}\right)\rightarrow f_{1}(x_{1})\cdot\cdot\cdot f_{n}(x_{n})
\]
is in $\mathcal{\mathscr{B}}^{\rho}(E_{1}\times...\times E_{n})$ which follows from continuity of multiplication.

By definition of the tensor product the linear map  $\Psi$  exists. It is injective since for
$
0\neq u\in\mathcal{\mathscr{B}}^{\rho_{1}}(E_{1})\otimes...\otimes\mathcal{\mathscr{B}}^{\rho_{n}}(E_{n})
$
 there is $m\in\mathbb{N}$ such that we can choose
a representation 
\[
u={\sum_{j=1}^{m}}f_{1}^{j}\otimes...\otimes f_{n}^{j},
\]
with $\left\{ f_{i}^{j}\right\} _{j\in\left\{ 1,...,m\right\} }\subset\mathcal{\mathscr{B}}^{\rho_{i}}(E_{i})$
for any $i\in\left\{ 1,...,n\right\} $ and 
$
\left\{ f_{1}^{j}\right\} _{j\in\left\{ 1,...,m\right\} },...,\left\{ f_{n}^{j}\right\} _{j\in\left\{ 1,...,m\right\} }
$
 linearly independent which implies that for
any $i\in\left\{ 1,...,n-1\right\} $ there is $z_{i}\in E_{i}$ such that $f_{i}^{1}(z_{i})\neq0$, hence by linear independence of $\left\{ f_{n}^{j}\right\} _{j\in\left\{ 1,...,m\right\} }$
\[
{\sum_{j=1}^{m}}f_{1}^{j}(z_{1})\cdot\cdot\cdot f_{n-1}^{j}(z_{n-1})f_{n}^{j}\neq0.
\]

Density of the image of $\Psi$ follows directly from Stone-Weierstrass
for $\mathcal{\mathscr{B}}^{\rho}$-spaces (Proposition \ref{prop:Stone-Weierstrass on B-rho spaces})
as the image is an algebra that separates points and contains
$1_{E_{1}\times...\times E_{n}}$.
\end{proof}

The following lemma also needed for Theorem \ref{thm:GFS induce Markov process} states that the composition of a $\mathcal{\mathscr{B}}^{\rho}(E)$-function  with a continuous bounded function is a continuous map.

\begin{lemma}
\label{lem:composition B-rho is continuous} Let $h\in C_{b}(\mathbb{R})$
and $f\in\mathcal{\mathscr{B}}^{\rho}(E)$. Then 
\begin{align*}
\mathcal{\mathscr{B}}^{\rho}(E)  \rightarrow\mathcal{\mathscr{B}}^{\rho}(E), \quad
f  \rightarrow h\circ f
\end{align*}
is a continuous map. 
\end{lemma}

\begin{proof}
Since $\left.h\circ f\right|_{K_{R}}$ is continuous for any $R>0$,
by Theorem \ref{thm: equivalence B-rho space} $h\circ f\in\mathcal{\mathscr{B}}^{\rho}(E)$ and the map is well defined. Let $g\in\mathcal{\mathscr{B}}^{\rho}(E)$,
$\varepsilon>0$ and choose $R_{\varepsilon}>\frac{2\left\Vert h\right\Vert _{\infty}}{\varepsilon}$.
Then 
\begin{align*}
\left\Vert h\circ f-h\circ g\right\Vert _{\rho} \leq\varepsilon+\frac{1}{\underset{x\in E}{\inf}\rho(x)}\cdot\left\Vert \left.\left(h\circ f-h\circ g\right)\right|_{K_{R_{\varepsilon}}}\right\Vert _{\infty}.
\end{align*}
Let $\left[a,b\right]$ be some interval such that $f(K_{R_{\varepsilon}})\subset\left[a,b\right]$. There is $\delta>0$ such that  $\left|x_{1}-x_{2}\right|<\delta$
implies $\left|h(x_{1})-h(x_{2})\right|<\varepsilon$ for
any $x_{1},x_{2}\in\left[a-1,b+1\right]$. Choosing $g$ such that $\left\Vert f-g\right\Vert _{\rho}<\frac{\delta}{R_{\varepsilon}}$
yields
\[
\left\Vert \left.\left(f-g\right)\right|_{K_{R_{\varepsilon}}}\right\Vert _{\infty}\leq\left\Vert f-g\right\Vert _{\rho}\cdot R_{\varepsilon}<\delta,
\]
 and consequently
\[
\left\Vert \left.\left(h\circ f-h\circ g\right)\right|_{K_{R_{\varepsilon}}}\right\Vert _{\infty}<\varepsilon.
\]

\end{proof}

For a locally compact space $E$ the space $\mathcal{\mathscr{B}}^{\rho}(E)$ is already given by the closure of $C_c(E)$, which is subject of Lemma \ref{lem:C_c dense in  B-rho for locally compact spaces} below. This is needed to establish a relationship between generalized Feller and standard Feller processes in Section \ref{sec:relation}.

\begin{lemma}
\label{lem:C_c dense in  B-rho for locally compact spaces}Let  $E$ be locally compact. Then $C_{c}(E)$
is dense in $\mathcal{\mathscr{B}}^{\rho}(E)$ with respect to $\left\Vert \cdot\right\Vert _{\rho}.$
\end{lemma}

\begin{proof}
By definition of $\mathcal{\mathscr{B}}^{\rho}(E)$ we only need to
show that $C_{c}(E)$ is dense in $C_{b}(E)$ with respect to $\left\Vert \cdot\right\Vert _{\rho}.$
Let $f\in C_{b}(E)$ and $\varepsilon>0$. Choose $R_{\varepsilon}:=\frac{\left\Vert f\right\Vert _{\infty}}{\varepsilon}$.
By local compactness each element in $K_{R_{\varepsilon}}$ has a
compact neighborhood hence by compactness of $K_{R_{\varepsilon}}$
finitely many such compact neighborhood cover $K_{R_{\varepsilon}}$.
The union of these finitely many neighborhoods is a compact neighborhood
$U_{K_{R_{\varepsilon}}}\supset K_{R_{\varepsilon}}.$ Let $V_{K_{R_{\varepsilon}}}\subset U_{K_{R_{\varepsilon}}}$
be an open neighborhood of $K_{R_{\varepsilon}}$ and define the map
$\tilde{g}_{\varepsilon}\in C_{b}(K_{R_{\varepsilon}}\cup (U_{K_{R_{\varepsilon}}}\setminus V_{K_{R_{\varepsilon}}}))$
as 
\begin{align*}
\tilde{g}_{\varepsilon}: & =\begin{cases}
f & \text{on }K_{R_{\varepsilon}}\\
0 & \text{on }U_{K_{R_{\varepsilon}}}\setminus V_{K_{R_{\varepsilon}}}.
\end{cases}
\end{align*}
By normality of compact sets and the Tietze-Urysohn theorem (see e.g.~\cite{kelley2017general}) this map can be extended to $g'_{\varepsilon}\in C_{b}(U_{K_{R_{\varepsilon}}})$
such that $\left\Vert g'_{\varepsilon}\right\Vert _{\infty}=\left\Vert f\right\Vert _{\infty}$.
Subsequently the map $g'_{\varepsilon}$ can be extended to $g_{\varepsilon}\in C_{c}(E)$
with $\left\Vert g_{\varepsilon}\right\Vert _{\infty}=\left\Vert f\right\Vert _{\infty}$
by stetting $g_{\varepsilon}\equiv0$ on $E\setminus U_{K_{R_{\varepsilon}}}$.
Then 
\begin{align*}
\left\Vert g_{\varepsilon}-f\right\Vert _{\rho} & \leq\underset{x\in K_{R_{\varepsilon}}}{\sup}\frac{\left|g_{\varepsilon}(x)-f(x)\right|}{\rho(x)}+\underset{x\in E\setminus K_{R_{\varepsilon}}}{\sup}\frac{\left|g_{\varepsilon}(x)-f(x)\right|}{\rho(x)}\\
 & \leq0+\frac{2\left\Vert f\right\Vert _{\infty}}{R_{\varepsilon}}\\
 & =2\varepsilon,
\end{align*}
which proves the lemma since $\varepsilon>0$ was arbitrary. 
\end{proof}


\end{document}